\documentclass[final,onefignum,onetabnum]{siamart220329}
\usepackage{times}
\usepackage{graphicx}
\usepackage{amssymb}
\usepackage{amsmath}
\usepackage{mathtools}
\usepackage{stanli}
\usepackage{float}
\usepackage{subcaption}
\usepackage{pgfplots}
\usepgfplotslibrary{fillbetween}
\usepackage{color}
\usepackage{bm}
\usepackage{hyperref}
\usepackage[nameinlink,capitalize]{cleveref}
\usepackage{ulem}

\headers{Global weight opt. of frames under free vibrations}{M. Tyburec, M. Ko\v{c}vara, M. Handa, and J. Zeman}

\title{Global weight optimization of frame structures under free-vibration eigenvalue constraints\thanks{Published in SIAM Journal on Optimization, Vol. 36, Iss. 3 (2026), \href{https://dx.doi.org/10.1137/24M1713934}{doi: 10.1137/24M1713934}, 
\funding{This work was funded by the Czech Science Foundation~(project No. 22-15524S) {and, from January 2025, co-funded by the European Union under the ROBOPROX project~(reg. no. CZ.02.01.01/00/22 008/0004590)}.}}}
\author{Marek Tyburec\thanks{Faculty of Civil Engineering, Czech Technical University in Prague, Prague, CZ and Institute of Information Theory and Automation, Czech Academy of Sciences, Prague, CZ (\email{marek.tyburec@cvut.cz}, \url{http://mech.fsv.cvut.cz/\string~tyburec/}).} \and Michal Ko\v{c}vara\thanks{Institute of Information Theory and Automation, Czech Academy of Sciences, Prague, CZ and School of Mathematics, University of Birmingham, Birmingham, UK (\email{m.kocvara@bham.ac.uk}).} \and Marouan Handa\thanks{Institute of Information Theory and Automation, Czech Academy of Sciences, Prague, CZ 
  (\email{handa@utia.cas.cz}).} \and Jan Zeman\thanks{Faculty of Civil Engineering, Czech Technical University in Prague, Prague, CZ {and Nečas Center for Mathematical Modeling, Prague, CZ} (\email{jan.zeman@cvut.cz}).} }

\DeclareMathOperator*{\argmin}{arg\,min}

\makeatletter
\@ifundefined{proof}{%
  \providecommand{\qedsymbol}{\ensuremath{\square}}
  \theorembodyfont{\normalfont}
  \theoremsymbol{\qedsymbol}
  \newtheorem*{proof}{Proof}%
}{}
\makeatother

\newcommand*\squared[1]{\tikz[baseline=(char.base)]{
		\node[shape=rectangle,draw,inner sep=0pt, minimum size=4mm] (char) {#1};}}

\newtheorem{assumption}{Assumption}

\crefname{theorem}{Theorem}{Theorems}
\crefname{lemma}{Lemma}{Lemmas}
\crefname{proposition}{Proposition}{Propositions}
\crefname{corollary}{Corollary}{Corollaries}
\crefname{definition}{Definition}{Definitions}
\crefname{assumption}{Assumption}{Assumptions}
\crefname{remark}{Remark}{Remarks}

\begin{document}

%\linenumbers
	
\maketitle

%\doublespace
 
% REQUIRED
\begin{abstract}
Topology optimization of frame structures under free-vibration eigenvalue constraints constitutes a challenging nonconvex polynomial optimization problem with disconnected feasible sets. In this article, we first formulate it as a polynomial semidefinite programming problem (SDP) of minimizing a linear function over a basic semi-algebraic feasible set. We then propose to solve this problem by Lasserre hierarchy of linear semidefinite relaxations providing a sequence of increasing lower bounds. To obtain also a sequence of upper bounds and thus conditions on global $\varepsilon$-optimality, we provide a bilevel reformulation that exhibits a special structure: The lower level is quasiconvex univariate and {it has a non-empty interior if} the constraints of the upper-level problem {are satisfied}. After deriving the conditions for the solvability of the lower-level problem, we thus provide a way to construct feasible points to the original SDP. Using such a feasible point, we modify the original nonlinear SDP to satisfy the conditions for the deployment of the Lasserre hierarchy. Solving arbitrary degree relaxation of the hierarchy, we prove that scaled first-order moments associated with the problem variables satisfy feasibility conditions for the lower-level problem and thus provide guaranteed upper and lower bounds on the objective function. Using these bounds, we develop a simple sufficient condition for global $\varepsilon$-optimality and prove that the optimality gap $\varepsilon$ converges to zero if the set of global minimizers is convex. Finally, we illustrate these results with {four} representative problems for which the hierarchy converges in at most {five} relaxation degrees.
\end{abstract}

% REQUIRED

	
\section{Introduction}

Topology optimization of frame structures is a classic problem in structural design. It generalizes the optimal design of truss structures, for which various convex reformulations have been developed \cite{Achtziger2008,Kocvara2017} by involving a polynomial dependence of the equilibrium equation on the design variables. Consequently, optimization of frame structures is a challenging nonconvex problem for which mostly local optimization algorithms have been used so far \cite{Yu2004,Yamada2015}.

A common goal of structural designers is to minimize structural weight while meeting performance requirements, which in this article include upper bounds on work performed by external forces (compliance), and lower bounds on the lowest nonzero free-vibration eigenfrequency.  These constraints are industrially significant and are among the most common constraints in topology optimization \cite{Bendsoe2004,Tyburec2019}. This optimization problem can be formulated as a semidefinite program \cite{Achtziger2007}
\vspace{-2mm}\begin{subequations}\label{eq:opt1}
\begin{align}
	\min_{\mathbf{a}\in \mathbb{R}^{n_{\mathrm{e}}}}\; & \sum_{e=1}^{n_\mathrm{e} } \rho_e \ell_e a_e\label{eq:opt_objective1}\\
	\text{s.t.}\; & \mathbf{K} (\mathbf{a}) - \underline{\lambda} \mathbf{M}(\mathbf{a}) \succeq 0,\label{eq:opt_fv1}\\
	& \begin{pmatrix}
		\overline{c} & -\mathbf{f}^\mathrm{T}\\
		-\mathbf{f} & \mathbf{K}(\mathbf{a})
	\end{pmatrix} \succeq 0,\label{eq:opt_static1}\\
	& \mathbf{a} \ge \mathbf{0},\label{eq:opt_nonneg1}
\end{align}
\end{subequations}
in which $\mathbf{a} \in \mathbb{R}^{n_\mathrm{e}}_{\ge 0}$ are the to-be-optimized cross-section areas of the $n_\mathrm{e} \in \mathbb{N}$ elements, $\bm{\ell} \in \mathbb{R}^{n_\mathrm{e}}_{>0}$ and $\bm{\rho} \in \mathbb{R}^{n_\mathrm{e}}_{>0}$ are the element lengths and material densities, respectively; hence the objective function \eqref{eq:opt_objective1} represents the weight of the structure. Furthermore, $\mathbf{K}(\mathbf{a}) \in \mathbb{S}^{n_\mathrm{dof}}_{\succeq0}$ with $\mathbf{M}(\mathbf{a}) \in \mathbb{S}^{n_\mathrm{dof}}_{\succeq 0}$ denote symmetric positive semidefinite stiffness and mass matrices with the latter matrix being a linear function of $\mathbf{a}$, and $n_\mathrm{dof} \in \mathbb{N}$ is the number of degrees of freedom. Thus, \eqref{eq:opt_static1} expresses the static equilibrium together with the upper bound on compliance and \eqref{eq:opt_fv1} with $\underline{\lambda} \in \mathbb{R}_{>0}$ {denoting} the lower bound on the lowest eigenfrequency.

This problem setting has already been thoroughly investigated for truss structures. In this case, the structural response incorporates only membrane stiffness effects, so that the stiffness matrix $\mathbf{K}(\mathbf{a})$ is a linear function of the cross-section areas. Therefore, \eqref{eq:opt1} is a {smooth and} convex linear semidefinite program \cite{Achtziger2007,Achtziger2008}, which overcomes the difficulty of nondifferentiability associated with multiple eigenvalues and rank drops of the stiffness matrix when some of the elements vanish through a zero cross-sectional area.

In optimal design of frames, we have to also consider the bending effects of frame elements, such as the Euler-Bernoulli elements used in this work. Then $\mathbf{K}(\mathbf{a})$ becomes a second- or third-degree polynomial function of the cross-section areas {(see Appendix~\ref{app:polydep})}. Due to this, the frame optimization problem becomes a {\it nonlinear non-convex semidefinite optimization problem that needs to be solved to global optimality.}

In this article, we propose to solve the problem by techniques of global polynomial optimization, namely the moment-sum-of-squares hierarchy due to Lasserre~\cite{Lasserre2001} and Parrilo~\cite{Parrilo2003} and its extension to polynomial matrix inequalities \cite{Henrion2006}. This hierarchy of Lasserre relaxations provides a monotonic sequence of lower bounds that eventually converge to the optimum in a finite number of steps. The lower bounds are solutions to linear semidefinite optimization problems (SDP relaxations) of quickly growing dimension. The dimension makes higher-order relaxation (typically of order higher than two or three) highly impractical to be solved numerically. Moreover, it is well known that, from the solution of the relaxation alone, it is impossible to know how close it is to the global solution of the original problem. We therefore propose a novel technique to obtain, in line with the lower bounds, also upper bounds to the optimal objective value. This allows us to guarantee global $\varepsilon$-optimality of the obtained solutions.

{This idea has been already investigated for compliance optimization problems under weight constraints in \cite{Tyburec2021}, where the feasible upper bounds were constructed by evaluating structural compliance with the first-order moments serving as the expected cross-sectional areas. The same procedure, however, cannot be used for weight optimization problems.}

The principal idea is as follows. We solve a (low-order) Lasserre relaxation of the problem \eqref{eq:opt1} to get a solution $\mathbf{a}$ that is, by nature, infeasible in \eqref{eq:opt1}, unless the relaxation is already optimal. We show that there is a multiplier $\delta\in\mathbb{R}$ such that $\delta\mathbf{a}$ is feasible in \eqref{eq:opt1}, i.e., $\delta\mathbf{a} \ge \mathbf{0}$ and satisfies \eqref{eq:opt_fv1}--\eqref{eq:opt_nonneg1}. We then solve a univariate problem of minimizing $\delta$ to remain feasible in \eqref{eq:opt1}. This feasible point will then give us the requested upper bound. For this purpose, we formally rewrite \eqref{eq:opt1} as a bilevel optimization problem with a lower-level variable $\delta$. We stress that we never solve this bilevel problem as such, we only use this construction to obtain the upper bounds.

This approach has been investigated in~\cite{Tyburec2023} for frame optimization problem without vibration constraint. {However, without further research it is not clear whether and how the method can be extended to the problem with free-vibration eigenvalue constraints \eqref{eq:opt_fv1}.} 

The new constraint makes the problem much more challenging. It is not only non-convex, but its feasible set is generally disconnected due to the singularity phenomenon of resonating thin elements \cite{Ni2014}, which generally results in the feasible space of the cross-section areas being the union of zero and a closed positive interval (see Section \ref{sec:illustrative} for a visual illustration). Because of this, common local techniques for topology optimization that rely on invertibility of a (dynamic) stiffness matrix fail due to the cut-off of the disconnected domains. %As in the case of trusses, the free-vibration eigenvalue constraint is nondifferentiable for repeated eigenvalues and also for zero cross-sectional areas, which, however, is circumvented by the reformulation \eqref{eq:opt1} \cite{Achtziger2008}. 
The formulation \eqref{eq:opt1} was already exploited in the work of Yamada and Kanno \cite{Yamada2015}, who proposed a sequence of semidefinite programming relaxations that connected the feasible set. With an increasing value of the relaxation parameter, the method converges to locally-optimal solutions of high quality. However, we are not aware of any published results that solved {instances of \eqref{eq:opt1}} globally.

%The first challenge is related to the nonconvexity of the compliance constraint \eqref{eq:opt_static1}. Weight optimization problems with compliance constraints alone have been investigated in~\cite{Tyburec2023}, where we have shown that the moment-sum-of-squares hierarchy can be used to solve the problem by providing a monotonic sequence of lower bounds that eventually converge to the optimum in a finite number of steps. Furthermore, relaxation solutions were projected onto the feasible set of the original problem under mild assumptions, guaranteeing the global $\varepsilon$-optimality of such feasible upper bounds. However, this hierarchy is not the only way to solve the problem globally. In particular, it is also possible to exploit recent advances in branch-and-bound methods and solve the emerging problem by spatial branching on linear approximations of a nonconvex quadratically-constrained quadratic problem \cite{Toragay2022}.

%\subsection{Aims and contributions}

%This work deals with a global solution to the weight minimization problems of frame structures with fundamental free-vibration eigenvalue and compliance constraints. On the one hand, this contribution can be seen as a substantial extension of the results related to the convex truss setting \cite{Achtziger2008} to the polynomial and nonconvex formulation for frame elements, or, on the other hand, as a nontrivial addendum to our previous work \cite{Tyburec2023} by accounting for the constraints of free vibrations.

The paper is organized as follows.
Starting by recalling the necessary background on the moment-sum-of-squares hierarchy in Section \ref{sec:background}, we first formalize the optimization problem as a nonlinear semidefinite program in Section \ref{sec:formulation}. We further establish its bilevel reformulation, as explained above: on the lower level is a quasiconvex problem optimizing the scaling $\delta$ of fixed cross-section areas $\mathbf{a}$ under the free-vibration and compliance constraints, while the upper-level problem takes care of the feasibility of the scaled areas $\delta\mathbf{a}$.
This structure secures that all feasible points for the lower-level problem are also feasible points for the original nonlinear semidefinite programming problem. Next, in Section \ref{sec:scalar}, we focus on the lower-level problem and establish the conditions on the fixed cross-section areas $\mathbf{a}$ under which the problem is solvable. It turns out that these conditions are semidefinite representable, and thus a feasible point for the lower-level, and consequently to the original nonlinear semidefinite problem, is provided at the cost of solving a single linear semidefinite program.  

Section \ref{sec:msos} deals with the solution to the main (single-level) nonlinear semidefinite programming problem using the moment-sum-of-squares hierarchy. From the existence of a feasible solution we obtain bounds on the optimization variables and show that the algebraic compactness condition needed for the convergence of the hierarchy is satisfied. We then solve the hierarchy of semidefinite relaxations to obtain a monotonic sequence of lower bounds to the original problem. Furthermore, we prove that the optimal solutions of the relaxations also satisfy the conditions on the cross-section areas $\mathbf{a}$ for the existence of a solution to the lower-level problem in the bilevel formulation. Consequently, for each relaxation we can construct a feasible upper bound to the original problem by globally solving the univarate quasi-convex lower-level problem that searches for the minimal scaling $\delta$. By comparing the lower and upper bounds, we obtain a certificate of global $\varepsilon$-optimality for the feasible upper bounds. Finally, we prove that $\varepsilon\rightarrow0$ for optimization problems with a convex set of global minimizers. With this, we obtain a sufficient condition of global optimality that is computationally inexpensive, assesses the quality of relaxations, and complements the Curto-Fialkow flat extension theorem \cite{Curto1996} for the optimization problem being considered.

We illustrate these theoretical developments on a set of {four} numerical examples, which highlight the challenges present in the weight optimization problems of frame structures under vibration constraints together with the benefits of our method: finite convergence, feasible designs at any relaxation degree, and a guarantee of global $\varepsilon$-optimality.

\newpage\section{Background: Moment-sum-of-squares hierarchy}\label{sec:background}

In this section, we summarize the basic results related to the optimization of polynomial functions over a~basic semialgebraic feasible set using the Lasserre moment-sum-of-squares hierarchy. Such optimization problems are formalized as
\begin{subequations}\label{eq:opt_general}
\begin{align}
	\min_{\mathbf{x} \in \mathbb{R}^n}\; & f(\mathbf{x})\\
	\mathrm{s.t.}\; & \mathbf{G}(\mathbf{x}) \succeq 0,\label{eq:pmi}
\end{align}
\end{subequations}
with $\mathbf{G}(\mathbf{x}) \in \mathbb{S}^m_{\succeq 0}$, where $\mathbb{S}^{m}$ denotes the space of $m\times m$ real symmetric matrices and $\mathbb{S}^{m}_{\succeq 0} := \left\{\mathbf{G}\in\mathbb{S}^{m} : \mathbf{G}\succeq 0\right\}$ the set of positive semidefinite matrices (with $\mathbf{G}\succeq 0$ meaning that $\mathbf{G}$ is positive semidefinite). Furthermore, $f(\mathbf{x}): \mathbb{R}^n \mapsto \mathbb{R}$ and $\mathbf{G}(\mathbf{x}): \mathbb{R}^{n}\mapsto \mathbb{S}^{m}$ are real polynomial mappings, so that \eqref{eq:pmi} is a polynomial matrix inequality. In what follows, we denote the feasible set of \eqref{eq:pmi} by $\mathcal{K}(\mathbf{G}(\mathbf{x}))$.

{In general, the optimization problem \eqref{eq:opt_general} is nonconvex and difficult to solve, covering $\mathcal{NP}$-hard problems such as binary programming and nonconvex quadratic programming. Its main challenge lies in certifying that $f(\mathbf{x}) - \lambda \geq 0$ for all $\mathbf{x} \in \mathcal{K}(\mathbf{G}(\mathbf{x}))$ for the largest possible $\lambda$. This requires showing that $(f(\mathbf{x})-\lambda)$ belongs to the cone of polynomials that are nonnegative on $\mathcal{K}(\mathbf{G}(\mathbf{x}))$.}

Consider a polynomial $f(\mathbf{x})$ of a degree $d$, which can be written as $f(\mathbf{x}) = \mathbf{p}^\mathrm{T}\mathbf{b}_d(\mathbf{x})$, where
\begin{equation}
\begin{aligned}
\mathbf{b}_d(\mathbf{x})=&\Big( 1 \quad x_1 \quad x_2 \quad \cdots \quad x_1^2 \quad x_1x_2 \quad  \cdots  \quad x_1x_n \\ &  x_2^2 \quad x_2x_3 \quad  \cdots \quad x_n^2 \quad \cdots \quad x_1^d \quad \cdots \quad x_n^d \Big)^\mathrm{T}
\end{aligned}
\label{strandard_basis}
\end{equation}
denotes the canonical monomial basis of degree up to $d$ and $\mathbf{p} \in \mathbb{R}^{|\mathbf{b}_d(\mathbf{x})|}$ is the corresponding coefficient vector. To certify polynomial nonnegativity, we rely on the Sum-Of-Squares (SOS) representation
\begin{definition}\label{def:sos}
    A matrix polynomial $\mathbf{\Sigma}(\mathbf{x}): \mathbb{R}^n\mapsto \mathbb{S}^m$ is SOS if there exist matrices $\mathbf{H}_i \in \mathbb{R}^{m\times |\mathbf{b}_d(\mathbf{x})|}$, $i = \{1,\ldots,\ell\}$, such that 
    $
        \mathbf{\Sigma}(\mathbf{x}) = \sum_{i=1}^{\ell} \mathbf{H}_i \mathbf{b}_d(\mathbf{x}) \mathbf{b}_d(\mathbf{x})^\mathrm{T} \mathbf{H}_i^\mathrm{T}.
    $
\end{definition}
We note that this reduces to the scalar case by setting $m=1$.

If the polynomial $(f{(\mathbf{x})}-\lambda)$ can be written as SOS on $\mathcal{K}(\mathbf{G}(\mathbf{x}))$, it is automatically nonnegative, {ensuring $f(\mathbf{x}) \geq \lambda$ for all $\mathbf{x}$. However, not every nonnegative polynomial can be written as SOS and we need to take into account the feasible set.}

{
The question of when SOS can characterize polynomial nonnegativity on basic semi-algebraic sets has been answered by Putinar \cite{putinar1993positive}. If $\mathcal{K}(\mathbf{G}(\mathbf{x}))$ satisfies the following assumption, then SOS representations can characterize polynomial nonnegativity:
\begin{assumption}[Archimedean assumption]\label{as:archimedean}
    There exist SOS polynomials $p_0(\mathbf{x})$ and $\bm{\Sigma}(\mathbf{x})$ such that the superlevel set 
    $\bigl\{\mathbf{x}\in\mathbb{R}^{n}\;\vert\; p_0(\mathbf{x}) + \langle \bm{\Sigma}(\mathbf{x}), \mathbf{G}(\mathbf{x})\rangle\ge 0\bigr\}$ 
    is compact, where $\langle \cdot, \cdot\rangle$ denotes the Frobenius inner product.
\end{assumption}
Under this assumption, polynomial nonnegativity on $\mathcal{K}(\mathbf{G}(\mathbf{x}))$ can be certified by expressing $(f(\mathbf{x})-\lambda)$ as $p_0(\mathbf{x}) + \langle \bm{\Sigma}(\mathbf{x}), \mathbf{G}(\mathbf{x})\rangle$ with SOS polynomials $p_0(\mathbf{x})$ and $\bm{\Sigma}(\mathbf{x})$ of \textit{a priori} unknown degrees. However, finding such SOS representations does not provide us with the actual minimizers $\mathbf{x}^*$. To avoid this, we work with the dual formulation using moments, which transforms the infinite-dimensional SOS problem into a hierarchy of finite-dimensional semidefinite programs.}

{To construct these truncations, we introduce moment variables $\mathbf{y} \in \mathbb{R}^{|\mathbf{b}_{2r}(\mathbf{x})|}$ indexed by the monomial basis $\mathbf{b}_{2r}(\mathbf{x})$. Further, let us define} the Riesz functional $L_\mathbf{y} (\mathbf{Z}(\mathbf{x}))$ that linearizes {any} polynomial matrix $\mathbf{Z}(\mathbf{x})$ using $\mathbf{y}$. Specifically, the operator is applied entry-wise and for the $(i,j)$-th component of $\mathbf{Z}$ we write
\begin{equation}
    Z_{i,j}(\mathbf{x}) = \mathbf{p}_{{{i,j}}}^\mathrm{T} \mathbf{b}_{2r}(\mathbf{x})\quad \Longleftrightarrow \quad L_\mathbf{y}(Z_{i,j}) = \mathbf{p}_{{{i,j}}}^\mathrm{T} \mathbf{y}\; \text{ with }\mathbf{y}\text{ indexed in }\mathbf{b}_{2r}(\mathbf{x})
\end{equation}
with a coefficient vector $\mathbf{p}_{{{i,j}}}$. In this manuscript, we further adopt the notation $y_{x_i^d}$ to express the component of $\mathbf{y}$ belonging to $x_i^d \in \mathbf{b}_{2r}(\mathbf{x})$.

{
Under Assumption \ref{as:archimedean} (Archimedean assumption), polynomial optimization can be solved via the moment-sum-of-squares hierarchy:
\begin{subequations}\label{eq:general_moment}
\begin{align}
    \min_{\mathbf{y}} \; & L_\mathbf{y}(f(\mathbf{x}))\\
    \mathrm{s.t.}\; & L_{\mathbf{y}}\bigl(\mathbf{b}_{r-d_\mathrm{G}}(\mathbf{x}) \mathbf{b}_{r-d_\mathrm{G}}(\mathbf{x})^\mathrm{T} \otimes \mathbf{G}(\mathbf{x})\bigr) \succeq 0, & \text{(localizing matrix)}\label{eq:localizing}\\
    & L_{\mathbf{y}}\bigl(\mathbf{b}_r(\mathbf{x}) \mathbf{b}_r(\mathbf{x})^\mathrm{T}\bigr) \succeq 0, & \text{(moment matrix)}\label{eq:moment}
\end{align}
\end{subequations}
where $d_\mathrm{G}= \lceil \mathrm{deg}(\mathbf{G}(\mathbf{x}))/2 \rceil$ (with $\lceil \bullet \rceil$ denoting the ceiling function), and $\otimes$ denotes the Kronecker product. The constraints define moment matrices $\mathbf{M}_r(\mathbf{y})$ and localizing matrices $\mathbf{M}_{r-d_\mathrm{G}} (\mathbf{G}\mathbf{y})$ that correspond to the SOS conditions in the dual formulation.}

{Although \eqref{eq:general_moment} is convex in the moments $\mathbf{y}$, it is computationally intractable for $r\rightarrow \infty$. Instead, we solve a hierarchy of finite-dimensional truncations, starting with relaxation order $r = \max\left\{d_\mathrm{G}; \left\lceil\mathrm{deg}(f(\mathbf{x}))/2\right\rceil\right\}$ and increasing until convergence occurs. Since constraints of relaxation $r$ are included in relaxation $r+1$, we obtain a monotonic sequence of lower bounds: $f^{(r)} \le f^{(r+1)} \le f^* = f^{(\infty)}$. Remarkably, the hierarchy also exhibits generic finite convergence \cite{Nie2013}, though the convergence degree is not known \textit{a priori}.}

\begin{theorem}[{\cite[Theorem 2.2]{Henrion2006}}]\label{th:convergence}
    Let Assumption \ref{as:archimedean} be satisfied. Then, $f^{(r)}\nearrow f^*$ as $r\rightarrow \infty$.
\end{theorem}

{Due to finite convergence, global optimality $f^{(r)}=f^*$ can be recognized when the moment matrix becomes ``flat.'' This flatness property, formalized by the Curto-Fialkow theorem \cite{Curto1996}, occurs when the rank of the moment matrix stabilizes:
\begin{equation}\label{eq:flatextension}
    s = \mathrm{Rank}(\mathbf{M}_r(\mathbf{y})) = \mathrm{Rank}(\mathbf{M}_{r-d_{\mathrm{G}}}(\mathbf{y})),
\end{equation}
where $\mathbf{M}_{r-d_{\mathrm{G}}}(\mathbf{y})$ is a submatrix of $\mathbf{M}_r(\mathbf{y})$. This indicates that the moments correspond to a finite number of atoms, implying at least $s$ distinct global minimizers that can be extracted \cite{Henrion2005}. However, this condition fails when there are infinitely many global minimizers, and comparing ranks based on approximate solutions can be numerically challenging.}

\section{Methods}

Following the background in the moment sum-of-squares hierarchy, this section presents the optimization problem of designing minimum-weight frame structures under compliance and free-vibration eigenvalue constraints (in Section \ref{sec:formulation}), its bilevel programming reformulation (Section \ref{sec:scalar}), and states the conditions for its solvability. In Section \ref{sec:msos}, we modify the single-level formulation from Section \ref{sec:formulation} to satisfy the algebraic compactness condition for convergence of the moment-sum-of-squares hierarchy. After showing that the first-order moments obtained by solving relaxations of the modified formulation satisfy conditions for solvability of the lower level problem in the bilevel formulation, we construct feasible upper bounds to the original problem. The section concludes with a convergence proof of the optimality gap between the upper and lower bounds and with implementation remarks.

\subsection{Optimization problem formulation}\label{sec:formulation}

{Let $\mathcal{L}_\mathrm{s}$ and $\mathcal{L}_\mathrm{fv}$ denote the sets of indices corresponding to static and free-vibration load cases, respectively.} We investigate the global solution to topology optimization of least-weight frame structures \eqref{eq:opt_orig_objective}, while accounting for upper bounds $\overline{c}_j \in \mathbb{R}_{>0}$, $j \in {\mathcal{L}_\mathrm{s}}$, on structural compliances \eqref{eq:opt_orig_compliance} of {$\lvert \mathcal{L}_\mathrm{s} \rvert$ static} load cases and lower bounds $\underline{\lambda}_j \in \mathbb{R}_{>0}${, $j \in \mathcal{L}_\mathrm{fv}$,} on the lowest nonzero free-vibration eigenvalues {of the $\lvert \mathcal{L}_\mathrm{fv} \rvert$ free-vibration load cases} expressed using the Rayleigh quotient \eqref{eq:opt_orig_fv}. Such optimization problems are naturally formalized in terms of cross-section areas $\mathbf{a}$, displacements $\mathcal{U} = {\{\mathbf{u}_j : j \in \mathcal{L}_\mathrm{s}\}}$ with $\mathbf{u}_j \in \mathbb{R}^{n_{\mathrm{dof},j}}$, and normalized eigenmodes $\mathcal{V} = {\{\mathbf{v}_j : j \in \mathcal{L}_\mathrm{fv}\}}$ with $\mathbf{v}_j \in \mathbb{R}^{n_{\mathrm{dof},j}}$ as
\begin{subequations}\label{eq:opt_orig}
	\begin{align}
		\min_{\mathbf{a}\in \mathbb{R}^{n_{\mathrm{e}}}\!,\;
        \mathcal{U},\, 
        \mathcal{V}} 
        & \sum_{e=1}^{n_\mathrm{e} } \rho_e \ell_e a_e\label{eq:opt_orig_objective}\\
		\text{s.t.}\; & \inf_{\mathbf{v}_j \in \mathbb{R}^{n_{\mathrm{dof},j}} \setminus \mathrm{Ker}\left( \mathbf{M}(\mathbf{a}) \right)} \frac{\mathbf{v}_j^\mathrm{T} \mathbf{K}_j(\mathbf{a}) \mathbf{v}_j}{\mathbf{v}_j^\mathrm{T} \mathbf{M}_j(\mathbf{a}) \mathbf{v}_j} \ge \underline{\lambda}_j,\quad j \in {\mathcal{L}_\mathrm{fv}}\label{eq:opt_orig_fv}\\
		& \mathbf{K}_j(\mathbf{a}) \mathbf{u}_j = \mathbf{f}_j,\quad j \in {\mathcal{L}_\mathrm{s}}\label{eq:opt_orig_static}\\
		& \mathbf{f}_j^\mathrm{T} \mathbf{u}_j \le \overline{c}_j,\quad j \in {\mathcal{L}_\mathrm{s}}\label{eq:opt_orig_compliance}\\
		& \mathbf{a} \ge \mathbf{0},\label{eq:opt_orig_nonneg}
	\end{align}
\end{subequations}
where $\mathbf{K}_j(\mathbf{a}) \in \mathbb{S}^{n_{\mathrm{dof},j}}_{\succeq 0}$ and $\mathbf{M}_j(\mathbf{a}) \in \mathbb{S}^{n_\mathrm{dof},j}_{\succeq 0}$ stand for the stiffness and mass matrices, $n_{\mathrm{dof},j}$ is the number of degrees of freedom in the $j$-th load case, and $\mathbf{f}_j \in \mathbb{R}^{n_{\mathrm{dof},j}}$ stands for the generalized force vector of the $j$-th load case, which includes the nodal forces and moments. In addition, $\bm{\ell} \in \mathbb{R}^{n_\mathrm{e}}_{>0}$ and $\bm{\rho} \in \mathbb{R}^{n_\mathrm{e}}_{>0}$ are vectors of lengths and densities of the $n_\mathrm{e}$ elements. {Because the stiffness and mass matrices may vary across the load cases, the setting of \eqref{eq:opt_orig} covers varying kinematic boundary conditions.}     {Notice that \eqref{eq:opt_orig} is a non-smooth problem, as the free-vibration eigenvalue constraints \eqref{eq:opt_orig_fv} are nondifferentiable for repeated eigenvalues.}

In bending-resistant structures such as frames, the stiffness matrix is usually a {low-degree} polynomial function of the design variables {(see Appendix \ref{app:polydep} for more details)}, and follows from the assembly
\begin{equation}\label{eq:stiffness}
	\mathbf{K}_j (\mathbf{a}) = \sum_{e=1}^{n_\mathrm{e}} \left[a_e \mathbf{K}_{j,e}^{\langle 1\rangle} + a_e^2 \mathbf{K}_{j,e}^{\langle 2\rangle} + a_e^3 \mathbf{K}_{j,e}^{\langle 3\rangle}\right],
\end{equation}
where $\mathbf{K}_{j,e}^{\langle i\rangle} \succeq 0$, $e \in \{1, \dots, n_\mathrm{e}\}$, $i \in \{1,2,3\}$, are auxiliary element stiffness matrices. In \eqref{eq:stiffness}, the terms $ a_e \mathbf{K}_{j,e}^{\langle 1\rangle}$ {typically} cover the membrane behavior, whereas the higher-order terms $a_e^2 \mathbf{K}_{j,e}^{\langle 2\rangle}$ and $a_e^3 \mathbf{K}_{j,e}^{\langle 3\rangle}$ capture the bending effects.

For the mass matrix $\mathbf{M}_j(\mathbf{a})$, we assume the form {(again, see Appendix \ref{app:polydep})}
\begin{equation}\label{eq:mass}
	\mathbf{M}_j (\mathbf{a}) = \mathbf{M}_j^{\langle 0\rangle} + \sum_{e=1}^{n_\mathrm{e}} a_e\mathbf{M}_{j,e}^{\langle 1\rangle},
\end{equation}
in which $a_e \mathbf{M}_{j,e}^{\langle 1\rangle} \succeq 0$, $e \in \{1, \dots, n_\mathrm{e}\}$, are the element mass matrices and $\mathbf{M}_j^{\langle 0\rangle} \succeq 0$ constitutes a non-structural mass. In this paper, we adopt stiffness-consistent mass matrices, but any other form can be used assuming it maintains a linear dependence on the cross-section areas. An extension to nonlinear terms in $\mathbf{M}_j(\mathbf{a})$ might be possible but is not covered in this manuscript.

One of the difficulties of the optimization problem \eqref{eq:opt_orig} lies in the fact that \eqref{eq:opt_orig_fv} and \eqref{eq:opt_orig_static} include products of state variables $\mathbf{u}_j$ and $\mathbf{v}_j$ with design variables $\mathbf{a}$. In \cite[Proposition 2.3]{Achtziger2008}, it has been shown that the state variables $\mathbf{u}_j$ and $\mathbf{v}_j$ can be eliminated from the formulation, leading to an equivalent semidefinite programming reformulation in the design variables only:
\begin{subequations}\label{eq:opt}
\begin{align}
	\min_{\mathbf{a}\in \mathbb{R}^{n_{\mathrm{e}}}}\; & \sum_{e=1}^{n_\mathrm{e} } \rho_e \ell_e a_e\label{eq:opt_objective}\\
	\text{s.t.}\; & \mathbf{K}_j(\mathbf{a}) - \underline{\lambda}_j \mathbf{M}_j(\mathbf{a}) \succeq 0,\quad j \in {\mathcal{L}_\mathrm{fv}}\label{eq:opt_fv}\\
	& \begin{pmatrix}
		\overline{c}_j & -\mathbf{f}_j^\mathrm{T}\\
		-\mathbf{f}_j & \mathbf{K}_j(\mathbf{a})
	\end{pmatrix} \succeq 0,\quad j \in {\mathcal{L}_\mathrm{s}}\label{eq:opt_static}\\
	& \mathbf{a} \ge \mathbf{0},\label{eq:opt_nonneg}
\end{align}
\end{subequations}
where the constraints \eqref{eq:opt_fv} and \eqref{eq:opt_static} are matrix inequalities. These inequalities are generally nonconvex due to their polynomial dependence on $\mathbf{a}$, and thus obtaining a (global) solution to \eqref{eq:opt} remains challenging. {However, the non-differentiability of the free-vibration eigenvalue constraints is avoided, as the Rayleigh quotient \eqref{eq:opt_orig_fv} is replaced by a polynomial matrix inequality \eqref{eq:opt_fv}.}

From the formulation \eqref{eq:opt}, we can state the following basic properties of the feasible designs $\hat{\mathbf{a}}$:
\begin{proposition}[Static admissibility {\cite[Proposition 1]{Tyburec2021}}]\label{lemma:statical_admissibility}
	Let $$\mathbf{\hat{a}} \in \left\{\mathbf{a} \; \vert \; \begin{pmatrix}
		\overline{c}_j & -\mathbf{f}_j^\mathrm{T}\\
		-\mathbf{f}_j & \mathbf{K}_j(\mathbf{a})
	\end{pmatrix} \succeq 0, \mathbf{a}\ge \mathbf{0}\right\}$$ hold true. Then, $\mathbf{f}_j \in \mathrm{Im}\left(\mathbf{K}_j(\hat{\mathbf{a}})\right)$.
\end{proposition}
\begin{proposition}[Mass admissibility]\label{lemma:dynamic_admissibility}
	Let $$\mathbf{\hat{a}}\in\left\{\mathbf{a} \; \vert \; \mathbf{K}_j(\mathbf{a}) - \underline{\lambda}_j \mathbf{M}_j(\mathbf{a}) \succeq 0, \mathbf{a}\ge \mathbf{0}\right\}.$$ Then, $\mathrm{Im}\left(\mathbf{M}_j(\hat{\mathbf{a}})\right) \subseteq \mathrm{Im}\left(\mathbf{K}_j(\hat{\mathbf{a}})\right)$.
\end{proposition}
\begin{proof}
	The matrix inequality \eqref{eq:opt_fv} is equivalent to
	\begin{equation}\label{eq:inimage}
		\mathbf{v}_j^\mathrm{T}\mathbf{K}_j(\mathbf{a})\mathbf{v}_j - \underline{\lambda}_j \mathbf{v}_j^\mathrm{T}\mathbf{M}_j(\mathbf{a})\mathbf{v}_j\ge 0\quad \forall \mathbf{v}_j \in \mathbb{R}^{n_{\mathrm{dof},j}}.
	\end{equation}
	Let us now proceed by contradiction and assume that $\mathrm{Im}\left(\mathbf{M}_j(\mathbf{a})\right) \nsubseteq \mathrm{Im}\left(\mathbf{K}_j(\mathbf{a})\right)$. Then, there must exist $\hat{\mathbf{v}}_j\in\left\{\mathbf{v}_j\;\vert\; \mathbf{v}_j\in \mathrm{Im}\left(\mathbf{M}_j(\mathbf{a})\right), \mathbf{v}_j\notin \mathrm{Im}\left(\mathbf{K}_j(\mathbf{a})\right)\right\}$, for which the left-hand side of \eqref{eq:inimage} evaluates as
	\begin{equation}
		-\underline{\lambda}_j \hat{\mathbf{v}}_j^\mathrm{T}\mathbf{M}_j(\mathbf{a})\hat{\mathbf{v}}_j < 0,
	\end{equation}
	which is strictly negative because $\underline{\lambda}_j>0$, $\mathbf{M}_j(\mathbf{a}) \succeq 0$, and $\hat{\mathbf{v}}_j\in \mathrm{Im}\left(\mathbf{M}_j(\mathbf{a})\right)$, contradicting~\eqref{eq:inimage}.
\end{proof}
Propositions \ref{lemma:statical_admissibility} and \ref{lemma:dynamic_admissibility} ensure that any feasible point to the problem \eqref{eq:opt} {with positive compliance} is statically admissible, i.e., can carry the design loads, {and with a positive fundamental free-vibration eigenvalue is} mass admissible, carrying the (non-)structural masses.

{In the opposite direction}, we have {that if the design is able to carry the (non)structural masses, then it exhibits a finite positive fundamental free-vibration eigenvalue.}

%\begin{proposition}[{Positive finite compliance}]\label{prop:exists_gamma_c}
%Let $\mathbf{f}_j \in \mathrm{Im}(\mathbf{K}_j(\mathbf{a}))$. Then, there exists $\gamma>0$ such that
%
%$$\begin{pmatrix}
%	\gamma & -\mathbf{f}_j^\mathrm{T}\\
%	-\mathbf{f}_j & \mathbf{K}_j(\mathbf{a})
%\end{pmatrix}\succeq 0.$$
%\end{proposition}
%\begin{proof}
%	Let $\mathbf{U}$ denote the range space basis of $\mathbf{K}_j(\mathbf{a})$. After projecting the matrix inequality onto the range space using
% $$
% \begin{pmatrix}
%     1 & \mathbf{0}^\mathrm{T}\\
%     \mathbf{0} & \mathbf{U}
% \end{pmatrix},
% $$
% %
% the matrix inequality reads as
%	%
%$$
%\begin{pmatrix}
%	\gamma & -\mathbf{f}_j^\mathrm{T}\mathbf{U}\\
%	-\mathbf{U}^\mathrm{T}\mathbf{f}_j& \mathbf{U}^\mathrm{T}\mathbf{K}_j(\mathbf{a})\mathbf{U}\\
%\end{pmatrix}\succeq 0,
%$$
%
%where $\mathbf{U}^\mathrm{T}\mathbf{f}_j \neq \mathbf{0}$. Using the Schur complement lemma, we receive $\gamma = \mathbf{f}_j^\mathrm{T}\mathbf{U} \left[\mathbf{U}^\mathrm{T} \mathbf{K}_j(\mathbf{a}) \mathbf{U} \right]^{-1} \mathbf{U}^\mathrm{T}\mathbf{f}_j > 0$.
%\end{proof}
%
\begin{proposition}[{Positive finite eigenvalue}]\label{prop:exists_gamma_fv}
	Let $\mathrm{Im}(\mathbf{M}_j(\mathbf{a})) \subseteq \mathrm{Im}(\mathbf{K}_j(\mathbf{a}))$. Then, there exists $\hat{\lambda}>0$ such that
	$\mathbf{K}_j(\mathbf{a}) - \hat{\lambda} \mathbf{M}_j(\mathbf{a})\succeq 0$.
\end{proposition}
\begin{proof}
Since $\mathbf{K}_j(\mathbf{a}) \in \mathbb{S}^{{n_\mathrm{dof},j}}_{\succeq 0}$ and $\mathbf{M}_j(\mathbf{a}) \in \mathbb{S}^{{n_\mathrm{dof},j}}_{\succeq 0}$, the nonzero eigenvalues of the generalized eigenvalue problem $\mathbf{K}_j(\mathbf{a})\mathbf{v}_j - \lambda \mathbf{M}_j(\mathbf{a}) \mathbf{v}_j = \mathbf{0}$, with $\mathbf{v}_j \in \mathrm{Im}\left(\mathbf{M}_j(\mathbf{a})\right)$, must be real and positive. Thus, setting $0<\hat{\lambda}= \lambda_1$ with $\lambda_1$ being the lowest nonzero eigenvalue, the eigenvalue constraint is satisfied.
\end{proof}

\subsection{Properties of scalarized problem}\label{sec:scalar}
In this section, we introduce a bilevel programming reformulation of the problem \eqref{eq:opt_orig} and provide conditions for its solvability. This reformulation will further enable us to construct feasible solutions to \eqref{eq:opt_orig}.

Setting $\mathbf{a}(\delta) := \delta \tilde{\mathbf{a}}$ with $\tilde{\mathbf{a}}\ge \mathbf{0}$, the optimization problem \eqref{eq:opt} becomes equivalent to a bilevel optimization problem, with the upper level seeking \textit{admissible} ratios of the cross-section areas $\tilde{\mathbf{a}}$ and the lower level searching for the optimal scaling factor $\delta$ of the fixed cross-section ratios:
\begin{subequations}\label{eq:bilevel}
	\begin{align}
		\min_{\tilde{\mathbf{a}}\in \mathbb{R}^{n_{\mathrm{e}}}\!,\; \delta^*\in\mathbb{R}}\; & \delta^* \sum_{e=1}^{n_\mathrm{e}} \rho_e \ell_e \tilde{a}_e\\
		\mathrm{s.t.}\ & \tilde{\mathbf{a}} \ge \mathbf{0},\label{eq:bilevel-atilde}\\
		& \mathbf{f}_j \in \mathrm{Im}\left(\mathbf{K}_j(\tilde{\mathbf{a}})\right), \quad j \in {\mathcal{L}_\mathrm{s}}\label{eq:bilevel_static}\\
		& \mathrm{Im}\left(\mathbf{M}_{j}(\tilde{\mathbf{a}})\right) \subseteq \mathrm{Im}\left(\mathbf{K}_{j}(\tilde{\mathbf{a}})\right),\quad j \in {\mathcal{L}_\mathrm{fv}}\label{eq:bilevel_mass}\\
		& \mathbf{W}_j^\mathrm{T}\left[ \sum_{e=1}^{n_\mathrm{e}}\tilde{a}_e\mathbf{K}_{j,e}^{\langle 1\rangle} - \underline{\lambda}_j \sum_{e=1}^{n_\mathrm{e}} \tilde{a}_e \mathbf{M}_{j,e}^{\langle 1\rangle} \right]\mathbf{W}_j \succeq 0, \quad j \in {\mathcal{L}_\mathrm{fv}}\label{eq:bilevel-membraine}\\
		& \begin{aligned}\label{eq:bilevel-lower}
			\delta^* \in \argmin_{\delta\in\mathbb{R}}\; & \delta\\
			\mathrm{s.t.}\; & \begin{pmatrix}
				\overline{c}_j & -\mathbf{f}_j^\mathrm{T}\\
				-\mathbf{f}_j & \mathbf{K}_j (\delta \tilde{\mathbf{a}})
			\end{pmatrix} \succeq 0, \quad j \in {\mathcal{L}_\mathrm{s}}\\
			& \mathbf{K}_j(\delta \tilde{\mathbf{a}}) - \underline{\lambda}_j \mathbf{M}_j(\delta \tilde{\mathbf{a}}) \succeq 0, \quad j \in {\mathcal{L}_\mathrm{fv}}\\
			& \delta \ge 0,
		\end{aligned}
	\end{align}
\end{subequations}
where $\mathbf{W}_j$ is an orthonormal basis of $\mathrm{Ker}\left(\sum_{e=1}^{n_\mathrm{e}}\left[\mathbf{K}_{j,e}^{\langle 2\rangle} + \mathbf{K}_{j,e}^{\langle 3\rangle}\right]\right)$ and the constraints \eqref{eq:bilevel-atilde}-\eqref{eq:bilevel-membraine} ensure that the lower-level problem \eqref{eq:bilevel-lower} {has a non-empty interior}. In particular, these {constraints} include nonnegative components of the cross-sectional ratios \eqref{eq:bilevel-atilde}, static \eqref{eq:bilevel_static} and mass \eqref{eq:bilevel_mass} admissibility of the ratios, and the pure membrane eigenvalues to be at least $\underline{\lambda}_j$ \eqref{eq:bilevel-membraine}, which we will prove in Proposition \ref{th:feasibility_cond} (that uses the notation set up in \eqref{eq:hatnotation}). Moreover, the lower-level problem \eqref{eq:bilevel-lower} is quasi-convex in $\delta$, allowing one to obtain a globally optimal scaling $\delta^*$ for fixed ratios $\tilde{\mathbf{a}}$ using bisection. Proving the quasi-convexity property and showing that the conditions \eqref{eq:bilevel-atilde}--\eqref{eq:bilevel-membraine} are sufficient and necessary for solvability of the lower-level problem \eqref{eq:bilevel-lower} constitutes the main goal of this section.

Let us first focus on the lower-level problem \eqref{eq:bilevel-lower}. We start by expressing the lowest nonzero
%non-singular 
free-vibration eigenvalue function $\lambda_{\min,j}(\delta \tilde{\mathbf{a}})$ using the Raleigh quotient \cite[Definition 2.2 and Appendix A.1]{Achtziger2008} as
\begin{equation}\label{eq:RQ}
	\lambda_{\min,j}(\delta \tilde{\mathbf{a}}) = \min_{\mathbf{v}_j \in \mathbb{R}^{n_{\mathrm{dof},j}}\setminus \mathrm{Ker}\left(\hat{\mathbf{M}}_j^{\langle 0\rangle} + \hat{\mathbf{M}}_j^{\langle 1\rangle}\right)} \frac{\mathbf{v}_j^\mathrm{T} \left[\delta\hat{\mathbf{K}}_j^{\langle 1\rangle} + \delta^2\hat{\mathbf{K}}_j^{\langle 2\rangle} + \delta^3\hat{\mathbf{K}}_j^{\langle 3\rangle} \right]  \mathbf{v}_j}{\mathbf{v}_j^\mathrm{T} \hat{\mathbf{M}}_j^{\langle 0\rangle} \mathbf{v}_j + \mathbf{v}_j^\mathrm{T} \delta\hat{\mathbf{M}}_j^{\langle 1\rangle} \mathbf{v}_j},
\end{equation}
in which we adopt the simplified notation
\begin{equation}\label{eq:hatnotation}
	\hat{\mathbf{K}}_j^{\langle i\rangle}(\tilde{\mathbf{a}}) = \sum_{e=1}^{n_\mathrm{e}}\tilde{a}_e^i\mathbf{K}_{j,e}^{\langle i\rangle}, \qquad 
	\hat{\mathbf{M}}_j^{\langle i\rangle}(\tilde{\mathbf{a}}) = \sum_{e=1}^{n_\mathrm{e}}\tilde{a}_e^i \mathbf{M}_{j,e}^{\langle i\rangle}.
\end{equation}
and omit the argument of $\hat{\mathbf{K}}_j^{\langle i\rangle}$ and $\hat{\mathbf{M}}_j^{\langle i\rangle}$ 
when it is clear from the context, to shorten the equations.

Because $\forall j \in \mathcal{L}_{\mathrm{fv}}: \mathrm{Im}\left(\mathbf{M}_{j}(\tilde{\mathbf{a}})\right) \subseteq \mathrm{Im}\left(\mathbf{K}_{j}(\tilde{\mathbf{a}})\right)$ (see Proposition~\ref{prop:exists_gamma_fv}), \eqref{eq:RQ} renders that $\lambda_{\min,j}>0$. Thus, to prove feasibility of the eigenvalue constraint in \eqref{eq:bilevel-lower} with respect to $\delta$, it remains to find the condition under which $\sup_{\delta} \lambda_{\min,j} \ge \underline{\lambda}_j$. 

Because of $\delta > 0$, equation~\eqref{eq:RQ} can further be rewritten to
\begin{equation}\label{eq:RQs}
	\begin{multlined}
		\lambda_{\min,j}(\delta \tilde{\mathbf{a}}) =\\
		\inf_{\mathbf{v}_j \in \mathbb{R}^{n_{\mathrm{dof},j}}\setminus \mathrm{Ker}\left(\hat{\mathbf{M}}_j^{\langle 0\rangle}+\hat{\mathbf{M}}_j^{\langle 1\rangle}\right)} \frac{\mathbf{v}_j^\mathrm{T} \hat{\mathbf{K}}_j^{\langle 1\rangle}\mathbf{v}_j + \delta \mathbf{v}_j^\mathrm{T}\hat{\mathbf{K}}_j^{\langle 2\rangle} \mathbf{v}_j + \delta^2 \mathbf{v}_j^\mathrm{T}\hat{\mathbf{K}}_j^{\langle 3\rangle} \mathbf{v}_j}{\frac{1}{\delta}\mathbf{v}_j^\mathrm{T} \hat{\mathbf{M}}_j^{\langle 0\rangle} \mathbf{v}_j + \mathbf{v}_j^\mathrm{T} \hat{\mathbf{M}}_j^{\langle 1\rangle} \mathbf{v}_j}.
	\end{multlined}
\end{equation}
From \eqref{eq:RQs} it follows that if $\mathbf{v}_j \in \mathrm{Ker}\left(\hat{\mathbf{K}}_j^{\langle 2\rangle} + \hat{\mathbf{K}}_j^{\langle 3\rangle}\right)$ and $\hat{\mathbf{M}}_j^{\langle 0\rangle} = \mathbf{0}$, then $\lambda_{\min,j}(\delta \tilde{\mathbf{a}})$ is a constant function. %As reported by Tong \textit{et al.}~\cite{Tong2000}, such situation appears, e.g., in the optimization of truss structures.

The next proposition shows that, in the general case, $\lambda_{\min,j}(\delta \tilde{\mathbf{a}})$ is a non-decreasing function of $\delta$:

\begin{proposition}\label{prop:monotonic}
Let $\lambda_{\min,j}(\delta \tilde{\mathbf{a}})$ be as defined in \eqref{eq:RQs} and $j \in \mathcal{L}_\mathrm{fv}$. Then, $\lambda_{\min,j}(\delta \tilde{\mathbf{a}})$ is a non-decreasing function.
\begin{proof}
Let $\Delta > 0$; we need to show that $\lambda_{\min,j}((\delta+\Delta) \tilde{\mathbf{a}}) - \lambda_{\min,j}(\delta \tilde{\mathbf{a}})\ge 0$. Using the definition \eqref{eq:RQs}, we have %\textcolor{purple}{ We don't really need to write equations (3.11), it is clear enough to directly proceed to (3.12). }
\begin{multline}
	\lambda_{\min,j}((\delta + \Delta)\tilde{\mathbf{a}}) - \lambda_{\min,j}(\delta \tilde{\mathbf{a}})=\\
	\inf_{\mathbf{v}_j \in \mathbb{R}^{n_{\mathrm{dof},j}}\setminus \mathrm{Ker}\left(\hat{\mathbf{M}}_j^{\langle 0\rangle}+\hat{\mathbf{M}}_j^{\langle 1\rangle}\right)} 
	\frac{\mathbf{v}_j^\mathrm{T} \hat{\mathbf{K}}_j^{\langle 1\rangle}\mathbf{v}_j + (\delta+\Delta) \mathbf{v}_j^\mathrm{T}\hat{\mathbf{K}}_j^{\langle 2\rangle} \mathbf{v}_j + (\delta+\Delta)^2 \mathbf{v}_j^\mathrm{T}\hat{\mathbf{K}}_j^{\langle 3\rangle} \mathbf{v}_j}{\frac{1}{\delta + \Delta }\mathbf{v}_j^\mathrm{T} \hat{\mathbf{M}}_j^{\langle 0\rangle} \mathbf{v}_j + \mathbf{v}_j^\mathrm{T} \hat{\mathbf{M}}_j^{\langle 1\rangle} \mathbf{v}_j} \\
	- \inf_{\mathbf{v}_j \in \mathbb{R}^{n_{\mathrm{dof},j}}\setminus \mathrm{Ker}\left(\hat{\mathbf{M}}_j^{\langle 0\rangle}+\hat{\mathbf{M}}_j^{\langle 1\rangle}\right)} \frac{\mathbf{v}_j^\mathrm{T} \hat{\mathbf{K}}_j^{\langle 1\rangle}\mathbf{v}_j + \delta \mathbf{v}_j^\mathrm{T}\hat{\mathbf{K}}_j^{\langle 2\rangle} \mathbf{v}_j + \delta^2 \mathbf{v}_j^\mathrm{T}\hat{\mathbf{K}}_j^{\langle 3\rangle} \mathbf{v}_j}{\frac{1}{\delta}\mathbf{v}_j^\mathrm{T} \hat{\mathbf{M}}_j^{\langle 0\rangle} \mathbf{v}_j + \mathbf{v}_j^\mathrm{T} \hat{\mathbf{M}}_j^{\langle 1\rangle} \mathbf{v}_j}.
\end{multline}
Splitting the first infimum into two parts using the triangle inequality, we obtain 
\begin{multline}\label{eq:lamdiff}
	\lambda_{\min,j}((\delta + \Delta)\tilde{\mathbf{a}}) - \lambda_{\min,j}(\delta \tilde{\mathbf{a}})\ge\\
	\inf_{\mathbf{v}_j \in \mathbb{R}^{n_{\mathrm{dof},j}}\setminus \mathrm{Ker}\left(\hat{\mathbf{M}}_j^{\langle 0\rangle}+\hat{\mathbf{M}}_j^{\langle 1\rangle}\right)} 
	\frac{\mathbf{v}_j^\mathrm{T} \hat{\mathbf{K}}_j^{\langle 1\rangle}\mathbf{v}_j + \delta \mathbf{v}_j^\mathrm{T}\hat{\mathbf{K}}_j^{\langle 2\rangle} \mathbf{v}_j + \delta^2 \mathbf{v}_j^\mathrm{T}\hat{\mathbf{K}}_j^{\langle 3\rangle} \mathbf{v}_j}{\frac{1}{\delta + \Delta}\mathbf{v}_j^\mathrm{T} \hat{\mathbf{M}}_j^{\langle 0\rangle} \mathbf{v}_j + \mathbf{v}_j^\mathrm{T} \hat{\mathbf{M}}_j^{\langle 1\rangle} \mathbf{v}_j} \\
	+ \inf_{\mathbf{v}_j \in \mathbb{R}^{n_{\mathrm{dof},j}}\setminus \mathrm{Ker}\left(\hat{\mathbf{M}}_j^{\langle 0\rangle}+\hat{\mathbf{M}}_j^{\langle 1\rangle}\right)} 
	\frac{\Delta \mathbf{v}_j^\mathrm{T}\hat{\mathbf{K}}_j^{\langle 2\rangle} \mathbf{v}_j + (2\delta \Delta + \Delta^2) \mathbf{v}_j^\mathrm{T}\hat{\mathbf{K}}_j^{\langle 3\rangle} \mathbf{v}_j}{\frac{1}{\delta + \Delta}\mathbf{v}_j^\mathrm{T} \hat{\mathbf{M}}_j^{\langle 0\rangle} \mathbf{v}_j + \mathbf{v}_j^\mathrm{T} \hat{\mathbf{M}}_j^{\langle 1\rangle} \mathbf{v}_j}\\
	- \inf_{\mathbf{v}_j \in \mathbb{R}^{n_{\mathrm{dof},j}}\setminus \mathrm{Ker}\left(\hat{\mathbf{M}}_j^{\langle 0\rangle}+\hat{\mathbf{M}}_j^{\langle 1\rangle}\right)} \frac{\mathbf{v}_j^\mathrm{T} \hat{\mathbf{K}}_j^{\langle 1\rangle}\mathbf{v}_j + \delta \mathbf{v}_j^\mathrm{T}\hat{\mathbf{K}}_j^{\langle 2\rangle} \mathbf{v}_j + \delta^2 \mathbf{v}_j^\mathrm{T}\hat{\mathbf{K}}_j^{\langle 3\rangle} \mathbf{v}_j}{\frac{1}{\delta}\mathbf{v}_j^\mathrm{T} \hat{\mathbf{M}}_j^{\langle 0\rangle} \mathbf{v}_j + \mathbf{v}_j^\mathrm{T} \hat{\mathbf{M}}_j^{\langle 1\rangle} \mathbf{v}_j}.
\end{multline}
Because $\forall \delta,\Delta >0: \frac{1}{\delta} \ge \frac{1}{\delta + \Delta}$, it holds for all $\mathbf{v}_j$, including the minimizing ones, that
\begin{equation}
	\frac{1}{\delta}\mathbf{v}_j^\mathrm{T} \hat{\mathbf{M}}_j^{\langle 0\rangle} \mathbf{v}_j \ge \frac{1}{\delta + \Delta}\mathbf{v}_j^\mathrm{T} \hat{\mathbf{M}}_j^{\langle 0\rangle} \mathbf{v}_j.
\end{equation}
Consequently,
\begin{multline}
	\inf_{\mathbf{v}_j \in \mathbb{R}^{n_{\mathrm{dof},j}}\setminus \mathrm{Ker}\left(\hat{\mathbf{M}}_j^{\langle 0\rangle}+\hat{\mathbf{M}}_j^{\langle 1\rangle}\right)} 
	\frac{\mathbf{v}_j^\mathrm{T} \hat{\mathbf{K}}_j^{\langle 1\rangle}\mathbf{v}_j + \delta \mathbf{v}_j^\mathrm{T}\hat{\mathbf{K}}_j^{\langle 2\rangle} \mathbf{v}_j + \delta^2 \mathbf{v}_j^\mathrm{T}\hat{\mathbf{K}}_j^{\langle 3\rangle} \mathbf{v}_j}{\frac{1}{\delta + \Delta}\mathbf{v}_j^\mathrm{T} \hat{\mathbf{M}}_j^{\langle 0\rangle} \mathbf{v}_j + \mathbf{v}_j^\mathrm{T} \hat{\mathbf{M}}_j^{\langle 1\rangle} \mathbf{v}_j} \\
	- \inf_{\mathbf{v}_j \in \mathbb{R}^{n_{\mathrm{dof},j}}\setminus \mathrm{Ker}\left(\hat{\mathbf{M}}_j^{\langle 0\rangle}+\hat{\mathbf{M}}_j^{\langle 1\rangle}\right)} \frac{\mathbf{v}_j^\mathrm{T} \hat{\mathbf{K}}_j^{\langle 1\rangle}\mathbf{v}_j + \delta \mathbf{v}_j^\mathrm{T}\hat{\mathbf{K}}_j^{\langle 2\rangle} \mathbf{v}_j + \delta^2 \mathbf{v}_j^\mathrm{T}\hat{\mathbf{K}}_j^{\langle 3\rangle} \mathbf{v}_j}{\frac{1}{\delta}\mathbf{v}_j^\mathrm{T} \hat{\mathbf{M}}_j^{\langle 0\rangle} \mathbf{v}_j + \mathbf{v}_j^\mathrm{T} \hat{\mathbf{M}}_j^{\langle 1\rangle} \mathbf{v}_j} \ge 0,
\end{multline}
which allows us to simplify \eqref{eq:lamdiff} to
\begin{equation}\label{eq:monotoneous}
\begin{multlined}
\lambda_{\min,j}((\delta + \Delta)\tilde{\mathbf{a}}) - \lambda_{\min,j}(\delta \tilde{\mathbf{a}}) \ge\\
 \inf_{\mathbf{v}_j \in \mathbb{R}^{n_{\mathrm{dof},j}}\setminus \mathrm{Ker}\left(\hat{\mathbf{M}}_j^{\langle 0\rangle}+\hat{\mathbf{M}}_j^{\langle 1\rangle}\right)} 
\frac{\Delta \mathbf{v}_j^\mathrm{T}\hat{\mathbf{K}}_j^{\langle 2\rangle} \mathbf{v}_j + (2\delta \Delta + \Delta^2) \mathbf{v}_j^\mathrm{T}\hat{\mathbf{K}}_j^{\langle 3\rangle} \mathbf{v}_j}{\frac{1}{\delta + \Delta}\mathbf{v}_j^\mathrm{T} \hat{\mathbf{M}}_j^{\langle 0\rangle} \mathbf{v}_j + \mathbf{v}_j^\mathrm{T} \hat{\mathbf{M}}_j^{\langle 1\rangle} \mathbf{v}_j} \ge 0.
\end{multlined}
\end{equation}
In \eqref{eq:monotoneous}, the nonnegativity follows from $\Delta \mathbf{v}_j^\mathrm{T}\hat{\mathbf{K}}_j^{\langle 2\rangle} \mathbf{v}_j \ge 0$ and $\left(2\delta \Delta + \Delta^2\right) \mathbf{v}_j^\mathrm{T}\hat{\mathbf{K}}_j^{\langle 3\rangle} \mathbf{v}_j\ge 0$, which is because of $\hat{\mathbf{K}}_j^{\langle 2\rangle}, \hat{\mathbf{K}}_j^{\langle 3\rangle} \succeq 0$ and $\delta\ge 0, \Delta > 0$.
\end{proof}
\end{proposition}

With this result, we have shown that by increasing $\delta$ we do not decrease $\lambda_{\min,j}(\delta \tilde{\mathbf{a}})$. {Notice that this holds up to $\delta\rightarrow \infty$, as we can choose $\Delta\rightarrow \infty$ in the above proof.} For the solvability of the lower-level problem \eqref{eq:bilevel-lower}, it is crucial to guarantee that there exits a $\delta$ such that $\lambda_{\min,j}(\delta \tilde{\mathbf{a}}) \ge \underline{\lambda}_j$. In other words, we require that $\sup_{\delta} \lambda_{\min,j}(\delta \tilde{\mathbf{a}}) \ge \underline{\lambda}_j$. To show this, let $\mathbf{v}_j := \mathbf{t}_j + \mathbf{w}_j$, where $\mathbf{t}_j \in \mathrm{Im}\left(\hat{\mathbf{K}}_j^{\langle 2\rangle} + \hat{\mathbf{K}}_j^{\langle 3\rangle}\right) + \mathrm{Ker}\left(\hat{\mathbf{M}}_j^{\langle 0\rangle} + \hat{\mathbf{M}}_j^{\langle 1\rangle}\right) $ and $\mathbf{w}_j \in \mathrm{Ker}\left(\hat{\mathbf{K}}_j^{\langle 2\rangle} + \hat{\mathbf{K}}_j^{\langle 3\rangle}\right) \cap \mathrm{Im}\left(\hat{\mathbf{M}}_j^{\langle 0\rangle} + \hat{\mathbf{M}}_j^{\langle 1\rangle}\right)$.
Then, we can write the expression for the supremum.

\begin{proposition}\label{prop:supremum}
Let $\lambda_{\min,j}(\delta \tilde{\mathbf{a}})$ be as defined in \eqref{eq:RQs}, and let $\mathbf{v}_j := \mathbf{t}_j + \mathbf{w}_j$ as defined above. Then, 
\begin{equation}\label{eq:supremum}
	\sup_{\delta} \lambda_{\min,j}(\delta \tilde{\mathbf{a}}) = \inf_{\mathbf{w}_j \in \mathrm{Ker}\left(\hat{\mathbf{K}}_j^{\langle 2\rangle} + \hat{\mathbf{K}}_j^{\langle 3\rangle}\right) \cap \mathrm{Im}\left(\hat{\mathbf{M}}_j^{\langle 0\rangle} + \hat{\mathbf{M}}_j^{\langle 1\rangle}\right)} \frac{\mathbf{w}_j^\mathrm{T} \hat{\mathbf{K}}_j^{\langle 1\rangle}\mathbf{w}_j}{\mathbf{w}_j^\mathrm{T} \hat{\mathbf{M}}_j^{\langle 1\rangle} \mathbf{w}_j}.
\end{equation}
\begin{proof}
We prove the statement by splitting $\mathbf{v}_j$ into $\mathbf{t}_j$ and $\mathbf{w}_j$ and showing that it must hold that $\mathbf{t}_j=\mathbf{0}$.  For this splitting, the supremum is expressed as
\begin{equation}\label{eq:splitting}
	\begin{multlined}
		\sup_\delta \lambda_{\min,j}(\delta \tilde{\mathbf{a}})= \\
		\sup_\delta \inf\limits_{\mathbf{t}_j, \mathbf{w}_j}
		 \frac{\left(\mathbf{t}_j+\mathbf{w}_j\right)^\mathrm{T} \hat{\mathbf{K}}_j^{\langle 1\rangle}\left(\mathbf{t}_j+\mathbf{w}_j\right) + \delta\mathbf{t}_j^\mathrm{T} \hat{\mathbf{K}}_j^{\langle 2\rangle}\mathbf{t}_j + \delta^2\mathbf{t}_j^\mathrm{T} \hat{\mathbf{K}}_j^{\langle 3\rangle}\mathbf{t}_j }{\frac{1}{\delta}\left(\mathbf{t}_j+\mathbf{w}_j\right)^\mathrm{T} \hat{\mathbf{M}}_j^{\langle 0\rangle} \left(\mathbf{t}_j+\mathbf{w}_j\right) + \left(\mathbf{t}_j+\mathbf{w}_j\right)^\mathrm{T} \hat{\mathbf{M}}_j^{\langle 1\rangle} \left(\mathbf{t}_j+\mathbf{w}_j\right)}.
	\end{multlined}
\end{equation}
Because of the monotonicity (recall Proposition~\ref{prop:monotonic}), $\sup_\delta \lambda_{\min,j}(\delta \tilde{\mathbf{a}}) = \lim\limits_{\delta\rightarrow \infty}\lambda_{\min,j}(\delta \tilde{\mathbf{a}})$. Further, suppose by {contradiction} that $\mathbf{t}_j \neq \mathbf{0}$ holds at the infimum. Then, $\lambda_{\min,j}(\delta \tilde{\mathbf{a}}) \rightarrow \infty$ as $\delta \rightarrow \infty$, due to $\mathbf{w}_j \in \mathrm{Im}\left(\hat{\mathbf{M}}_j^{\langle 0\rangle} + \hat{\mathbf{M}}_j^{\langle 1\rangle}\right)$ and because $\mathbf{t}_j$ cannot be in $\mathrm{Ker}\left(\hat{\mathbf{K}}_j^{\langle 2\rangle} + \hat{\mathbf{K}}_j^{\langle 3\rangle}\right)$. On the other hand, for $\mathbf{t}_j = \mathbf{0}$, $\lambda_{\min,j}(\delta \tilde{\mathbf{a}})$ is finite for all $\delta$ due to $\mathbf{w}_j \in \mathrm{Im}\left(\hat{\mathbf{M}}_j^{\langle 0\rangle} + \hat{\mathbf{M}}_j^{\langle 1\rangle}\right)$. Consequently, $\mathbf{t}_j^*=\mathbf{0}$ must hold at the infimum. After inserting $\mathbf{t}_j^*=\mathbf{0}$ and $\delta^*\rightarrow \infty$ into \eqref{eq:splitting}, we receive \eqref{eq:supremum}, which completes the proof.
\end{proof}
\end{proposition}

%\begin{remark}
%	Based on Propositions~\ref{prop:monotonic} and \ref{prop:supremum}, $\delta$ does not influence the eigenvalues belonging to pure membrane eigenmodes. On the contrary, it enables increasing the eigenvalues associated with bending, up to the level $\delta\rightarrow \infty$, in which case the re-scaled design $\delta \tilde{\mathbf{a}}$ becomes infinitely stiff in bending.
%\end{remark}

Using Proposition \ref{prop:supremum}, it is now possible to state the conditions for $\tilde{\mathbf{a}}$ under which there exists a $\delta$ such that the free-vibration constraint in the lower-level problem \eqref{eq:bilevel-lower} is feasible:
\begin{proposition}\label{prop:scaling}
	There exists $\delta$ such that $\mathbf{K}_j(\delta \tilde{\mathbf{a}}) - \underline{\lambda}_j \mathbf{M}_j(\delta \tilde{\mathbf{a}}) \succeq 0$ if and only if $\tilde{\mathbf{a}} \in \left\{\mathbf{a} \;\vert\; \mathbf{a}\ge \mathbf{0}, \mathrm{Im}\left(\mathbf{M}_j(\mathbf{a})\right) \subseteq \mathrm{Im}\left(\mathbf{K}_j(\mathbf{a})\right),\mathbf{W}_j^\mathrm{T}\left[\hat{\mathbf{K}}_{j}^{\langle 1\rangle} - \underline{\lambda}_j \hat{\mathbf{M}}_{j}^{\langle 1\rangle} \right]\mathbf{W}_j \succeq 0\right\}$, where $\mathbf{W}_j$ denotes an orthonormal basis of $\mathrm{Ker}\left(\hat{\mathbf{K}}_j^{\langle 2\rangle} + \hat{\mathbf{K}}_j^{\langle 3\rangle}\right) \cap \mathrm{Im}\left(\hat{\mathbf{M}}_j^{\langle 0\rangle} + \hat{\mathbf{M}}_j^{\langle 1\rangle}\right)$.
\end{proposition}
\begin{proof}
\fbox{$\Rightarrow$} Let $\mathbf{K}_j(\delta \tilde{\mathbf{a}}) - \underline{\lambda}_j \mathbf{M}_j(\delta \tilde{\mathbf{a}}) \succeq 0$ hold. Then, $\mathrm{Im}\left(\mathbf{M}_j(\tilde{\mathbf{a}})\right) \subseteq \mathrm{Im}\left(\mathbf{K}_j(\tilde{\mathbf{a}})\right)$ based on Proposition \ref{lemma:dynamic_admissibility}. Next, we project the matrix inequality using $\mathbf{W}_j$, yielding
\begin{equation}
	\delta \mathbf{W}_{{j}}^\mathrm{T} \hat{\mathbf{K}}_j^{\langle 1\rangle} \mathbf{W}_j- \underline{\lambda}_j \delta \mathbf{W}_{{j}}^\mathrm{T} \hat{\mathbf{M}}_j^{\langle 1\rangle} \mathbf{W}_j - \underline{\lambda}_j\mathbf{W}_{{j}}^\mathrm{T} \hat{\mathbf{M}}_j^{\langle 0\rangle} \mathbf{W}_j \succeq 0.
\end{equation}
For $\delta>0$, this is further equivalent to
\begin{equation}
	\mathbf{W}_{{j}}^\mathrm{T} \hat{\mathbf{K}}_j^{\langle 1\rangle} \mathbf{W}_j- \underline{\lambda}_j \mathbf{W}_{{j}}^\mathrm{T} \hat{\mathbf{M}}_j^{\langle 1\rangle} \mathbf{W}_j \succeq \frac{1}{\delta}\underline{\lambda}_j\mathbf{W}_{{j}}^\mathrm{T} \hat{\mathbf{M}}_j^{\langle 0\rangle} \mathbf{W}_j \succeq 0.
\end{equation}
\fbox{$\Leftarrow$} Let $\mathrm{Im}\left(\mathbf{M}_j(\tilde{\mathbf{a}})\right) \subseteq \mathrm{Im}\left(\mathbf{K}_j(\tilde{\mathbf{a}})\right)$ and $\mathbf{W}_j^\mathrm{T}\left[\hat{\mathbf{K}}_{j}^{\langle 1\rangle} - \underline{\lambda}_j \hat{\mathbf{M}}_{j}^{\langle 1\rangle} \right]\mathbf{W}_j \succeq 0$ hold. Because of the first assumption and Proposition \ref{prop:exists_gamma_fv}, it holds that
\begin{equation}
\forall \mathbf{v}_j \in \mathrm{Im}\left(\mathbf{K}_j(\tilde{\mathbf{a}})\right) \exists \gamma>0: \mathbf{v}_j^\mathrm{T}\mathbf{K}_j(\tilde{\mathbf{a}})\mathbf{v}_j - \gamma\mathbf{v}_j^\mathrm{T}\mathbf{M}_j(\tilde{\mathbf{a}})\mathbf{v}_j \ge 0,
\end{equation}
and thus also $\mathbf{K}_j(\tilde{\mathbf{a}}) - \gamma\mathbf{M}_j(\tilde{\mathbf{a}}) \succeq 0$ for $\gamma > 0$. Consequently, it remains to show that $\gamma\ge \underline{\lambda}_j$.

Because of $\mathbf{W}_j^\mathrm{T}\left[\hat{\mathbf{K}}_{j}^{\langle 1\rangle} - \underline{\lambda}_j \hat{\mathbf{M}}_{j}^{\langle 1\rangle} \right]\mathbf{W}_j \succeq 0$, we also have
\begin{equation}\label{eq:inequality}
	\forall \mathbf{w}_j \in \mathrm{Ker}\left(\hat{\mathbf{K}}_j^{\langle 2\rangle} + \hat{\mathbf{K}}_j^{\langle 3\rangle}\right) \cap \mathrm{Im}\left(\hat{\mathbf{M}}_j^{\langle 0\rangle} + \hat{\mathbf{M}}_j^{\langle 1\rangle}\right):  \mathbf{w}_j^\mathrm{T} \left[\hat{\mathbf{K}}_{j}^{\langle 1\rangle} - \underline{\lambda}_j \hat{\mathbf{M}}_{j}^{\langle 1\rangle} \right] \mathbf{w}_j \ge 0.
\end{equation}
In particular, the above holds for $w_j$ minimizing the left-hand side of \eqref{eq:inequality}, yielding
\begin{equation}
	\underline{\lambda}_j \le \inf_{\mathbf{w}_j \in \mathrm{Ker}\left(\hat{\mathbf{K}}_j^{\langle 2\rangle} + \hat{\mathbf{K}}_j^{\langle 3\rangle}\right) \cap \mathrm{Im}\left(\hat{\mathbf{M}}_j^{\langle 0\rangle} + \hat{\mathbf{M}}_j^{\langle 1\rangle}\right)} \frac{\mathbf{w}_j^\mathrm{T} \hat{\mathbf{K}}_{j}^{\langle 1\rangle} \mathbf{w}_j}{\mathbf{w}_j^\mathrm{T} \hat{\mathbf{M}}_{j}^{\langle 1\rangle} \mathbf{w}_j} = \sup_{\delta} \lambda_{\min,j}(\mathbf{a}(\delta))
\end{equation}
based on Proposition \ref{prop:supremum}.
\end{proof}

Having stated the conditions for the feasibility of the free-vibration constraint, we extend this result to the setting of the lower-level program \eqref{eq:bilevel-lower} to account for compliance constraints as well.

\begin{proposition}\label{th:feasibility_cond}
    The optimization problem \eqref{eq:bilevel-lower} has a non-empty interior if and only if $\tilde{\mathbf{a}} \in \mathbb{R}^{n_{\mathrm{e}}}$ satisfies {the following feasibility conditions:}
    {\begin{itemize}
    \item $\tilde{\mathbf{a}} \ge \mathbf{0}$,
    \item $\mathbf{f}_j \in \mathrm{Im}\left(\mathbf{K}_j (\tilde{\mathbf{a}})\right)$ for all $j \in \mathcal{L}_\mathrm{s}$,
    \item $\mathrm{Im}\left(\mathbf{M}_j(\tilde{\mathbf{a}})\right) \subseteq \mathrm{Im}\left(\mathbf{K}_j(\tilde{\mathbf{a}})\right)$ for all $j \in \mathcal{L}_\mathrm{fv}$,
    \item $\mathbf{W}_j^\mathrm{T}\left[\hat{\mathbf{K}}_{j}^{\langle 1\rangle} - \underline{\lambda}_j \hat{\mathbf{M}}_{j}^{\langle 1\rangle} \right]\mathbf{W}_j \succeq 0$ for all $j \in \mathcal{L}_\mathrm{fv}$,
    \end{itemize}}%
    \noindent where $\mathbf{W}_j$ is an orthonormal basis of $\mathrm{Ker}\left(\hat{\mathbf{K}}_j^{\langle 2\rangle} + \hat{\mathbf{K}}_j^{\langle 3\rangle}\right) \cap \mathrm{Im}\left(\hat{\mathbf{M}}_j^{\langle 0\rangle} + \hat{\mathbf{M}}_j^{\langle 1\rangle}\right)$.
\end{proposition}
\begin{proof}
	Based on \cite[Proposition 5]{Tyburec2023}, for $\mathbf{f}_j \in \mathrm{Im}\left(\mathbf{K}_j (\tilde{\mathbf{a}})\right)$ there exists a finite $\delta_{\mathrm{c},j}^*>0$ such that the static matrix inequality in \eqref{eq:bilevel-lower} is satisfied. Conversely, for $\mathbf{f}_j \notin \mathrm{Im}\left(\mathbf{K}_j (\tilde{\mathbf{a}})\right)$, the optimization problem \eqref{eq:bilevel-lower} is infeasible.
	
	Using Proposition \ref{prop:scaling}, there exists $\delta_{\mathrm{fv},j}^*>0$ such that the free-vibration constraint is satisfied if and only if $\mathrm{Im}\left(\mathbf{M}_j(\tilde{\mathbf{a}})\right) \subseteq \mathrm{Im}\left(\mathbf{K}_j(\tilde{\mathbf{a}})\right)$ and $\mathbf{W}_j^\mathrm{T}\left[\hat{\mathbf{K}}_{j}^{\langle 1\rangle} - \underline{\lambda}_j \hat{\mathbf{M}}_{j}^{\langle 1\rangle} \right]\mathbf{W}_j \succeq 0$.
    Finally, there exists $\delta^* = \max\{\max_{j{\in \mathcal{L}_\mathrm{s}}} \delta^*_{\mathrm{c},j}, \max_{j{\in \mathcal{L}_\mathrm{fv}}} \delta^*_{\mathrm{fv},j} \}>0$ such that the free-vibration and compliance constraints are satisfied if and only if the {right-hand-side} conditions of the proposition are satisfied. This is (i) because we have upper bounds for the compliance $\overline{c}_j$ and compliance is a nonincreasing function in $\delta$ \cite[Proposition 1]{Tyburec2023}, and (ii) because we have lower bounds for the lowest free-vibration eigenvalues $\underline{\lambda}_j$ and the lowest free-vibration eigenvalue function is a nondecreasing function in $\delta$ (Proposition \ref{prop:monotonic}).
\end{proof}

The {feasibility} conditions of Proposition \ref{th:feasibility_cond} allow for a representation using a linear semidefinite program, which provides an efficient procedure for finding a feasible point to the bilevel problem \eqref{eq:bilevel}.

\begin{lemma}\label{lemma:initial_ub}
	Let $0<\gamma$ be a (small enough) positive number and let
	\begin{subequations}\label{eq:linprob}
	\begin{align}
		\tilde{\mathbf{a}}^* \in \argmin_{\tilde{\mathbf{a}}\in\mathbb{R}^{n_\mathrm{e}}} \; &  \sum_{e=1}^{n_\mathrm{e}} \rho_e \ell_e \tilde{a}_e\\
		\mathrm{s.t.}\; & \mathbf{W}_j^\mathrm{T}\left[\hat{\mathbf{K}}_{j}^{\langle 1\rangle}(\tilde{\mathbf{a}}) - \underline{\lambda}_j \hat{\mathbf{M}}_{j}^{\langle 1\rangle}(\tilde{\mathbf{a}}) \right]\mathbf{W}_j \succeq 0, \quad j \in {\mathcal{L}_\mathrm{fv}}\\
		& \tilde{\mathbf{K}}_j(\tilde{\mathbf{a}}) - \gamma\mathbf{M}_j(\tilde{\mathbf{a}}) \succeq 0, \quad j \in {\mathcal{L}_\mathrm{fv}}\label{eq:constr_mineig}\\
		& \begin{pmatrix}
			\gamma^{-1} & \mathbf{f}_j^\mathrm{T}\\
			\mathbf{f}_j & \tilde{\mathbf{K}}_j(\tilde{\mathbf{a}})
		\end{pmatrix} \succeq 0, \quad j \in {\mathcal{L}_\mathrm{s}}\\
		& \tilde{\mathbf{a}} \ge \mathbf{0},
	\end{align}
	with the linearized stiffness matrix $\tilde{\mathbf{K}}_j (\tilde{\mathbf{a}}) = \sum_{i=1}^3\sum_{e=1}^{n_\mathrm{e}} \tilde{a}_e\mathbf{K}_{j,e}^{\langle i\rangle}$. Then, there exists $\delta$ such that the pair $(\tilde{\mathbf{a}}^*$, $\delta)$ is a feasible point to \eqref{eq:bilevel}.
\end{subequations}
\end{lemma}
\begin{proof}
	Because of $\mathrm{Span}(\mathbf{K}_j(\mathbf{\tilde{a}})) = \mathrm{Span}(\tilde{\mathbf{K}}_j(\mathbf{\tilde{a}}))$, we have $\mathbf{f}_j \in \mathrm{Im}\left(\mathbf{K}_j(\mathbf{\tilde{a}})\right)$ and $\mathrm{Im}\left(\mathbf{M}_j(\tilde{\mathbf{a}})\right) \subseteq \mathrm{Im}\left(\mathbf{K}_j(\mathbf{\tilde{a}})\right)$, based on Propositions \ref{lemma:statical_admissibility} and \ref{lemma:dynamic_admissibility}. Together with the linear constraint \eqref{eq:constr_mineig}, the {feasibility} conditions of Proposition \ref{th:feasibility_cond} are satisfied, so that there exists $\delta$ such that the pair $(\delta$, $\tilde{\mathbf{a}}^*)$ is feasible for \eqref{eq:bilevel}.
\end{proof}

Using $\tilde{\mathbf{a}}$ that satisfies the {feasibility} conditions of Proposition \ref{th:feasibility_cond}, the optimal solution $\delta^*$ of {\eqref{eq:bilevel-lower}} can be found using a bisection-type algorithm, since both types of matrix inequalities in \eqref{eq:bilevel-lower} are monotonic in $\delta$ (Proposition \ref{prop:monotonic} and \cite[Propositon 1]{Tyburec2023}). In fact, \eqref{eq:bilevel-lower} is therefore a quasiconvex optimization problem because $\delta$ is a minimized upper bound on monotonic functions.

\subsection{Moment-SOS}\label{sec:msos}

Let us now return to the single-level nonlinear semidefinite programming formulation \eqref{eq:opt} which is, after three small modifications, suitable for the deployment of the moment-sum-of-squares hierarchy. These modifications involve:
\begin{enumerate}
	\item bounding the design variables to make the feasible space compact to satisfy Assumption \ref{as:archimedean} that is needed for the convergence of the hierarchy, and to tighten the feasible space of the relaxations \cite[Appendix 2]{Tyburec2021};
	\item scaling the design variables to the $[-1,1]$ domain for numerically reliable solution;
        \item adding a weight constraint that tightens the feasible space of the moments in relaxations of degree greater than one.
\end{enumerate}
{Let us start with the first modification.} Since the objective function \eqref{eq:opt_objective}, representing the weight of the structure, is a conic combination of nonnegative variables, any feasible point to the optimization problem \eqref{eq:opt} provides an upper-bound on  the objective function value $\overline{w}$ that also bounds the cross-section areas from above. This upper bound can be obtained by solving the convex problem \eqref{eq:linprob}, providing the cross-sectional ratios $\tilde{\mathbf{a}}$, and{, subsequently,} by finding the minimal scaling factor $\delta^*$ by solving the lower-level problem \eqref{eq:bilevel-lower} globally using bisection. Consequently, the optimal cross-section areas are bounded as
\begin{equation}
	0 \le a_e \le \frac{\overline{w}}{\rho_e\ell_e}, \forall e \in \{1,\dots,n_\mathrm{e}\}.
\end{equation}
Furthermore, using the substitution $a_e = \frac{\overline{w}}{2\rho_e\ell_e} (a_{\mathrm{s},e} + 1)$, these bounds can be written as $a_{\mathrm{s},e}^2 \le 1$. %Similarly, we add a weight constraint to strengthen the relaxations as
%
%\begin{equation}
%    \sum_{e=1}^{n_\mathrm{e}} \rho_e \ell_e a_e \le \overline{w}\quad\longleftrightarrow\quad 2-n_\mathrm{e} - \sum_{e=1}^{n_\mathrm{e}} a_{\mathrm{s},e} \ge 0.
%\end{equation}
Similarly, we add a weight constraint $\sum_{e=1}^{n_\mathrm{e}} \rho_e \ell_e a_e \le \overline{w}$ to strengthen the relaxations.
The influence of this constraint on the feasible set of relaxations will be illustrated in Section~\ref{sec:illustrative}. Using the above substitution of $a_e$, we obtain the formulation
\begin{subequations}\label{eq:opt_po}
	\begin{align}
		\min_{\mathbf{a}_{{\mathrm{s}}}\in\mathbb{R}^{n_\mathrm{e}}}\; & {\frac{1}{2}} \overline{w} (n_\mathrm{e} + \mathbf{1}^\mathrm{T} \mathbf{a}_\mathrm{s})\label{eq:opt_objective_po}\\
		\text{s.t.}\; & \mathbf{K}_j(\mathbf{a}_\mathrm{s}) - \underline{\lambda}_j \mathbf{M}_j(\mathbf{a}_\mathrm{s}) \succeq 0, \quad j \in {\mathcal{L}_\mathrm{fv}}\label{eq:opt_fv_po}\\
		& \begin{pmatrix}
			\overline{c}_j & -\mathbf{f}_j^\mathrm{T}\\
			-\mathbf{f}_j & \mathbf{K}_j(\mathbf{a}_\mathrm{s})
		\end{pmatrix} \succeq 0,\quad j \in {\mathcal{L}_\mathrm{s}}\label{eq:opt_static_po}\\
            & 2-n_\mathrm{e} - \sum_{e=1}^{n_\mathrm{e}} a_{\mathrm{s},e} \ge 0,\label{eq:opt_weight_po}\\
		& 1 - a_\mathrm{s,e}^2 \ge 0, \quad e \in \{1,\dots,n_\mathrm{e}\}\label{eq:opt_nonneg_po}
	\end{align}
\end{subequations}
{in the scaled variables $\mathbf{a}_{\mathrm{s}} \in \mathbb{R}^{n_\mathrm{e}}$. Here,} we slightly abuse the notation and express the stiffness and mass matrices in terms of the scaled variables $\mathbf{a}_\mathrm{s}$ as $\mathbf{K}_j(\mathbf{a}_\mathrm{s})$ and $\mathbf{M}_j(\mathbf{a}_\mathrm{s})$, respectively. {Although this compactification means that \eqref{eq:opt} and \eqref{eq:opt_po} do not have the same feasible set (after rescaling), they maintain the same set of optimal solutions. Thus, solution to \eqref{eq:opt_po} also solves \eqref{eq:opt}.}

Because of the quadratic bound constraints \eqref{eq:opt_nonneg_po}, the optimization problem \eqref{eq:opt_po} satisfies Assumption \ref{as:archimedean} for the convergence of the moment-sum-of-squares hierarchy:

\begin{proposition}
The optimization problem \eqref{eq:opt_po} satisfies Assumption \ref{as:archimedean}.
\end{proposition}
\begin{proof}
	Let $\mathbf{\mathbf{G}}(\mathbf{a_\mathrm{s}}) \in \mathbb{S}_{\succeq 0}$ be a block-diagonal matrix with the blocks being the left-hand sides of the constraints \eqref{eq:opt_fv_po}--\eqref{eq:opt_nonneg_po}. Further, let $\mathbf{H} \in \mathbb{S}_{\succeq 0}$,
	$$\mathbf{H} = \begin{pmatrix}
		\mathbf{0} & \mathbf{0}\\
		\mathbf{0} & \mathbf{I}
	\end{pmatrix},$$
	be of the same dimensions and containing the identity matrix $\mathbf{I} \in \mathbb{S}^{n_\mathrm{e}}_{\succ 0}$, where the position of $\mathbf{I}$ in $\mathbf{H}$ matches the position of \eqref{eq:opt_nonneg_po} in $\mathbf{G}(\mathbf{a}_\mathrm{s})$.
	
	Because $\mathbf{H}$ is a matrix SOS polynomial according to Definition \ref{def:sos}, due to $\mathbf{H} = \mathbf{H}\mathbf{H}^\mathrm{T}$, we have
	\begin{equation}
		p(\mathbf{a}_\mathrm{s}) := \langle \mathbf{H}\mathbf{H}^\mathrm{T}, \mathbf{G}(\mathbf{a}_\mathrm{s}) \rangle = n_\mathrm{e} - \sum_{e=1}^{n_\mathrm{e}} a^2_{\mathrm{s},e},
	\end{equation}
	which shows that{, because of \eqref{eq:opt_nonneg_po},} the superlevel set $\left\{ \mathbf{a}_\mathrm{s}\in \mathbb{R}^{n_\mathrm{e}} \; \lvert \; p(\mathbf{a}_\mathrm{s}) \ge 0 \right\}$ is compact.
\end{proof}

\subsubsection{Feasible upper bounds from first-order moments}
In what follows, we show that the first-order moments obtained by solving any relaxation of \eqref{eq:opt_po} satisfy the {feasibility conditions} set in Proposition \ref{th:feasibility_cond}, which allows us to construct a feasible upper-bound solution to \eqref{eq:opt}.

To show this, let us introduce the notation
\begin{subequations}
	\begin{align}
		a_{\mathrm{rel},e}^{\langle 1\rangle}(\mathbf{y}) &= 0.5\,{\frac{\overline{w}}{\rho_e \ell_e}} (y_{a_{\mathrm{s},e}^1} + 1),\label{eq:rel_a_1}\\
		a_{\mathrm{rel},e}^{\langle 2\rangle}(\mathbf{y}) &= 0.25\, {\frac{\overline{w}^2}{\rho_e^2 \ell_e^2}} \left( y_{a_{\mathrm{s},e}^2} + 2 y_{a_{\mathrm{s},e}^1} + 1 \right),\label{eq:rel_a_2}\\
		a_{\mathrm{rel},e}^{\langle 3\rangle}(\mathbf{y}) &= 0.125\, {\frac{\overline{w}^3}{\rho_e^3 \ell_e^3}} \left( y_{a_{\mathrm{s},e}^3} + 3 y_{a_{\mathrm{s},e}^2} + 3 y_{a_{\mathrm{s},e}^1} + 1 \right)\label{eq:rel_a_3}
	\end{align}
\end{subequations}
for evaluating {expected values of the} monomials of the (unscaled) cross-section areas based on the (scaled) moments $\mathbf{y}$. In this notation, the superscript $\bullet^{\langle i\rangle}$ indicates that it corresponds to the $i$-th degree monomial of $\bullet$. We wish to emphasize here that, generally, $a_{\mathrm{rel},e}^{\langle 2\rangle}(\mathbf{y}) \neq \left[a_{\mathrm{rel},e}^{\langle 1\rangle}(\mathbf{y})\right]^2$ and $a_{\mathrm{rel},e}^{\langle 3\rangle}(\mathbf{y}) \neq \left[a_{\mathrm{rel},e}^{\langle 1\rangle}(\mathbf{y})\right]^3$ due to the relaxation procedure.

As the first step in the proof, we need to show that replacing higher-order moments with the powers of the first-order moments does not change the rank of the stiffness matrix. Since the stiffness matrix is a conic combination of positive semidefinite matrices with the coefficients being monomials of the cross-section areas, recall \eqref{eq:stiffness}, it suffices to show that strict positivity of $a_{\mathrm{rel},e}^{\langle 3\rangle}(\mathbf{y})$ or $a_{\mathrm{rel},e}^{\langle 2\rangle}(\mathbf{y})$ implies strict positivity of $a_{\mathrm{rel},e}^{\langle 1\rangle}(\mathbf{y})$.

\begin{proposition}\label{prop:inertia}
	Let $\mathbf{y}^{*(r)}$ be the optimal moments associated with the cross-section areas $\mathbf{a}$ obtained by solving the $r$-th degree relaxation. If $a_{\mathrm{rel},e}^{\langle 2\rangle}(\mathbf{y}^{*(r)}) > 0$ or $a_{\mathrm{rel},e}^{\langle 3\rangle}(\mathbf{y}^{*(r)}) > 0$, then $a_{\mathrm{rel},e}^{\langle 1\rangle}(\mathbf{y}^{*(r)}) > 0$.
	\begin{proof}
        By contradiction, let us assume that $a_{\mathrm{rel},e}^{\langle 1\rangle}(\mathbf{y}^{*(r)}) = 0$, i.e., $y_{a_{\mathrm{s},e}^1}^{*(r)} = -1$. Then, we need to show that $a_{\mathrm{rel},e}^{\langle 2\rangle}(\mathbf{y}^{*(r)}) = 0$ and $a_{\mathrm{rel},e}^{\langle 3\rangle}(\mathbf{y}^{*(r)}) = 0$.
		
		Consider first the case of $a_{\mathrm{rel},e}^{\langle 2\rangle}(\mathbf{y}^{*(r)})$. The moment matrix {(recall the definition in \eqref{eq:moment})} then contains the principal submatrix
		\begin{equation}
			\begin{pmatrix}
				1 & y_{a_{\mathrm{s},e}^1}^{*(r)}\\
				y_{a_{\mathrm{s},e}^1}^{*(r)} & y_{a_{\mathrm{s},e}^2}^{*(r)}
			\end{pmatrix} \succeq 0.
		\end{equation}
		The determinant of this matrix must be nonnegative, so that we obtain the inequality
		\begin{equation}
			y_{a_{\mathrm{s},e}^2}^{*(r)} - \left[y_{a_{\mathrm{s},e}^1}^{*(r)}\right]^2 \ge 0.
		\end{equation}
		After inserting $y_{a_{\mathrm{s},e}^1}^{*(r)} = -1$, we obtain $y_{a_{\mathrm{s},e}^2}^{*(r)} \ge 1$. Due to the bound constraints \eqref{eq:opt_nonneg_po}, $y_{a_{\mathrm{s},e}^2}^{*(r)} = 1$. Using $y_{a_{\mathrm{s},e}^1}^{*(r)} = -1$ with $y_{a_{\mathrm{s},e}^2}^{*(r)} = 1$, \eqref{eq:rel_a_2} is then evaluated as $a_{\mathrm{rel},e}^{\langle 2\rangle}(\mathbf{y}^{*(r)})=0$.
		
		For the case of, $a_{\mathrm{rel},e}^{\langle 3\rangle}(\mathbf{y}^{*(r)})$, the moment matrix contains the principal submatrix
		\begin{equation}
			\begin{pmatrix}
				1 & y_{a_{\mathrm{s},e}^1}^{*(r)} & y_{a_{\mathrm{s},e}^2}^{*(r)}\\
				y_{a_{\mathrm{s},e}^1}^{*(r)} & y_{a_{\mathrm{s},e}^2}^{*(r)} & y_{a_{\mathrm{s},e}^3}^{*(r)}\\
				y_{a_{\mathrm{s},e}^2}^{*(r)} & y_{a_{\mathrm{s},e}^3}^{*(r)} & y_{a_{\mathrm{s},e}^4}^{*(r)}
			\end{pmatrix} \succeq 0.
		\end{equation}
		After inserting $y_{a_{\mathrm{s},e}^1}^{*(r)} = -1$ and $y_{a_{\mathrm{s},e}^2}^{*(r)} = 1$, the determinant evaluates as
		\begin{equation}
			- \left(y_{a_{\mathrm{s},e}^3}^{*(r)} + 1\right)^2 \ge 0,
		\end{equation}
		so that $y_{a_{\mathrm{s},e}^3}^{*(r)}=-1$ is the only feasible solution. Inserting $y_{a_{\mathrm{s},e}^1}^{*(r)} = -1$, $y_{a_{\mathrm{s},e}^2}^{*(r)} = 1$, and $y_{a_{\mathrm{s},e}^3}^{*(r)} = -1$ into \eqref{eq:rel_a_3}, we again receive $a_{\mathrm{rel},e}^{\langle 3\rangle}(\mathbf{y}^{*(r)})=0$, which finishes the proof.
	\end{proof}
\end{proposition}

Using this result, we can now prove that $\tilde{\mathbf{a}} = \mathbf{a}_{\mathrm{rel},e}^{\langle 1\rangle}(\mathbf{y}^{*(r)})$ accommodates the conditions of Proposition \ref{prop:scaling} to satisfy the eigenvalue constraint.

\begin{proposition}\label{th:ub}
	Let $\mathbf{y}^{*(r)}$ be the optimal moments obtained by solving the $r$-th degree relaxation and set $\tilde{\mathbf{a}} := \mathbf{a}_{\mathrm{rel},e}^{\langle 1\rangle}(\mathbf{y}^{*(r)})$. Then, there exists $\delta$ such that
	\begin{equation}\label{eq:scalar_fvineq}
		\mathbf{K}_j(\delta\tilde{\mathbf{a}}) - \underline{\lambda}_j \mathbf{M}_j(\delta\tilde{\mathbf{a}}) \succeq 0\,.
	\end{equation}
\end{proposition}
\begin{proof}
	From the hierarchy construction {\eqref{eq:general_moment}}, the {localizing matrix \eqref{eq:localizing} will always contain $L_\mathbf{y}(1 \otimes \left[\mathbf{K}_j (\mathbf{a}) - \underline{\lambda} \mathbf{M}_j(\mathbf{a})\right])$ as principal submatrix due to} \eqref{eq:opt_fv}{, which} can be written in terms of the lifted variables as
	\begin{multline}\label{eq:fvrelaxed}
		\hat{\mathbf{K}}_j^{\langle 1\rangle}(\tilde{\mathbf{a}}) + \hat{\mathbf{K}}_j^{\langle 2\rangle}\left(\left[\mathbf{a}_{\mathrm{rel}}^{\langle 2\rangle}(\mathbf{y}^{*(r)})\right]^{\circ\frac{1}{2}}\right) + \hat{\mathbf{K}}_j^{\langle 3\rangle}\left(\left[\mathbf{a}_{\mathrm{rel}}^{\langle 3\rangle}(\mathbf{y}^{*(r)})\right]^{\circ\frac{1}{3}}\right) \\
		- \underline{\lambda}_j\left( \hat{\mathbf{M}}_j^{\langle 0\rangle} + \hat{\mathbf{M}}_j^{\langle 1\rangle}(\tilde{\mathbf{a}}) \right)\succeq 0,
	\end{multline}
	with $\bullet^{\circ q}$ raising each element in $\bullet$ to the power of $q$. We note here that the power factors in \eqref{eq:fvrelaxed} eliminate the powers introduced in the notation \eqref{eq:hatnotation}, so that \eqref{eq:fvrelaxed} is indeed linear in $\mathbf{a}_\mathrm{rel}^{(i)}$ for all $i$.
 
    After projecting via an orthonormal basis of $\mathrm{Ker}\left(\hat{\mathbf{K}}_j^{\langle 2\rangle} + \hat{\mathbf{K}}_j^{\langle 3\rangle}\right) \cap \mathrm{Im}\left(\hat{\mathbf{M}}_j^{\langle 0\rangle} + \hat{\mathbf{M}}_j^{\langle 1\rangle}\right)$ denoted with $\mathbf{W}_j$, we receive
	\begin{equation}\label{eq:rel_cond1}
		\mathbf{W}_j^\mathrm{T}\left[\hat{\mathbf{K}}_{j}^{\langle 1\rangle}(\tilde{\mathbf{a}}) - \underline{\lambda}_j \hat{\mathbf{M}}_{j}^{\langle 1\rangle}(\tilde{\mathbf{a}}) \right]\mathbf{W}_j \succeq \underline{\lambda}_j \mathbf{W}_j^\mathrm{T} \hat{\mathbf{M}}_{j}^{\langle 0\rangle}\mathbf{W}_j \succeq 0.
	\end{equation}
	Further, because of Proposition \ref{prop:inertia}, {$a_{\mathrm{rel},e}^{\langle 1 \rangle} \mathbf{K}_{j,e}^{\langle 1\rangle}\neq \mathbf{0}$} whenever {$a_{\mathrm{rel},e}^{\langle2\rangle} \mathbf{K}_{j,e}^{\langle 2\rangle} + a_{\mathrm{rel},e}^{\langle3 \rangle} \mathbf{K}_{j,e}^{\langle 3\rangle} \neq \mathbf{0}$}. Consequently, 
	\begin{multline}\label{eq:spans}
		\mathrm{Span}\left(\hat{\mathbf{K}}_j^{\langle 1\rangle} (\tilde{\mathbf{a}}) + \hat{\mathbf{K}}_j^{\langle 2\rangle}\left(\left[\mathbf{a}_{\mathrm{rel}}^{\langle 2\rangle}\right]^{\circ\frac{1}{2}}\right) + \hat{\mathbf{K}}_j^{\langle 3\rangle}\left(\left[\mathbf{a}_{\mathrm{rel}}^{\langle 3\rangle}\right]^{\circ\frac{1}{3}}\right)\right)\\ = \mathrm{Span}\left(\hat{\mathbf{K}}_j^{\langle 1 \rangle} (\tilde{\mathbf{a}}) + \hat{\mathbf{K}}_j^{\langle 2\rangle}(\tilde{\mathbf{a}}) + \hat{\mathbf{K}}_j^{{\langle 3\rangle}}(\tilde{\mathbf{a}}) \right).
	\end{multline}
	Thus, because \eqref{eq:fvrelaxed} is feasible due to $\mathbf{y}^{*(r)}$ being optimal, it holds based on Proposition \ref{lemma:dynamic_admissibility} and based on \eqref{eq:spans} that 
	\begin{equation}\label{eq:rel_cond2}
		\mathrm{Im}\left(\mathbf{M}_j(\tilde{\mathbf{a}})\right) \subseteq \mathrm{Im}\left(\mathbf{K}_j(\tilde{\mathbf{a}})\right).
	\end{equation}
	Using \eqref{eq:rel_cond1} and \eqref{eq:rel_cond2}, Proposition \ref{prop:scaling} certifies existence of $\delta$ such that the matrix inequality \eqref{eq:scalar_fvineq} holds.
\end{proof}

Having shown that first-order moments can be used to find a feasible solution to the free-vibration inequality, we combine this result with that of \cite{Tyburec2023} to show that upper bounds can still be constructed even if compliance constraints are present.

\begin{proposition}\label{prop:feasibleub}
	Let $\mathbf{y}^{*(r)}$ be the optimal moments obtained by solving the $r$-th degree relaxation and set $\tilde{\mathbf{a}} := \mathbf{a}_{\mathrm{rel},e}^{\langle 1\rangle}(\mathbf{y}^{*(r)})$. Then, there exists a $\delta$ such that $\mathbf{a}(\delta):=\delta\tilde{\mathbf{a}}$ is feasible to the lower-level problem \eqref{eq:bilevel-lower}.
\end{proposition}
\begin{proof}
	For free-vibration constraints, the proof follows from Proposition \ref{th:ub}. For static constraints, based on Proposition \ref{prop:inertia}, equation~\eqref{eq:spans} holds. Thus, since
	\begin{equation}
		\begin{pmatrix}
			\overline{c}_j & -\mathbf{f}_j^\mathrm{T}\\
			-\mathbf{f}_j & \hat{\mathbf{K}}_j^{\langle 1\rangle}(\tilde{\mathbf{a}}) + \hat{\mathbf{K}}_j^{\langle 2\rangle}\left(\left[\mathbf{a}_{\mathrm{rel}}^{\langle 2\rangle}\right]^{\circ\frac{1}{2}}\right) + \hat{\mathbf{K}}_j^{\langle 3\rangle}\left(\left[\mathbf{a}_{\mathrm{rel}}^{\langle 3\rangle}\right]^{\circ\frac{1}{3}}\right)
		\end{pmatrix} \succeq 0
	\end{equation}
	is satisfied by optimality of $\mathbf{y}^{*(r)}$, the condition $\mathbf{f}_j \in \mathbf{K}_j(\mathbf{\tilde{a}})$ holds true due to Proposition \ref{lemma:statical_admissibility}. Based on Proposition \ref{th:feasibility_cond}, the upper bound solution can thus be constructed.
	
	Consequently, there exists $\delta$ such that both types of constraints are satisfied simultaneously.
\end{proof}

\subsubsection{Convergence of the hierarchy}

When solving {arbitrary-degree} relaxation of \eqref{eq:opt_po}, we obtain a lower bound to the objective function value. Furthermore, using Proposition \ref{prop:feasibleub}, we reach a feasible upper-bound solution to \eqref{eq:bilevel}, and thus also to \eqref{eq:opt}. Note that we generally do not have an upper bound to the problem \eqref{eq:opt_po} due to possibly violated variable upper bounds and the weight constraint which were introduced to make the feasible set compact. Nevertheless, we can state a simple condition of global $\varepsilon$-optimality.

\begin{proposition}
	Let $\delta^* \tilde{\mathbf{a}}$ be a feasible {point} to \eqref{eq:opt} constructed based on Proposition \ref{prop:feasibleub}. Then,
	\begin{equation}
		(\delta^* - 1)\frac{\overline{w}}{2} \left(n_\mathrm{e} + \mathbf{1}^\mathrm{T} \mathbf{y}^{(r)}_{\mathbf{a}_{\mathrm{s}}^1}
        \right) \le \varepsilon
	\end{equation}
	is a sufficient condition of global $\varepsilon$-optimality of the objective function.
\end{proposition}
\begin{proof}
	The proof follows from expanding $(\delta^*-1) \sum_{e=1}^{n_\mathrm{e}}\rho_e\ell_e a_{\mathrm{rel},e}^{\langle 1\rangle}
(\mathbf{y}^{(r)}_{\mathbf{a}_{\mathrm{s}}})
 \le 
 \varepsilon$.
\end{proof}

In general, this condition may not be tight for $r\rightarrow \infty$. For example, if there is more than one but a finite number of global minimizers, $\tilde{\mathbf{a}}$ will be a point in the convex hull of the global minimizers, which may not be feasible to the original problem, see \cite[Section 4.2]{Tyburec2021} for a graphical illustration. A similar situation occurs if there are infinitely many global minimizers in a nonconvex set.

On the other hand, if the set of global minimizers is convex, then we can prove convergence $\delta^* \rightarrow 1$ as $r \rightarrow \infty$.

\begin{theorem}\label{th:convergence_up}
	Let $\delta^* \tilde{\mathbf{a}}$ be a feasible {point} to \eqref{eq:opt} constructed based on Proposition \ref{prop:feasibleub}. If the set of global minimizers is convex, then
	\begin{equation}
		(\delta^* - 1)\frac{\overline{w}}{2} \left(n_\mathrm{e} + \mathbf{1}^\mathrm{T} \mathbf{y}^{(r)}_{\mathbf{a}_{\mathrm{s}}^1}\right) = 0
	\end{equation}
	as $r\rightarrow\infty$.
\end{theorem}
\begin{proof}
	Because of Theorem \ref{th:convergence} and satisfied Assumption \ref{as:archimedean}, optimization over the feasible set $\mathcal{K}$ is equivalent to optimization over $\mathrm{Conv}\left(\mathcal{K}\right)$ \cite[Proposition 7]{Tyburec2021}. Thus, since $\mathrm{Conv}\left(\mathcal{K}\right)$ is compact, it can be expressed as a convex hull of its limit points $\mathbf{d}_1, \mathbf{d}_2,\dots$ as
	\begin{equation}
		\mathrm{Conv}\left(\mathcal{K}\right) = \mathrm{Conv}\left(\cup_{i=1}^{\infty}\mathbf{d}_i\right).
	\end{equation}
	Because we have assumed convexity of the set of global minimizers, there exists a convex set $\mathrm{Conv}\left(\cup_{i=1}^{\infty} \mathbf{d}_i^*\right) \subseteq \mathrm{Conv}\left(\mathcal{K}\right)$ with the points $\mathbf{d}_i^*$ belonging to the minimum.
\end{proof}

Comparing Theorem \ref{th:convergence_up} with the rank flatness condition of Curto and Fialkow \cite{Curto1996} (recall \eqref{eq:flatextension}), we observe that the condition in Theorem \ref{th:convergence_up} is numerically simpler to check, provides information about the quality of relaxations, and also allows to strengthen subsequent relaxations by compactifications based on the feasible upper bounds. However, the condition in Theorem \ref{th:convergence_up} remains restricted to the problems we have investigated in this manuscript. 

In terms of applicability to optimizing the weight of frame structures, both conditions complement each other. In particular, they certify finite convergence for optimization problems with a unique global minimizer{, which typically happens unless the problem exhibits structural symmetries}, but the rank-flatness condition also holds for finitely many minimizers. The latter setting does not apply to the condition in Theorem \ref{th:convergence_up}, which requires the set of global minimizers to be convex. However, the condition in Theorem \ref{th:convergence_up} also recognizes convergence for infinitely many minimizers in a convex set{, for which the rank does not stabilize}. Nevertheless, both conditions fail to recognize optimality if there are infinitely many minimizers in a nonconvex domain.

\subsubsection{Implementation remarks} 
In the previous sections, we have developed theoretical foundations for optimizing frame structures under free-vibration eigenvalue and compliance constraints. {For reader's convenience, we summarize the algorithm developed in this work in Fig.~\ref{fig:optflow}: We start by finding a feasible point $\delta^*\tilde{\mathbf{a}}$ to \eqref{eq:opt} by first getting $\tilde{\mathbf{a}}$ that solves \eqref{eq:linprob}, and then optimizing $\delta^*$ in \eqref{eq:bilevel-lower} using bisection. Then, we take the lowest relaxation order $r = \left\lceil \mathrm{deg}(\mathbf{K}(\mathbf{a}))/2 \right\rceil$ and solve the $r$-th degree relaxation of \eqref{eq:opt_po}, which is a compactified version on \eqref{eq:opt}. By construction of the compactification, \eqref{eq:opt} and \eqref{eq:opt_po} have the same set of global solutions, so we obtain a lower bound $f^{(r)}$ for \eqref{eq:opt}. Using the relaxation solution $\mathbf{y}^{*(r)}$, we set $\tilde{\mathbf{a}} = \mathbf{a}_{\mathrm{rel}}^{\langle 1\rangle}(\mathbf{y}^{*(r)})$ and compute an upper bound $\overline{w}$ to the objective function value of \eqref{eq:opt} by solving \eqref{eq:bilevel-lower} using bisection. If the difference between the upper and lower bound is not small enough, we increase the relaxation order $r$ and repeat the process. Finally, we return the globally $(\overline{w}-f^{(r)})$-optimal solution to \eqref{eq:opt}.}
\begin{figure}[!htbp]
    \centering
    \begin{tikzpicture}[node distance=0.5cm and 1.25cm]
      % Define styles for nodes
      \tikzstyle{block} = [draw, rectangle, rounded corners=5pt, text width=4.5cm, align=center, minimum height=1cm]
    \tikzstyle{decision} = [draw, shape=diamond, text width=1.95cm, align=center, inner sep=0pt, minimum height=1cm]
      \tikzstyle{line} = [draw, -latex, thick]
  
      % Nodes
    \node[block] (start) {Find feasible point to \eqref{eq:opt}: solve \eqref{eq:linprob} for $\tilde{\mathbf{a}}$, then optimize $\delta^*$ in \eqref{eq:bilevel-lower} via bisection. Set objective upper bound $\overline{w} = \delta^* \sum_{e=1}^{n_\mathrm{e}} \rho_e \ell_e \tilde{a}_e$.};
      \node[block, below=of start] (rinit) {Set $r = \left\lceil \mathrm{deg}(\mathbf{K}(\mathbf{a}))/2 \right\rceil$.};
      \node[block, below=of rinit] (relaxation) {Solve $r$-th degree relaxation of \eqref{eq:opt_po} (built according to \eqref{eq:general_moment}) to get a lower bound $f^{(r)}$ for \eqref{eq:opt} and the relaxation solution $\mathbf{y}^{*(r)}$.};
      \node[block, below=of relaxation] (ub) {Let $\tilde{\mathbf{a}} := \mathbf{a}_{\mathrm{rel}}^{\langle 1\rangle}(\mathbf{y}^{*(r)})$. Compute upper bound $\overline{w}$ to \eqref{eq:opt} by optimizing $\delta^*$ in \eqref{eq:bilevel-lower} using bisection.};
      \node[decision, below=of ub] (check) {$\overline{w}-f^{(r)}\approx 0$};
      \node[block, below=of check] (end) {Return globally $(\overline{w}-f^{(r)})$-optimal solution to \eqref{eq:opt}.};
      \node[block, right=of check] (increase) {Increase relaxation order $r$};
  
      % Connections
      \path[line] (start) -- (rinit);
      \path[line] (rinit) -- (relaxation);
      \path[line] (relaxation) -- (ub);
      \path[line] (ub) -- (check);
      \path[line] (check.south) -- ++(0,-0.25) node[midway,right]{yes} -- (end.north);
      \path[line] (check.east) -- ++(1cm,0) node[midway,above]{no} -- (increase.west);
      \path[line] (increase.north) |- (relaxation.east);
      \end{tikzpicture}
    \caption{{Flowchart of weight optimization under free-vibration eigenvalue and static compliance constraints. The process iteratively refines bounds on the global minimum of the objective function until convergence.}}
    \label{fig:optflow}
\end{figure}

Although these results are theoretically sufficient, the solution to the relaxations and construction of feasible upper bounds requires resolving a few details related to the numerical treatment. These are described in this subsection.

\paragraph{Accurate finite element model} In contrast to static problems, where it is sufficient to model each design element using a single finite element, this discretization is insufficient to accurately evaluate the free-vibration response. In particular, the value of the lowest free-vibration eigenvalue converges with the mesh refinement from above, which implies that optimized structures will always have lower fundamental eigenvalues than the prescribed bound $\underline{\lambda}$. Fortunately, a sufficiently accurate response already follows from the use of two finite elements per design element. For this reason, we adopt this setting in the next section.

\paragraph{Detecting zero cross-section areas} The {feasibility} conditions of Proposition \ref{th:feasibility_cond} needed for the construction of feasible upper bounds are based on the range space of $\mathbf{K}_j(\tilde{\mathbf{a}})$ and $\mathbf{M}_j(\tilde{\mathbf{a}})$. The range space is influenced by the values of $\tilde{\mathbf{a}}$, with possible rank drops in the case where $\tilde{a}_e = 0$ for some $e \in \{1,\dots,n_\mathrm{e}\}$. Therefore, it is crucial to decide which entries in $\tilde{\mathbf{a}}$ are positive and which are \textit{exactly} zero. However, in practical computations, we also have entries $\tilde{a}_e$ that are small but positive numbers instead of exact zeros. Clearly, this prevents a conclusive decision. To avoid this ambiguity, we have implemented an iterative scheme in which we first sort the values in $\tilde{\mathbf{a}}$ and successively investigate the influence of increasing (small) threshold values on the free vibration eigenvalue, satisfaction of the {feasibility} conditions in Proposition \ref{th:feasibility_cond}, and on the scaling factor $\delta$. In the case of constructing feasible upper bounds by solving \eqref{eq:bilevel-lower}, this means that we actually investigate a few different values of $\tilde{\mathbf{a}}$, which differ in the number of zeros, and return the lowest value of the scaling factor $\delta$ together with the associated vector~$\tilde{\mathbf{a}}$.

\paragraph{Strengthening the relaxations} The third point related to the numerical treatment deals with making the feasible space of the relaxations as tight as possible. In this direction, it helps to include constraints that are redundant {when directly solving} \eqref{eq:opt} {(using local optimization techniques for example)}, but at the same time tighten the set of moments in the relaxations. For example, $-1\le y_{a_{\mathrm{s},e}^1} \le 1$ is a weaker constraint than $y_{a_{\mathrm{s},e}^2} \le 1$ \cite[Appendix 2]{Tyburec2021}. Here, we exploit two techniques. First, if after solving a relaxation we reach a better feasible upper bound (up to a certain tolerance), we recompute the same relaxation with updated bounds for the variables, leading to better lower bounds. In some cases, this also decreases the number of relaxation degrees for convergence. Second, we impose a weight constraint based on the best-known upper-bound objective function value. Again, this strengthens the feasible set of relaxations and improves the lower bounds.

\paragraph{Possibility of using reduced polynomial basis} While the theoretical convergence results in this section rely on convergence of the hierarchy {using the canonical basis} as described in Section \ref{sec:background}, we observed that numerical convergence also occurs for a modified version of the hierarchy {after} removing all mixed terms from the basis, providing the so-called nonmixed term (NMT) basis \cite{handa2024tssos}
\begin{equation}
    \mathbf{b}_{\mathrm{NMT},r}(\mathbf{x}) = \begin{pmatrix}
        1 & x_1 & \dots & x_n & x_1^2 & \dots & x_n^2 & \dots x_n^r
    \end{pmatrix}.
\end{equation}
{Although we do not have guaranteed theoretical convergence, we have not encountered a problem for which the canonical basis converges and the NMT basis does not. Thus, we demonstrate numerical} advantages of the NMT basis {by} examples in Sections \ref{sec:ex20}--\ref{sec:ex52} below.

\section{Numerical examples}

In this section, we illustrate our theoretical results using {four} numerical examples. The first example is a small academic problem that provides insight into the challenges associated with the optimization of frame structures under the fundamental free-vibration eigenvalue constraints on one hand and, on the other hand, allows us to visualize the developed optimization technique. For this problem, we also reveal the effect of {feasible set tightening using} the weight constraint. The second example illustrates the applicability of the method to a global solution to slightly larger problems and compares the standard moment-sum-of-squares hierarchy with a hierarchy depending on the nonmixed term basis. {The third problem illustrates that the method handles multiple loading scenarios with varying kinematic boundary conditions.} In the final example, we solve a larger problem using the nonmixed basis alone.

The optimization problems were solved using a desktop computer, fitted with an Intel Xeon CPU E5-2630 v3 and $128$~GB of RAM. The associated code is implemented in MATLAB and available at \url{https://gitlab.com/tyburec/pof-dyna}. The optimization problems {(linear semidefinite programs)} were modeled using Yalmip \cite{Lofberg2004} and solved by the MOSEK optimizer \cite{mosek}. {The stopping criterion for the convergence of the hierarchy was the relative difference of the lower and upper bounds $10^{-5}$, which is consistent with the expected accuracy of the opimizer.}

\subsection{Illustrative problem}\label{sec:illustrative}

We start with an illustrative problem introduced in \cite{Ni2014} to highlight the challenges in the design of frame structures under the fundamental free-vibration eigenvalue constraints. The problem was also investigated in \cite{Yamada2015}.

Let us consider a frame with three elements and the finite element discretization depicted in Figure~\ref{fig:illustration2}, with the upper nodes of the structure clamped and a nonstructural mass of weight $0.5$~kg placed at the bottom. All finite elements share a circular cross-section shape, with the values reported in the cm$^2$ units. In particular, elements \squared{2} and \squared{5} have cross-sectional area $a_1$, while the remaining elements \squared{1}, $\squared{3}$, $\squared{4}$ and $\squared{6}$ share cross-sectional area $a_2$. The structure is made of a linear elastic material of Young modulus $E=210$~GPa and density $\rho=7,800$~kg/m$^3$.

\begin{figure}[!htbp]
\centering
\begin{subfigure}{0.45\linewidth}
    \begin{tikzpicture}
        \scaling{2}
        \point{a}{-1.000000}{1.000000};
        \point{b}{0.000000}{1.000000};
        \point{c}{1.000000}{1.000000};
        \point{d}{-0.500000}{0.500000};
        \point{e}{0.000000}{0.500000};
        \point{f}{0.500000}{0.500000};
        \point{g}{0.000000}{0.000000};
        \beam{2}{a}{d}
        \notation{4}{a}{d}[$1$]
        \beam{2}{b}{e}
        \notation{4}{b}{e}[$2$]
        \beam{2}{c}{f}
        \notation{4}{c}{f}[$3$]
        \beam{2}{d}{g}
        \notation{4}{d}{g}[$4$]
        \beam{2}{e}{g}
        \notation{4}{e}{g}[$5$]
        \beam{2}{f}{g}
        \notation{4}{f}{g}[$6$][0.25]
        \node[circle, fill,minimum size=10pt,inner sep=0pt, outer sep=0pt] at (g) {};
        \support{3}{a}[180];
        \support{3}{b}[180];
        \support{3}{c}[180];

        \notation{3}{a}{g}[][0.5][];
        \notation{3}{b}{g}[][0.5][];
        \notation{3}{c}{g}[][0.5][];

        \dimensioning{1}{a}{d}{-0.90}[$0.5$~m];
        \dimensioning{1}{d}{g}{-0.90}[$0.5$~m];
        \dimensioning{1}{g}{f}{-0.90}[$0.5$~m];
        \dimensioning{1}{f}{c}{-0.90}[$0.5$~m];
        \dimensioning{2}{a}{d}{2.5}[$0.5$~m];
        \dimensioning{2}{d}{g}{2.5}[$0.5$~m];

        \notation{1}{g}{$M=0.5$ kg}[below right];
    \end{tikzpicture}
    \vspace{3mm}
    \caption{}
    \label{fig:illustration2}
\end{subfigure}\\
\begin{subfigure}{1\linewidth}
    \includegraphics[width=\linewidth]{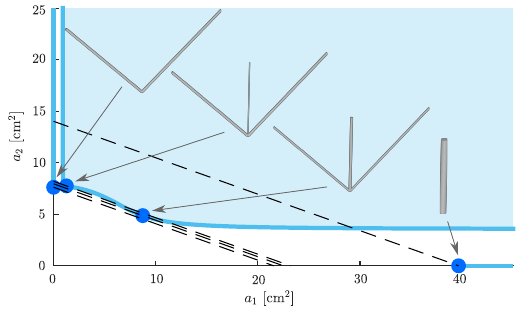}
    \caption{}
    \label{fig:fset}
\end{subfigure}
\caption{Illustrative problem: (a) discretization and boundary conditions, and (b) disconnected feasible set with four locally minimal solutions indicated with blue circles. The dashed lines denote the contour lines of the weight objective function {$16.80$, $17.54$, $18.08$ and $30.92$~kg}.}
\end{figure}

Following \cite{Yamada2015}, we consider a single load case problem with $n_{\mathrm{dof}} = 12$ degrees of freedom, which is the size of the polynomial matrix in the free-vibrations inequality \eqref{eq:opt_fv_po}, and set $\underline{\lambda}=10000\pi^2$~rad$^2$/s$^2$, which corresponds to the lowest resonance frequency being greater than or equal to $50$~Hz. For this setting, the feasible set is disconnected---see Figure~\ref{fig:fset}---which is due to the local vibration modes of the individual beams. This singularity phenomenon is inherent to the topology optimization problems of frame structures with free-vibration eigenvalue constraints and makes the solution particularly challenging. On the other hand, when solving a sizing problem with $\mathbf{a}>\mathbf{0}$, the singularity issue disappears, but the problem remains nonconvex nonetheless. This is also nicely visible in the considered illustrative problem. 
% Its admissible domain consists of three disconnected sets and there are at least four locally minimal points: $(39.64, 0)$~cm$^2$ of the weight $30.92$~kg, $(8.73, 4.87)$~cm$^2$ of the weight $17.54$~kg, $(1.26, 7.75)$~cm$^2$ of the weight $18.08$~kg, and $(0, 7.62)$~cm$^2$ of the weight $16.80$~kg that is the global solution.
Its admissible domain consists of three disconnected sets and there are at least four locally minimal points: $\mathbf{a}_1=(39.64, 0.0)$ of the weight $30.92$~kg, $\mathbf{a}_2=(8.73, 4.87)$ of the weight $17.54$~kg, $\mathbf{a}_3=(1.26, 7.75)$ of the weight $18.08$~kg, and $\mathbf{a}_4=(0.0, 7.62)$ of the weight $16.80$~kg that is the global solution.

Using Lemma \ref{lemma:initial_ub} with $\underline{\lambda}$ defined above, we obtain the point $(2.2, 2.0)\times 10^{-3}$ for which there must exist a $\delta$ such that $\tilde{\mathbf{a}}=\delta (2.2, 2.0)\times 10^{-3}$ is a feasible point for the original problem \eqref{eq:opt_orig}. Using bisection, we find that the smallest such $\delta^*$ is $2920$, which provides a feasible solution $\mathbf{a}_{\rm f} = (6.41, 5.94)$ of the weight $\overline{w}=18.11$~kg.

\subsubsection{Solution without weight constraint}

Subsequently, using the upper bound $\overline{w}$ computed above, we make the feasible set compact and scale the design variables following the developments in Section \ref{sec:msos}. Then we solve the Lasserre hierarchy of convex relaxations for the formulation \eqref{eq:opt_po} but without the {tightening the} weight constraint \eqref{eq:opt_weight_po}; see the numerical results in Table \ref{tab:illustrative}.\ %
Notice that the value of the bound already cuts off one of the singular local solutions.

\begin{figure}[!htbp]
    \centering
    \begin{subfigure}{0.5\linewidth}
        \raggedright
        \includegraphics[width=0.99\linewidth]{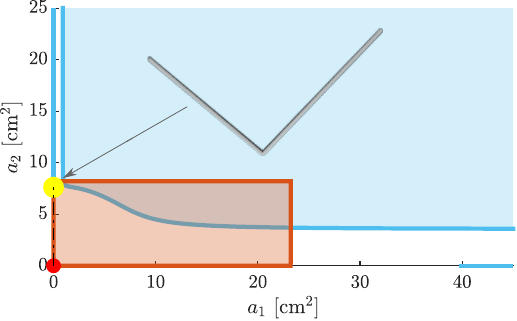}
        \caption{}
        \label{fig:rel1}
    \end{subfigure}%
    \begin{subfigure}{0.5\linewidth}
        \raggedleft
        \includegraphics[width=0.99\linewidth]{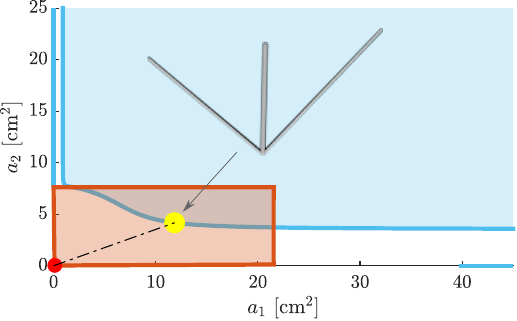}
        \caption{}
        \label{fig:rel2}
    \end{subfigure}\\
    \begin{subfigure}{0.5\linewidth}
        \raggedright
        \includegraphics[width=0.99\linewidth]{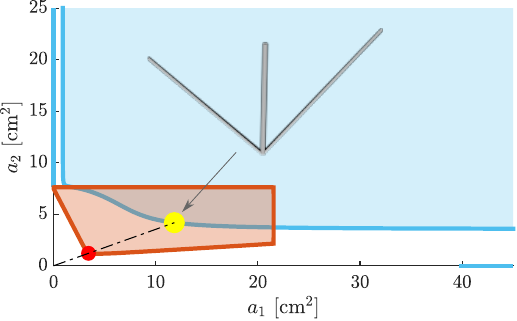}
        \caption{}
        \label{fig:rel3}
    \end{subfigure}%
    \begin{subfigure}{0.5\linewidth}
        \raggedleft
        \includegraphics[width=0.99\linewidth]{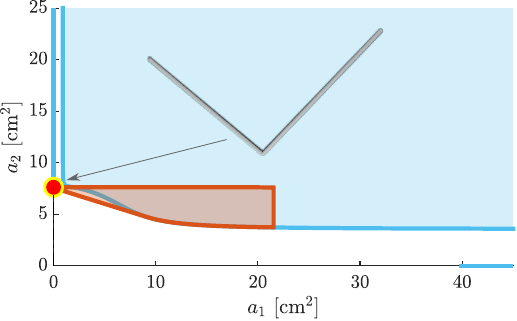}
        \caption{}
        \label{fig:rel4}
    \end{subfigure}
    \caption{Illustrative problem without the weight constraint: feasible sets of the (a) first, (b) second, (c) third, and (d) fourth relaxations shown in red, and the feasible set of the original optimization problem drawn in blue. Yellow circles denote upper bounds computed by projecting relaxation solutions (red circles) in the direction of the dashed dotted lines based on Proposition \ref{prop:feasibleub}.}
\end{figure}

In the lowest, first, relaxation, the feasible set of first-order moments almost matches the prescribed variable bounds; Fig.~\ref{fig:rel1}, which we obtained by finding an interior point within the feasible set of the first-order moments and optimizing the most distant feasible points in sampled directions from the interior point. However, the optimal solution to the first relaxation of \eqref{eq:opt_po}, provides a strictly positive lower bound objective function value $7.6\times 10^{-3}$~kg. The strict positivity of this lower bound follows from Proposition~\ref{th:ub}, as there is a non-structural mass $M$ present in this problem\footnote{{For $M=0$, the trivial optimum of $\mathbf{a}$ would be found and certified as globally optimal for $r=1$.}}. The associated upper bound $(0, 7.62)$ of the weight $16.80$~kg constructed based on Proposition \ref{prop:feasibleub} is, in fact, the global solution (currently not certified). However, because the upper bound has a weight lower than that of the original feasible point, we further adopt this new bound in the compactification and scaling of subsequent relaxations.

The second relaxation exhibits a similar feasible set of first-order moments, Fig.~\ref{fig:rel2}, yet provides a slightly improved lower bound $0.18$~kg, and the associated feasible upper bound design $(11.85, 4.18)$ has the weight of $18.46$~kg. In the third relaxation, we achieve a considerably improved lower bound of $5.33$~kg and a tighter feasible set of the first-order moments, Fig.~\ref{fig:rel3}, but an upper bound design very similar to the one obtained in the second relaxation: $(11.82, 4.18)$ of the weight $18.45$~kg. 

In the final, fourth relaxation, we are able to certify the global optimality of the design. The lower bound is evaluated as $16.80$~kg and the associated cross-section areas $(0, 7.62)$ are feasible for the original optimization problem (i.e., $\delta^*\approx1$ in \eqref{eq:bilevel-lower}). In fact, we have the absolute optimality gap $\varepsilon=-10^{-6}${. This negative value corresponds to the tolerance of the MOSEK optimizer and occurs only due to inexact numerical solution of the relaxation}. The feasible set of the final relaxation illustrated in Fig.~\ref{fig:rel4} shows that it corresponds to the convex hull of the original feasible set within the introduced variable bounds. This suggests that the quality of the relaxation depends on the value of the supplied bounds, which is consistent with what we observed in the computations.

\begin{table}[!htbp]
    \centering
    \begin{tabular}{rrrrrrr}
         $r$ & l.b.\ [kg] & u.b.\ [kg] & $\varepsilon$ [kg] & $\lvert \mathbf{y}\rvert$ & $n_\mathrm{c} \times m$  & $t$ [s]\\ \hline
         1 & $0.01$ & $16.80$ & $16.80$ & $5$ & $2\times1$, $1\times3$, $1\times12$ & $0.25$\\
         2 & $0.18$ & $18.46$ & $16.62$ & $14$ & $2\times3$, $1\times6$, $1\times36$ & $0.31$\\
         3 & $5.33$ & $18.45$ & $11.47$ & $27$ & $2\times6$, $1\times10$, $1\times72$ & $1.58$\\
         4 & $16.80$ & $16.80$ & $-10^{-6}$ & $44$ & $2\times10$, $1\times15$, $1\times120$ & $0.69$
    \end{tabular}
    \caption{Numerical results of the illustrative problem without the weight constraint \eqref{eq:opt_weight_po}: $r$ denotes the relaxation degree {(see Section \ref{sec:background})}, l.b.\ and u.b.\ abbreviate the lower and upper bounds, and $\varepsilon$ denotes the absolute gap between the best u.b.\ and current l.b. Furthermore, $\lvert\mathbf{y}\rvert$ represents the number of moment variables, $n_\mathrm{c}\times m$ indicates that there are $n_\mathrm{c}$ constraints of size $m\times m$, and $t$ provides the computational time.}
    \label{tab:illustrative}
\end{table}

\subsubsection{Solution with weight constraint} 

Next, we solve the same problem again, but now with an additional weight constraint \eqref{eq:opt_weight_po} with the bound value provided based on the best current upper bound. In the lowest relaxation of the moment-sum-of-squares hierarchy \eqref{eq:opt_po}, we reach the same lower-bound weight of $7.6 \times 10^{-3}$~kg as in the case without the weight constraint. Consequently, we also have the same feasible upper bound design $(0, 7.62)$. The feasible set of first-order moments is again similar to the previous case, with the difference having no effect on the relaxation solution; see Fig.~\ref{fig:rel1w}.

However, when solving the second relaxation, we already have an improved lower bound of $0.42$~kg, which is an improvement when compared to $0.18$~kg obtained without the weight constraint. The associated upper bound $(15.21, 3.91)$ has the weight of $20.48$~kg. The feasible set in Fig.~\ref{fig:rel2w} is now clearly tighter. Finally, we solve the third relaxation and reach the lower bound of $16.80$~kg, which is a significant increase when compared to $5.33$~kg reached without the weight constraint. The lower bound is now almost equal to the weight $16.80$~kg of the globally-optimal upper-bound design $(0, 7.62)$. The lower bounds do not increase with higher relaxations, suggesting that the third relaxation is optimal up to the accuracy of the solver. The feasible set in Fig.~\ref{fig:rel3w} contains a single point.

\begin{figure}[!htbp]
    \centering
    \begin{subfigure}{0.5\linewidth}
        \raggedright
        \includegraphics[width=0.99\linewidth]{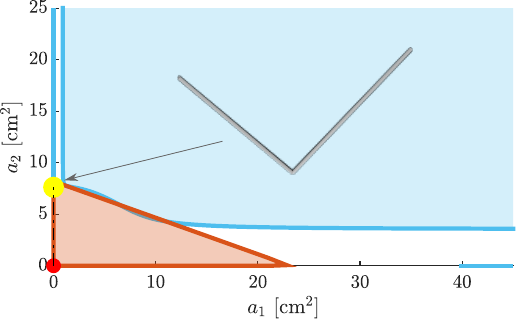}
        \caption{}
        \label{fig:rel1w}
    \end{subfigure}%
    \begin{subfigure}{0.5\linewidth}
        \raggedleft
        \includegraphics[width=0.99\linewidth]{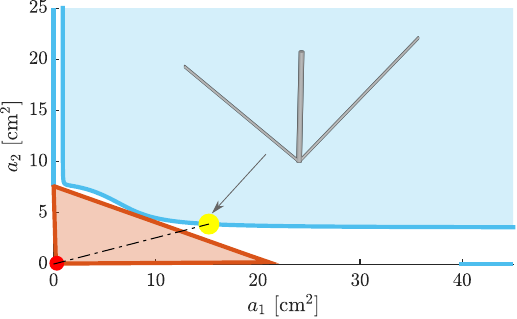}
        \caption{}
        \label{fig:rel2w}
    \end{subfigure}\\
    \centering\begin{subfigure}{0.5\linewidth}
        \raggedright
        \includegraphics[width=0.99\linewidth]{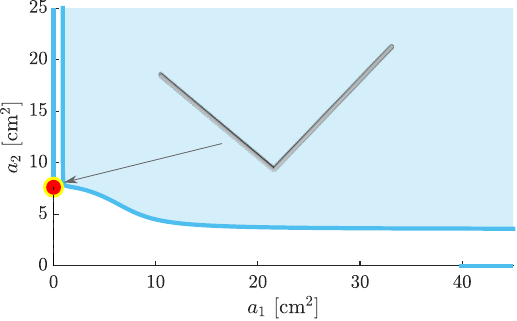}
        \caption{}
        \label{fig:rel3w}
    \end{subfigure}
    \caption{Illustrative problem with the weight constraint: feasible sets of the (a) first, (b) second, and (c) third relaxations shown in red, and the feasible set of the original optimization problem drawn in blue. Yellow circles denote upper bounds computed by projecting relaxation solutions (red circles) in the direction of the dashed dotted lines based on Proposition \ref{prop:feasibleub}.}
\end{figure}

\begin{table}[!htbp]
    \centering
    \begin{tabular}{rrrrrrr}
         $r$ & l.b.\ [kg] & u.b.\ [kg] & $\varepsilon$ [kg] & $\lvert \mathbf{y}\rvert$ & $n_\mathrm{c} \times m$  & $t$ [s]\\ \hline
         1 & $0.01$ & $16.80$ & $16.80$ & $5$ & $3\times1$, $1\times3$, $1\times12$ & $0.15$\\
         2 & $0.42$ & $20.48$ & $16.38$ & $14$ & $3\times3$, $1\times6$, $1\times36$ & $0.18$\\
         3 & $16.80$ & $16.80$ & $0.01$ & $27$ & $3\times6$, $1\times10$, $1\times72$ & $0.42$
    \end{tabular}
    \caption{Numerical results of the illustrative problem with the weight constraint \eqref{eq:opt_weight_po}: $r$ denotes the relaxation degree, l.b.\ and u.b.\ abbreviate the lower and upper bounds, and $\varepsilon$ denotes the absolute gap between the best u.b.\ and current l.b. Furthermore, $\lvert\mathbf{y}\rvert$ represents the number of moment variables, $n_\mathrm{c}\times m$ indicates that there are $n_\mathrm{c}$ constraints of size $m\times m$, and $t$ provides the computational time.}
    \label{tab:illustrativew}
\end{table}

\subsection{20-element cantilever beam}\label{sec:ex20}

As the second problem, we consider the structure of $10$ segments, each discretized with $2$ elements, as shown in Fig.~\ref{fig:20}. For this discretization, we have $n_\mathrm{dof}=42$ degrees of freedom, so $\mathbf{K}(\mathbf{a}), \mathbf{M}(\mathbf{a}) \in \mathbb{S}^{42}_{\succeq 0}$. The structure is made of a linear-elastic material with the Young modulus $E=68.9$~GPa and the density $\rho=2,770$~kg/m$^3$. Similarly to the previous example, we adopt circular cross sections and restrict elements in each segment to share the values of cross-sectional areas: $a_1=a_2$, $a_3=a_4$, $a_5=a_6$, $a_7=a_8$, $a_9=a_{10}$, $a_{11}=a_{12}$, $a_{13}=a_{14}$, $a_{15}=a_{16}$, $a_{17}=a_{18}$ and $a_{19}=a_{20}$. Consequently, we have ten optimized cross sections.

The left nodes of the structure are clamped and an upward load of size $10$~kN is acting at the bottom right corner of the structure. In addition, we incorporate a nonstructural mass $M=100$~kg in the center of the bottom edge. Under a single loading scenario, we enforce both the compliance and the fundamental free-vibrations eigenvalue constraints. In particular, we set $\overline{c} = 1$, which restricts the vertical displacement at the point and direction of the load to be at most $0.1$~mm, and $\underline{\lambda}=40000\pi^2$ rad$^2$/s$^2$, which requires the lowest resonance frequency to be at least $100$~Hz. 

\begin{figure}[!htbp]
\begin{subfigure}{0.5\linewidth}
\begin{tikzpicture}
    \scaling{2.5}
    \point{a}{0.000000}{0.000000}
    \point{b}{1.000000}{0.000000}
    \point{c}{2.000000}{0.000000}
    \point{d}{0.000000}{1.000000}
    \point{e}{1.000000}{1.000000}
    \point{f}{2.000000}{1.000000}
    \point{g}{0.500000}{0.000000}
    \point{h}{1.500000}{0.000000}
    \point{i1}{0.45}{0.45}
    \point{i2}{0.55}{0.55};
    \point{j}{0.500000}{0.500000}
    \point{k}{1.000000}{0.500000}
    \point{l1}{1.45}{0.45}
    \point{l2}{1.55}{0.55}
    \point{m}{1.500000}{0.500000}
    \point{n}{2.000000}{0.500000}
    \point{o}{0.500000}{1.000000}
    \point{p}{1.500000}{1.000000}
    \beam{2}{d}{o}
    \notation{4}{d}{o}[$1$]
    \beam{2}{e}{o}
    \notation{4}{e}{o}[$2$]
    \beam{2}{d}{j}
    \notation{4}{d}{j}[$3$]
    \beam{2}{b}{j}
    \notation{4}{b}{j}[$4$][0.65]
    \beam{2}{a}{i1}
    \notation{4}{a}{i1}[$5$]
    \beam{2}{e}{i2}
    \notation{4}{e}{i2}[$6$]
    \beam{2}{a}{g}
    \notation{4}{a}{g}[$7$]
    \beam{2}{b}{g}
    \notation{4}{b}{g}[$8$]
    \beam{2}{b}{k}
    \notation{4}{b}{k}[$9$]
    \beam{2}{e}{k}
    \notation{4}{k}{e}[$10$]
    \beam{2}{e}{p}
    \notation{4}{e}{p}[$11$]
    \beam{2}{f}{p}
    \notation{4}{f}{p}[$12$]
    \beam{2}{e}{m}
    \notation{4}{e}{m}[$13$]
    \beam{2}{c}{m}
    \notation{4}{c}{m}[$14$][0.65]
    \beam{2}{b}{l1}
    \notation{4}{b}{l1}[$15$]
    \beam{2}{f}{l2}
    \notation{4}{f}{l2}[$16$]
    \beam{2}{b}{h}
    \notation{4}{b}{h}[$17$]
    \beam{2}{c}{h}
    \notation{4}{c}{h}[$18$]
    \beam{2}{c}{n}
    \notation{4}{c}{n}[$19$]
    \beam{2}{f}{n}
    \notation{4}{n}{f}[$20$]
    \draw[ultra thick] (i2) arc[start angle=45, end angle=225, radius=0.175];
    \draw[ultra thick] (l2) arc[start angle=45, end angle=225, radius=0.175];
    \node[circle, fill,minimum size=10pt,inner sep=0pt, outer sep=0pt] at (b) {};
    \notation{3}{a}{b}[][0.5][];
    \notation{3}{b}{c}[][0.5][];
    \notation{3}{b}{e}[][0.5][];
    \notation{3}{c}{f}[][0.5][];
    \notation{3}{d}{e}[][0.5][];
    \notation{3}{e}{f}[][0.5][];
    \support{3}{a}[270];
    \support{3}{d}[270];
    \load{1}{c}[-90][0.6][0.05];
    \notation{1}{c}{$10$~kN}[below right];
    \notation{1}{b}{$100$~kg}[below=1mm];
    \dimensioning{1}{a}{g}{-0.95}[$0.5$~m];
    \dimensioning{1}{g}{b}{-0.95}[$0.5$~m];
    \dimensioning{1}{b}{h}{-0.95}[$0.5$~m];
    \dimensioning{1}{h}{c}{-0.95}[$0.5$~m];
    \dimensioning{2}{a}{j}{5.75}[$0.5$~m];
    \dimensioning{2}{j}{d}{5.75}[$0.5$~m];
    \end{tikzpicture}
    \caption{}
    \label{fig:20bc}
\end{subfigure}%
\hfill\raisebox{2.6cm}{\begin{minipage}{0.5\linewidth}
\begin{subfigure}{0.5\linewidth}
\includegraphics[height=1.75cm]{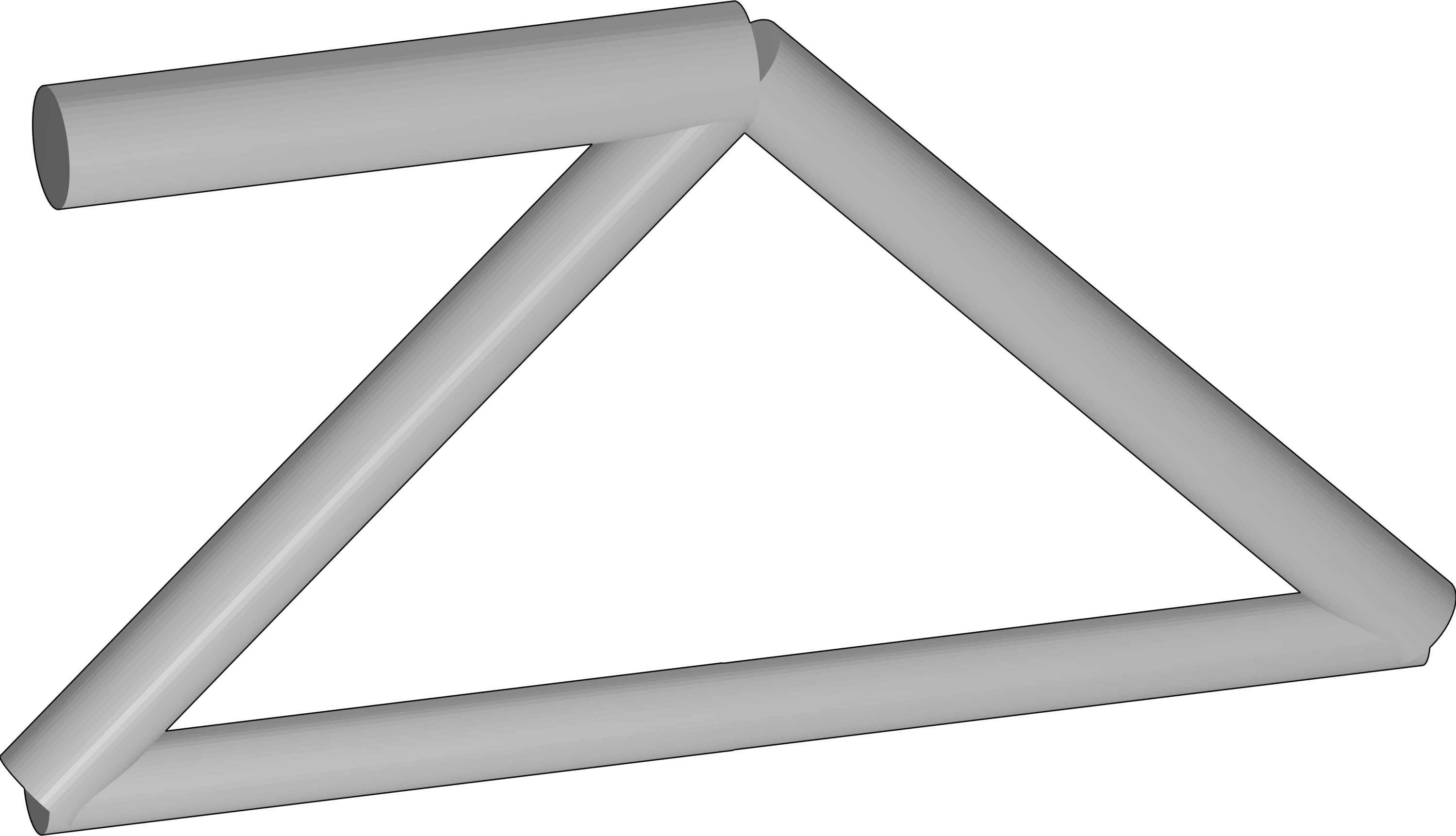}
\caption{}
\label{fig:20stat}
\end{subfigure}%
\hfill\begin{subfigure}{0.5\linewidth}
\includegraphics[height=1.75cm]{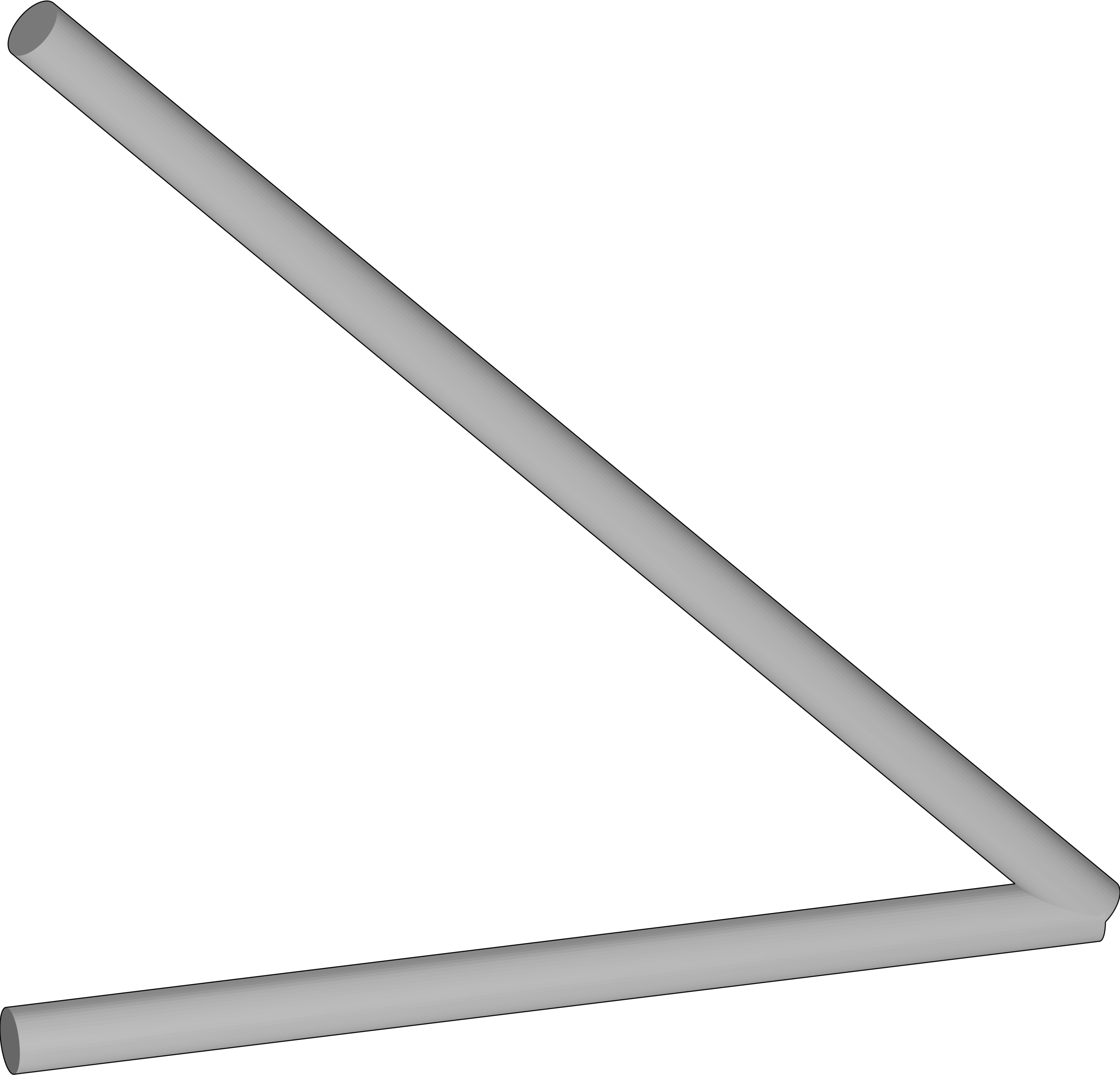}
\caption{}
\label{fig:20vib}
\end{subfigure}\\
\begin{subfigure}{0.5\linewidth}   
\includegraphics[height=1.75cm]{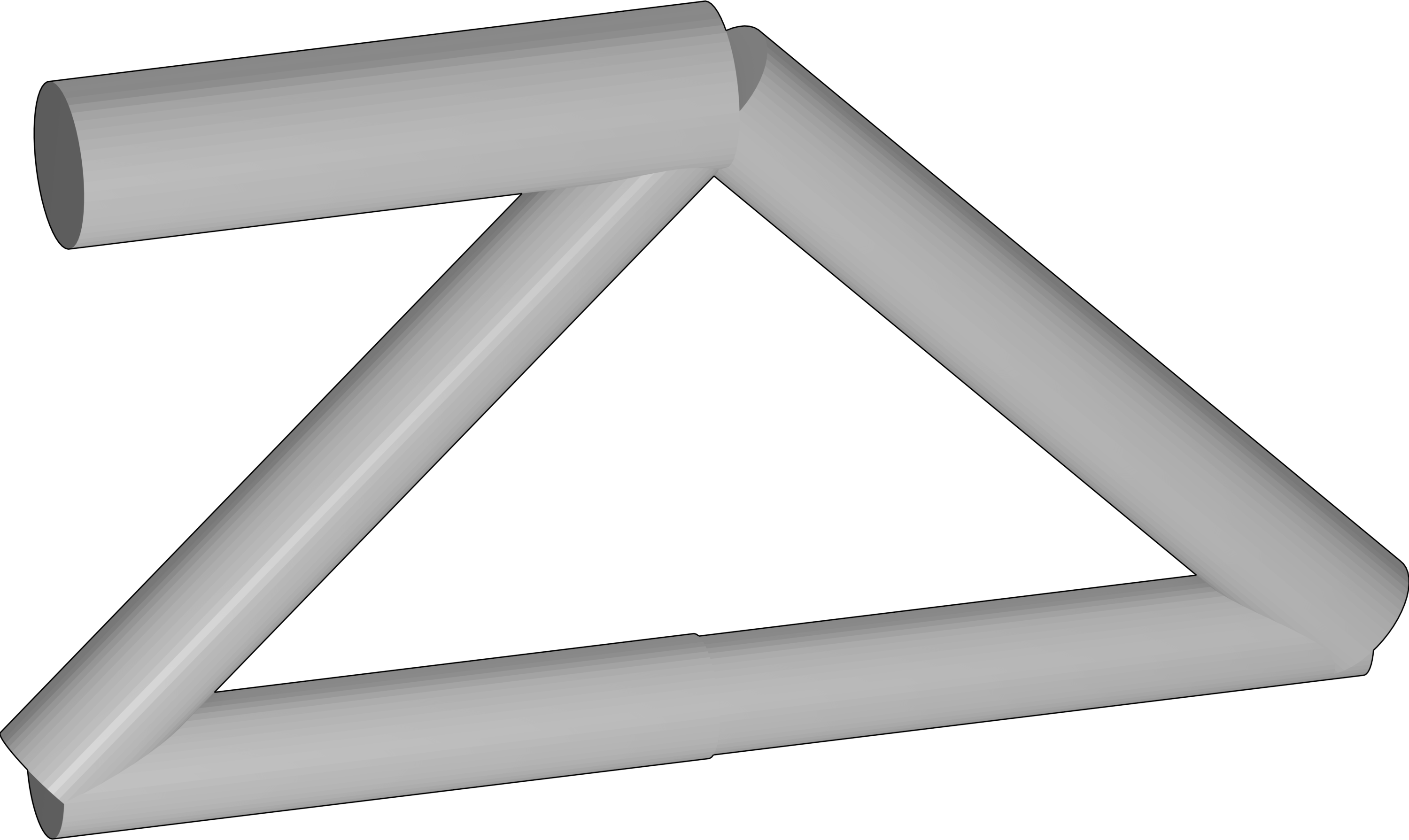}
\caption{}
\label{fig:20r1}
\end{subfigure}%
\hfill\begin{subfigure}{0.5\linewidth}
\includegraphics[height=1.75cm]{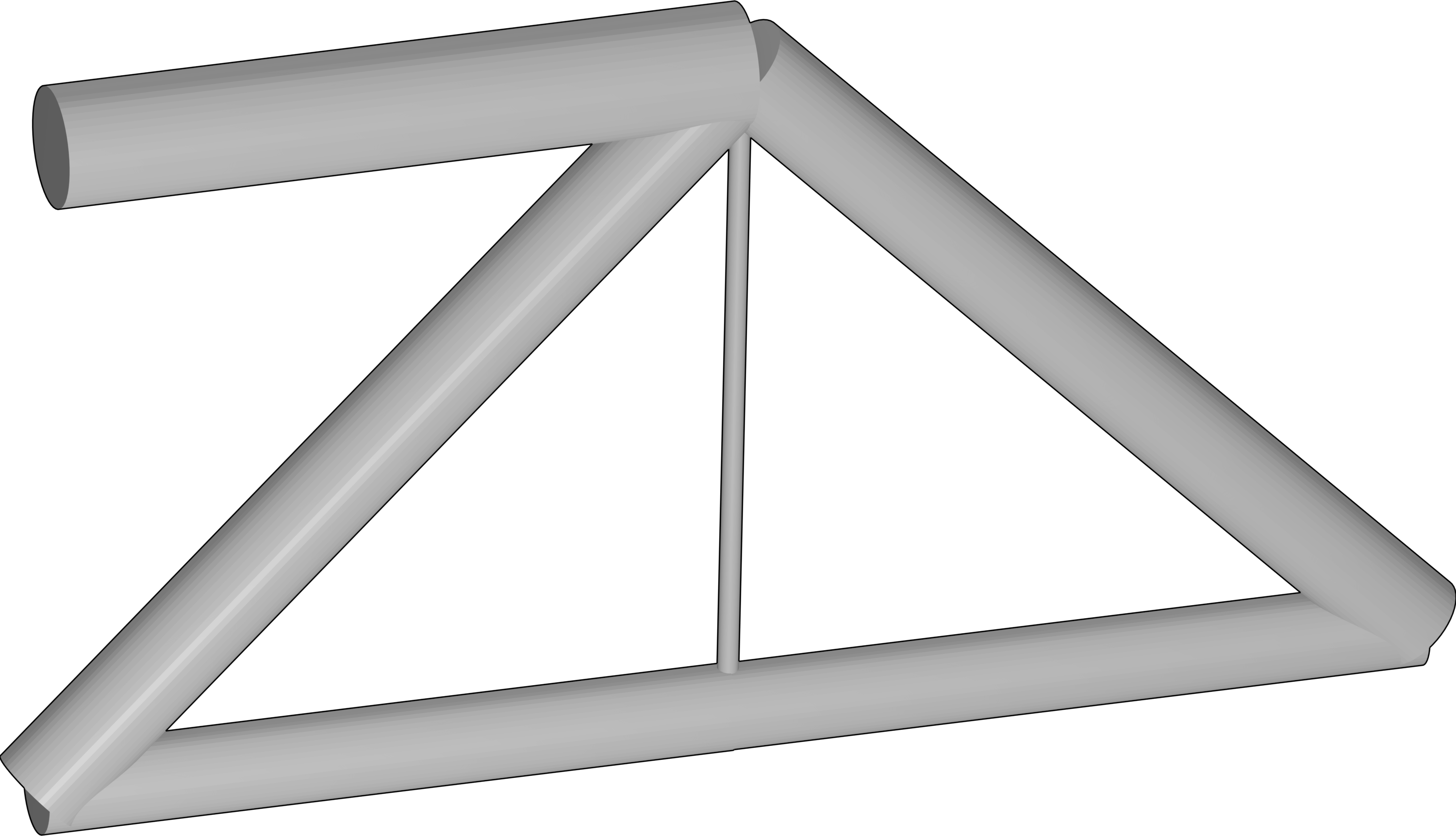}
\caption{}
\label{fig:20r2}
\end{subfigure}
\end{minipage}}
    \caption{20-element cantilever problem: (a) discretization and boundary conditions, optimal designs for the (b) compliance and (c) free-vibration constraints alone, and feasible upper-bound designs based on the (d) first and (e) second relaxation under compliance and free-vibration constraints. The design in (e) is globally optimal.}
    \label{fig:20}
\end{figure}

\subsubsection{MSOS hierarchy using canonical basis}
For illustrative purposes, we first investigate the solution to the problem without applying the free-vibration eigenvalue constraint. For this setting, the first relaxation provides a lower bound of $236.27$~kg and a feasible upper bound with the weight of $253.71$~kg. In the second relaxation, we reach an equality of bounds with the weight of $252.53$~kg, proving the global optimality of the design in Fig.~\ref{fig:20stat}.

In contrast, if we neglect the compliance constraint and optimize the structure for the free-vibration eigenvalue, the first relaxation provides a lower bound of $16.52$~kg and a feasible upper bound design weighting $17.88$~kg. The hierarchy converges again in the second relaxation, with the equality of bounds, the optimal weight of $17.84$~kg and the design shown in Fig.~\ref{fig:20vib}.

Finally, we solve the problem with both constraints applied simultaneously. Then, the initial feasible upper bound design based on Proposition \ref{th:feasibility_cond} weights $716.94$~kg. In the first relaxation, we reach a lower bound weight of $231.88$~kg and a feasible upper-bound design of the weight $504.56$~kg shown in Fig.~\ref{fig:20r1}. In the second relaxation, we achieve approximate equality of bounds at $254.47$~kg, proving the global $\varepsilon$-optimality of the design {weight} with a guarantee smaller than $5$~g, see Table~\ref{tab:20}. We note here that for the optimal design in Fig.~\ref{fig:20r2}, both the compliance and the free-vibration eigenvalue constraints are active.

\begin{table}[!htbp]
    \centering
    \begin{tabular}{rrrrrrr}
         $r$ & l.b.\ [kg] & u.b.\ [kg] & $\varepsilon$ [kg] & $\lvert \mathbf{y}\rvert$ & $n_\mathrm{c} \times m$  & $t$ [s]\\ \hline
         1 & $231.88$ & $504.56$ & $79.4$ & $65$ & $11\times1$, $1\times11$, $1\times42$, $1 \times 43$ & $0.2$\\
         2 & $254.47$ & $254.47$ & $5 \times 10^{-3}$ & $1000$ & $11\times11$, $1\times66$, $1\times462$, $1\times473$ & $38.6$
    \end{tabular}
    \caption{Numerical results of the $20$-element cantilever problem using the canonical basis: $r$ denotes the relaxation degree, l.b.\ and u.b.\ abbreviate the lower and upper bounds, and $\varepsilon$ denotes the absolute gap between the best u.b.\ and current l.b. Furthermore, $\lvert\mathbf{y}\rvert$ represents the number of moment variables, $n_\mathrm{c}\times m$ indicates that there are $n_\mathrm{c}$ constraints of size $m\times m$, and $t$ provides the computational time.}
    \label{tab:20}
\end{table}

\subsubsection{MSOS hiearchy using nonmixed term basis} Up to now, we have been solving the standard moment-sum-of-squares hierarchy as defined in Section~\ref{sec:background}, which is based on the moment matrices constructed using the canonical polynomial basis $\mathbf{b}_r(\mathbf{x})$; recall the program \eqref{eq:general_moment}. In our test experiments, we observed that the hierarchy also numerically converges if we remove all mixed terms from the canonical basis, resulting in the $\mathbf{b}_{\mathrm{NMT},r}(\mathbf{x})$ basis. Numerical convergence may be related to the fact that there are no mixed terms in the problem we are dealing with. Although we use this more scalable hierarchy to solve our problems globally, we do not have an asymptotic convergence proof for this case. 

Using this setting, the initial feasible upper bound and the first relaxation remain the same as in the standard hierarchy by construction. In the second relaxation, we obtain a smaller problem which is due to a smaller number of terms in the basis; see Table~\ref{tab:20nmt}. In this case, although the size of the problem was smaller, the time of solution was comparable. However, we were able to solve the second relaxation more accurately. In particular, we reached an approximate bound equality at the weight of $254.47$~kg with the guarantee of $1$~g only.

\begin{table}[!htbp]
    \centering
    \begin{tabular}{rrrrrrr}
         $r$ & l.b.\ [kg] & u.b.\ [kg] & $\varepsilon$ [kg] & $\lvert \mathbf{y}\rvert$ & $n_\mathrm{c} \times m$  & $t$ [s]\\ \hline
         1 & $231.88$ & $504.56$ & $79.4$ & $65$ & $11\times1$, $1\times11$, $1\times42$, $1 \times 43$ & $0.2$\\
         2 & $254.47$ & $254.47$ & $1 \times 10^{-3}$ & $790$ & $11\times11$, $1\times21$, $1\times462$, $1\times473$ & $36.7$
    \end{tabular}
    \caption{Numerical results of the $20$-element cantilever problem using the nonmixed-term basis: $r$ denotes the relaxation degree, l.b.\ and u.b.\ abbreviate the lower and upper bounds, and $\varepsilon$ denotes the absolute gap between the best u.b.\ and current l.b. Furthermore, $\lvert\mathbf{y}\rvert$ represents the number of moment variables, $n_\mathrm{c}\times m$ indicates that there are $n_\mathrm{c}$ constraints of size $m\times m$, and $t$ provides the computational time.}
    \label{tab:20nmt}
\end{table}

{
\subsection{Multiple-loadcase problem}

Next, we consider the problem of minimizing the weight of a structure shown in Fig.~\ref{fig:multilca} of the total length $0.8$~m. The structure consists of $4$ segments, each of them having a constant rectangular cross-section area of prescribed width $0.1$~m and optimized height. Discretizing each segment with two Euler-Bernoulli elements, we obtain $8$ elements and $4$ cross-sectional variables, i.e., $a_1 = a_2$, $a_3=a_4$, $a_5=a_6$, and $a_7=a_8$. All elements are made of linear elastic material with the Young modulus $E=210.0$~GPa and density $\rho=7,800$~kg/m$^3$. Further, we place a nonstructural mass $m=1$~kg at the middle of the structure.

We consider two load cases, differing in the kinematic boundary conditions and natural frequency constraints. In the first case, shown at the top part of Fig.~\ref{fig:multilca}, we apply simply-supported boundary conditions and require the eigenfrequency to be at least $300$~Hz (i.e., $\underline{\lambda}_1 = 600^2\pi^2$). For the second case, shown at the bottom part of Fig.~\ref{fig:multilca}, we clamp the left end and require the lowest nonzero eigenfrequency to be at least $200$~Hz (i.e., $\underline{\lambda}_2 = 400^2\pi^2$).

\begin{figure}[!htbp]
\begin{subfigure}{0.5\linewidth}
\begin{tikzpicture}
    \scaling{6.5}
    \point{a}{0.000000}{0}
    \point{b}{0.100000}{0}
    \point{c}{0.200000}{0}
    \point{d}{0.300000}{0}
    \point{e}{0.400000}{0}
    \point{f}{0.500000}{0}
    \point{g}{0.600000}{0}
    \point{h}{0.700000}{0}
    \point{i}{0.800000}{0}
    \beam{2}{a}{b}
    \notation{4}{a}{b}[$1$]
    \beam{2}{b}{c}
    \notation{4}{b}{c}[$2$]
    \beam{2}{c}{d}
    \notation{4}{c}{d}[$3$]
    \beam{2}{d}{e}
    \notation{4}{d}{e}[$4$]
    \beam{2}{e}{f}
    \notation{4}{e}{f}[$5$]
    \beam{2}{f}{g}
    \notation{4}{f}{g}[$6$]
    \beam{2}{g}{h}
    \notation{4}{g}{h}[$7$]
    \beam{2}{h}{i}
    \notation{4}{h}{i}[$8$]
    \node[circle, fill,minimum size=10pt,inner sep=0pt, outer sep=0pt] at (e) {};
    \support{3}{a}[270];
    \dimensioning{1}{a}{b}{-0.95}[$0.1$];
    \dimensioning{1}{b}{c}{-0.95}[$0.1$];
    \dimensioning{1}{c}{d}{-0.95}[$0.1$];
    \dimensioning{1}{d}{e}{-0.95}[$0.1$];
    \dimensioning{1}{e}{f}{-0.95}[$0.1$];
    \dimensioning{1}{f}{g}{-0.95}[$0.1$];
    \dimensioning{1}{g}{h}{-0.95}[$0.1$];
    \dimensioning{1}{h}{i}{-0.95}[$0.1$];

    \notation{1}{e}{$1$~kg}[above=7mm];

    \point{an}{0.000000}{0.25}
    \point{bn}{0.100000}{0.25}
    \point{cn}{0.200000}{0.25}
    \point{dn}{0.300000}{0.25}
    \point{en}{0.400000}{0.25}
    \point{fn}{0.500000}{0.25}
    \point{gn}{0.600000}{0.25}
    \point{hn}{0.700000}{0.25}
    \point{in}{0.800000}{0.25}
    \beam{2}{an}{bn}
    \notation{4}{an}{bn}[$1$]
    \beam{2}{bn}{cn}
    \notation{4}{bn}{cn}[$2$]
    \beam{2}{cn}{dn}
    \notation{4}{cn}{dn}[$3$]
    \beam{2}{dn}{en}
    \notation{4}{dn}{en}[$4$]
    \beam{2}{en}{fn}
    \notation{4}{en}{fn}[$5$]
    \beam{2}{fn}{gn}
    \notation{4}{fn}{gn}[$6$]
    \beam{2}{gn}{hn}
    \notation{4}{gn}{hn}[$7$]
    \beam{2}{hn}{in}
    \notation{4}{hn}{in}[$8$]
    \node[circle, fill,minimum size=10pt,inner sep=0pt, outer sep=0pt] at (en) {};
    \support{1}{an}[0];
    \support{2}{in}[0];
\end{tikzpicture}%
\caption{}
\label{fig:multilca}
\end{subfigure}
\hfill\begin{subfigure}{0.45\linewidth}
\includegraphics[width=\linewidth]{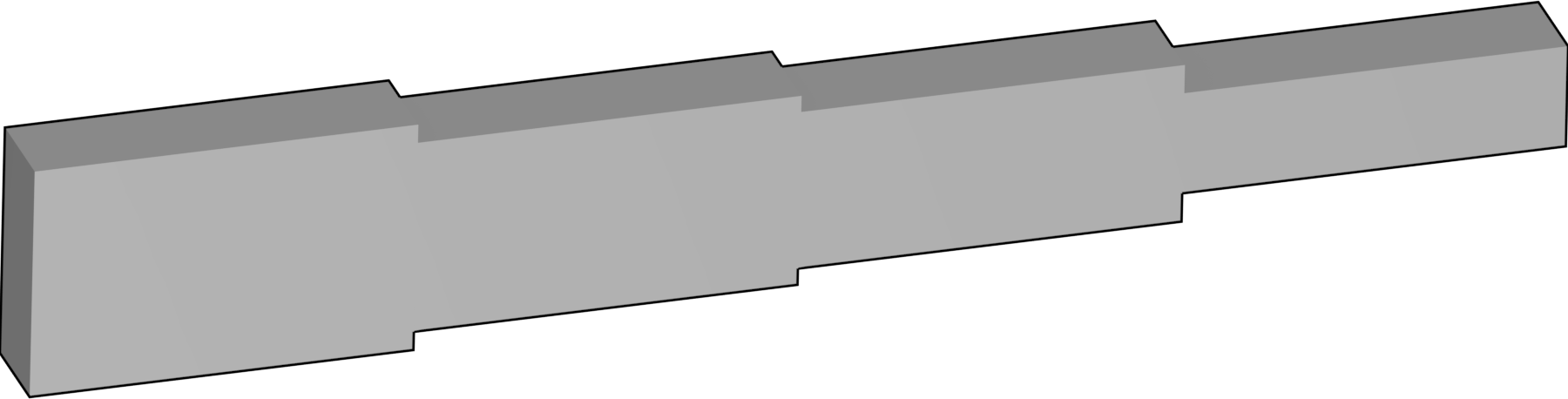}
\\
\vspace{8mm}
\caption{}
\label{fig:multilcb}
\end{subfigure}
\caption{Multiple-loadcase problem: (a) discretization and boundary conditions, and (b) globally-optimal design.}
\label{fig:multilc}
\end{figure}

Again, we employed the moment-sum-of-squares hierarchy with the nonmixed-term basis to solve the problem. In this case, the stiffness matrix is a degree-three polynomial function of the cross-section areas, see Appendix~\ref{app:polydep}, so we need to start with $r=2$. This relaxation provides us with the lower bound of $0.05$~kg and the feasible upper-bound design of the weight $37.50$~kg. In the third relaxation, we reach a lower bound of $2.75$~kg and the corresponding upper bound of $37.00$~kg. The fourth relaxation provides us with a lower bound of $29.22$~kg and an upper bound of $29.23$~kg, which is almost equal to the weight of the globally-optimal design shown in Fig.~\ref{fig:multilcb}. Finally, in the fifth relaxation, we reach a lower bound of $29.23$~kg and an upper bound of $29.23$~kg, proving the global $\varepsilon$-optimality of the design with a guarantee smaller than $3$~g; see Table~\ref{tab:multilc}. In this case, both the eigenvalue constraints are active at the optimal design.

\begin{table}[!htbp]
    \centering
    \setlength{\tabcolsep}{5pt}
    \begin{tabular}{rrrrrrr}
         $r$ & l.b.\ [kg] & u.b.\ [kg] & $\varepsilon$ [kg] & $\lvert \mathbf{y}\rvert$ & $n_\mathrm{c} \times m$  & $t$ [s]\\ \hline
         2 & $0.05$ & $37.50$ & $37.44$ & $68$ & $5\times5$, $1\times9$, $2\times24$ & $0.09$\\
         3 & $2.75$ & $37.00$ & $34.25$ & $146$ & $5\times9$, $1\times13$, $2\times120$ & $1.49$\\
         4 & $29.22$ & $29.23$ & $8.9\times10^{-3}$ & $292$ & $5\times13$, $1\times17$, $2\times216$ & $4.40$\\
         5 & $29.23$ & $29.23$ & $2.9\times10^{-3}$ & $514$ & $5\times17$, $1\times21$, $2\times312$ & $13.33$
    \end{tabular}
    \caption{Numerical results of the multiple-load case problem using the nonmixed-term basis: $r$ denotes the relaxation degree, l.b.\ and u.b.\ abbreviate the lower and upper bounds, and $\varepsilon$ denotes the absolute gap between the best u.b.\ and current l.b. Furthermore, $\lvert\mathbf{y}\rvert$ represents the number of moment variables, $n_\mathrm{c}\times m$ indicates that there are $n_\mathrm{c}$ constraints of size $m\times m$, and $t$ provides the computational time.}
    \label{tab:multilc}
\end{table}
}

\subsection{52-element problem}\label{sec:ex52}

As the last example, we investigate the minimization of the weight of the structure shown in Fig.~\ref{fig:52bc}. {Similarly to the previous example}, we adopt here the nonmixed-term basis{, as it} would be {im}possible {to obtain the global solution} using the {available} hardware otherwise. The problem consists of $26$ frame segments, each of them having a constant cross-section area and discretized with two Euler-Bernoulli frame elements. Consequently, we have $52$ elements and $26$ cross-sectional variables. All elements have square cross sections made of linear elastic material with Young modulus $E=68.9$~GPa and density $\rho=2,770$~kg/m$^3$.

We again have a single loading scenario: the left nodes of the structure are clamped, an upward vertical force of $10$~kN acts at the bottom right node, and nonstructural masses of $100$~kg are placed in the middle of the bottom and top boundary of the domain. For this setting, we have $96$ degrees of freedom, so that $\mathbf{K}(\mathbf{a}), \mathbf{M}(\mathbf{a}) \in \mathbb{S}^{96}_{\succeq 0}$.

For optimization constraints, we impose the compliance upper bound $\overline{c} = 1$, which limits the displacement at the point and direction of the load to $0.1$~mm, and a lower bound of the fundamental free-vibration eigenvalue $\underline{\lambda} = 40000\pi^2$~rad/s$^2$, which requires the lowest nonzero eigenfrequency to be at least $100$~Hz.

When using the compliance constraint alone, the first relaxation provides us with the lower bound of $35.17$~kg and the corresponding feasible upper-bound design of the weight $36.02$~kg. In the second relaxation, we reach a bound equality at the weight of $35.967$~kg, see Fig.~\ref{fig:52w} for the optimal design. On the other hand, imposing the free-vibration constraint alone leads to the lower bound of $7.79$~kg in the first relaxation, and the corresponding upper bound has the weight of $7.97$~kg. In the second relaxation, we again reach a bound equality at the weight of $7.97$~kg and the design shown in Fig.~\ref{fig:52fv}.

\begin{figure}[!htbp]
    \begin{subfigure}{0.45\linewidth}
    \begin{tikzpicture}
        \scaling{3.5}
        \point{n1}{0.000000}{0.000000}
        \point{n2}{0.000000}{0.500000}
        \point{n3}{0.000000}{1.000000}
        \point{n4}{0.500000}{0.000000}
        \point{n5}{0.500000}{0.500000}
        \point{n6}{0.500000}{1.000000}
        \point{n7}{1.000000}{0.000000}
        \point{n8}{1.000000}{0.500000}
        \point{n9}{1.000000}{1.000000}
        \point{n10}{0.250000}{0.000000}
        \point{n11}{0.250000}{0.250000}
        \point{n12}{0.250000}{0.500000}
        \point{n13}{0.500000}{0.250000}
        \point{n14}{0.250000}{0.250000}
        \point{n15}{0.250000}{0.500000}
        \point{n16}{0.250000}{0.750000}
        \point{n17}{0.500000}{0.250000}
        \point{n18}{0.500000}{0.750000}
        \point{n19}{0.250000}{0.500000}
        \point{n20}{0.250000}{0.750000}
        \point{n21}{0.250000}{1.000000}
        \point{n22}{0.500000}{0.750000}
        \point{n23}{0.500000}{0.250000}
        \point{n24}{0.750000}{0.000000}
        \point{n25}{0.750000}{0.250000}
        \point{n26}{0.750000}{0.500000}
        \point{n27}{0.500000}{0.750000}
        \point{n28}{0.750000}{0.250000}
        \point{n29}{0.750000}{0.500000}
        \point{n30}{0.750000}{0.750000}
        \point{n31}{0.750000}{0.500000}
        \point{n32}{0.750000}{0.750000}
        \point{n33}{0.750000}{1.000000}
        \point{n34}{1.000000}{0.250000}
        \point{n35}{1.000000}{0.750000}
        \beam{2}{n1}{n10}
        \beam{2}{n4}{n10}
        \beam{2}{n1}{n11}
        \beam{2}{n5}{n11}
        \beam{2}{n1}{n12}
        \beam{2}{n6}{n12}
        \beam{2}{n1}{n13}
        \beam{2}{n8}{n13}
        \beam{2}{n2}{n14}
        \beam{2}{n4}{n14}
        \beam{2}{n2}{n15}
        \beam{2}{n5}{n15}
        \beam{2}{n2}{n16}
        \beam{2}{n6}{n16}
        \beam{2}{n2}{n17}
        \beam{2}{n7}{n17}
        \beam{2}{n2}{n18}
        \beam{2}{n9}{n18}
        \beam{2}{n3}{n19}
        \beam{2}{n4}{n19}
        \beam{2}{n3}{n20}
        \beam{2}{n5}{n20}
        \beam{2}{n3}{n21}
        \beam{2}{n6}{n21}
        \beam{2}{n3}{n22}
        \beam{2}{n8}{n22}
        \beam{2}{n4}{n23}
        \beam{2}{n5}{n23}
        \beam{2}{n4}{n24}
        \beam{2}{n7}{n24}
        \beam{2}{n4}{n25}
        \beam{2}{n8}{n25}
        \beam{2}{n4}{n26}
        \beam{2}{n9}{n26}
        \beam{2}{n5}{n27}
        \beam{2}{n6}{n27}
        \beam{2}{n5}{n28}
        \beam{2}{n7}{n28}
        \beam{2}{n5}{n29}
        \beam{2}{n8}{n29}
        \beam{2}{n5}{n30}
        \beam{2}{n9}{n30}
        \beam{2}{n6}{n31}
        \beam{2}{n7}{n31}
        \beam{2}{n6}{n32}
        \beam{2}{n8}{n32}
        \beam{2}{n6}{n33}
        \beam{2}{n9}{n33}
        \beam{2}{n7}{n34}
        \beam{2}{n8}{n34}
        \beam{2}{n8}{n35}
        \beam{2}{n9}{n35}
        \node[circle, fill,minimum size=10pt,inner sep=0pt, outer sep=0pt] at (n4) {};
        \node[circle, fill,minimum size=10pt,inner sep=0pt, outer sep=0pt] at (n6) {};
        \support{3}{n1}[270];
        \support{3}{n2}[270];
        \support{3}{n3}[270];
        \load{1}{n7}[90][-0.6][-0.1];
        \notation{1}{n7}{$10$~kN}[below right];
        \notation{1}{n4}{$100$~kg}[below=1mm];
        \notation{1}{n6}{$100$~kg}[above=1mm];
        \dimensioning{1}{n1}{n4}{-0.95}[$0.5$~m];
        \dimensioning{1}{n4}{n7}{-0.95}[$0.5$~m];
        \dimensioning{2}{n1}{n2}{4.25}[$0.5$~m];
        \dimensioning{2}{n2}{n3}{4.25}[$0.5$~m];
    \end{tikzpicture}
    \caption{}
    \label{fig:52bc}
    \end{subfigure}%
    \raisebox{2.825cm}{\begin{minipage}{0.55\linewidth}
        \centering%
    \begin{subfigure}{0.3\linewidth}
        \includegraphics[height=0.9\linewidth]{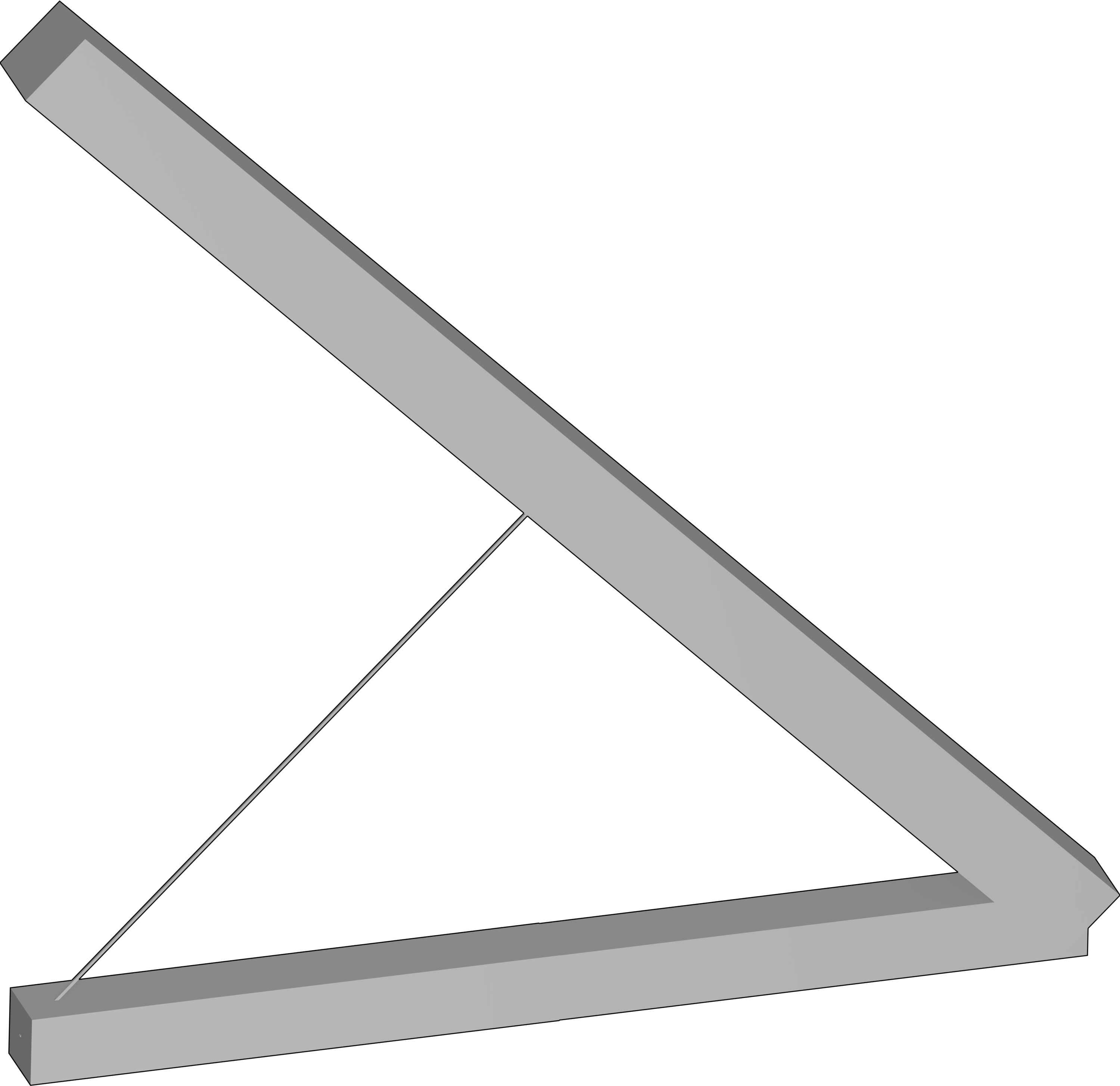}
        \caption{}
        \label{fig:52w}
    \end{subfigure}%
    \hspace{0.03\linewidth}\begin{subfigure}{0.3\linewidth}
        \includegraphics[height=0.9\linewidth]{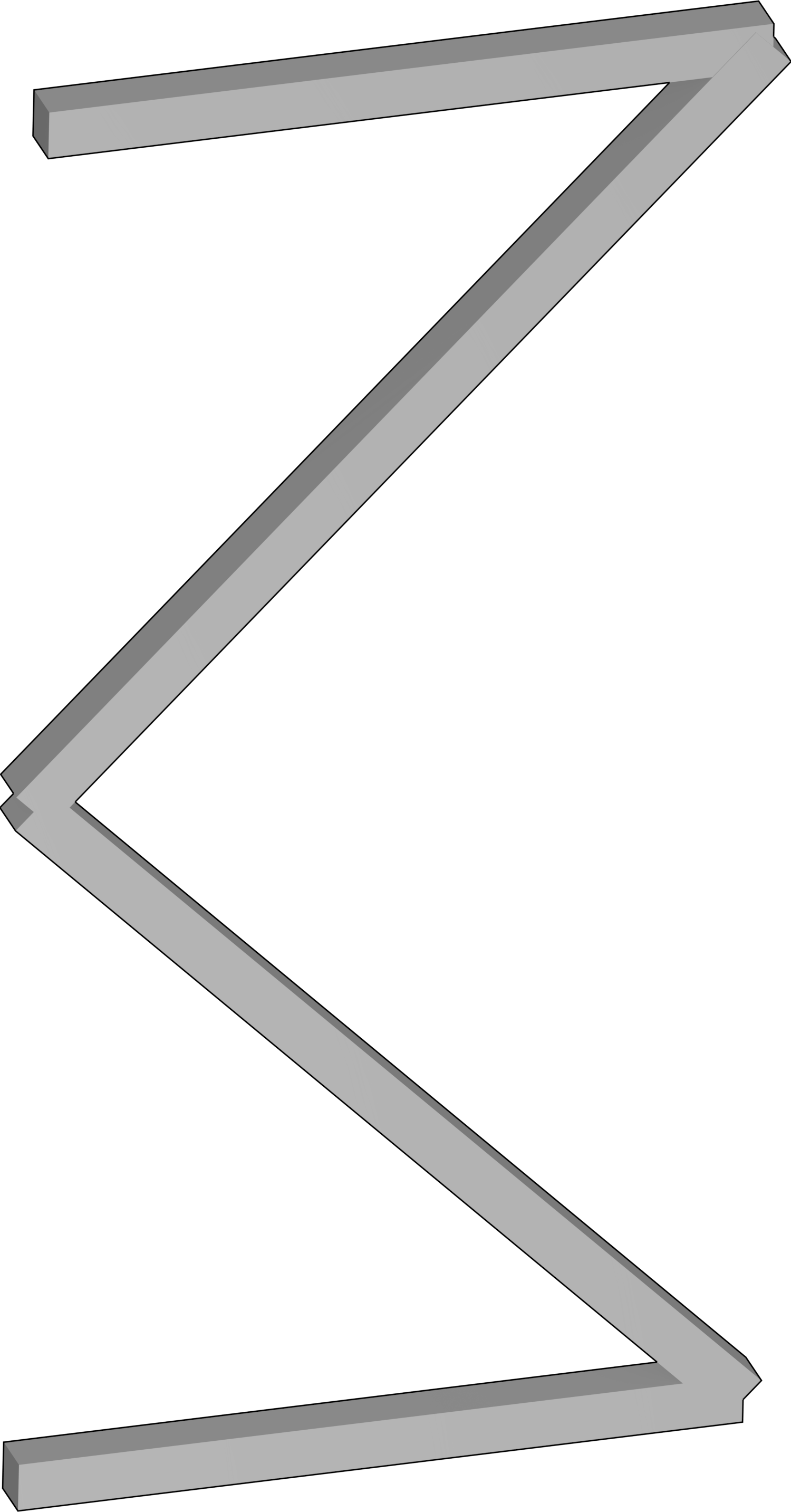}
        \caption{}
        \label{fig:52fv}
    \end{subfigure}\\
    \begin{subfigure}{0.3\linewidth}
        \includegraphics[height=0.9\linewidth]{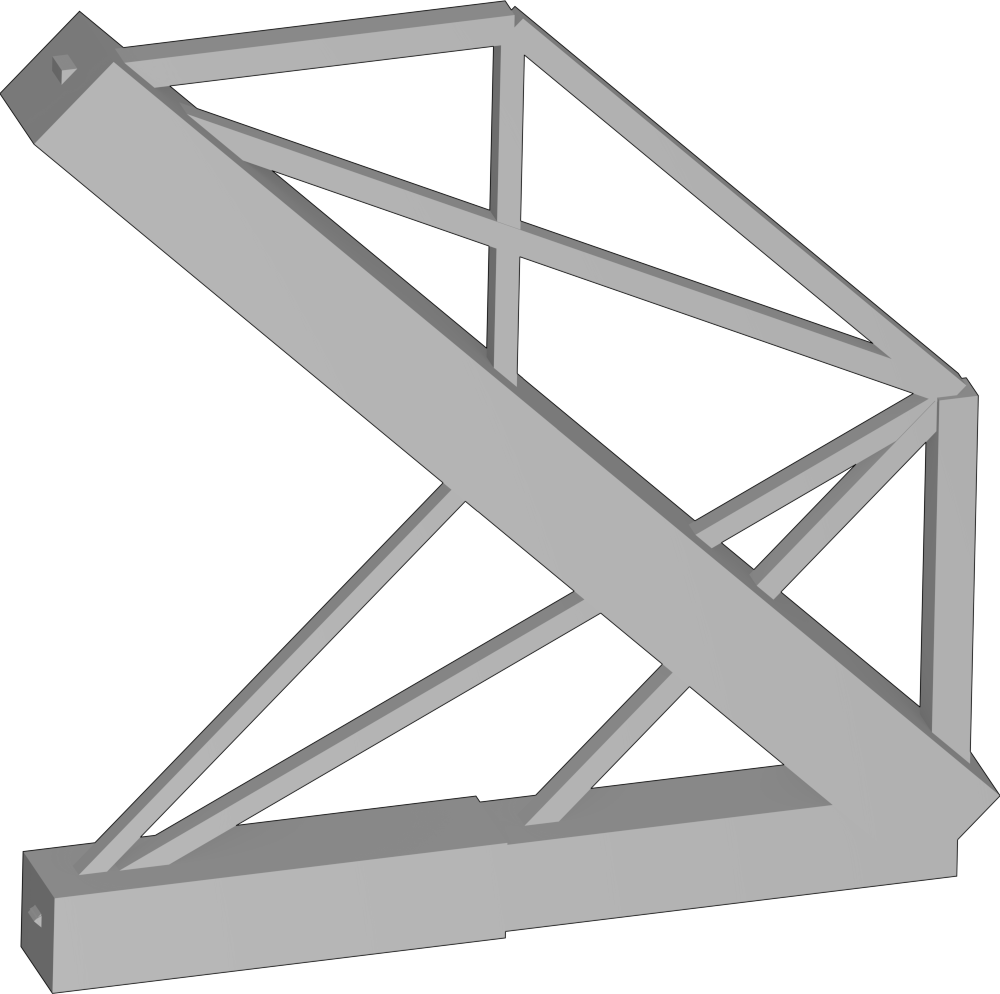}
        \caption{}
        \label{fig:52r1}
    \end{subfigure}%
    \hfill\begin{subfigure}{0.3\linewidth}
        \includegraphics[height=0.9\linewidth]{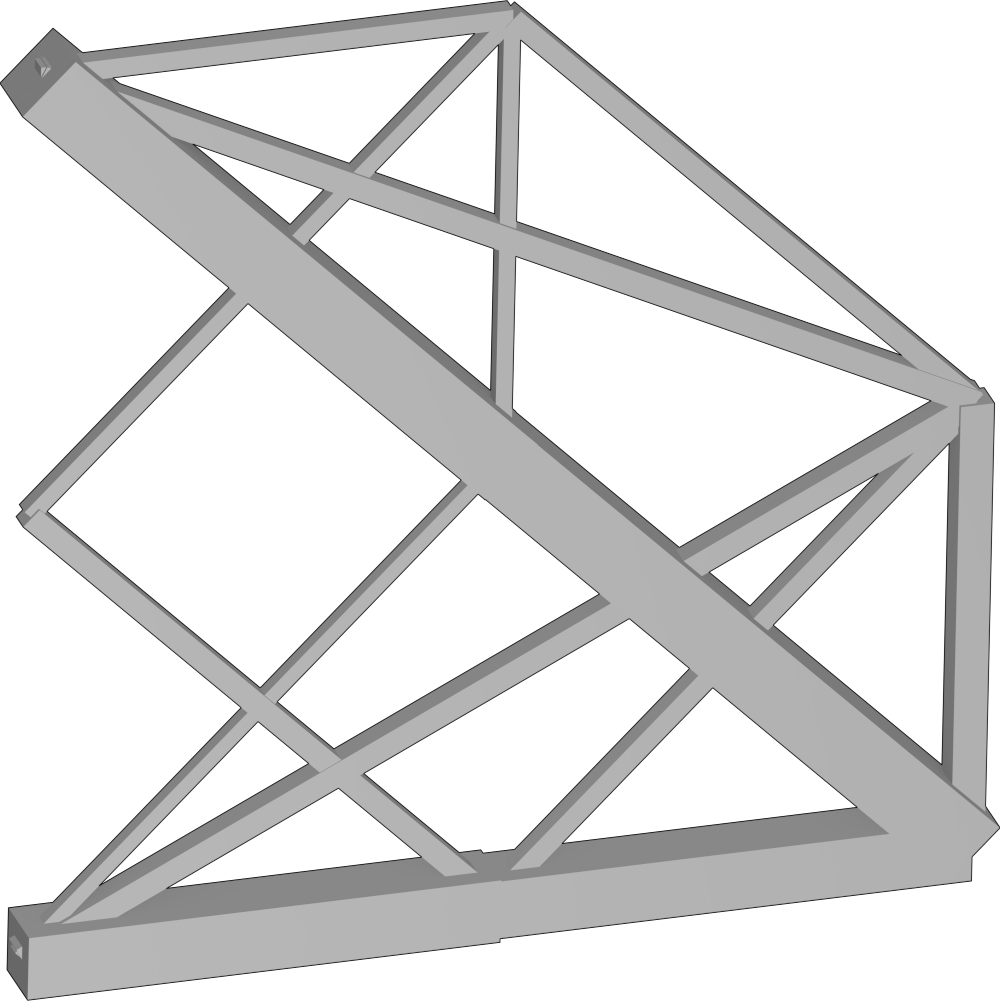}
        \caption{}
        \label{fig:52r2}
    \end{subfigure}%
    \hfill\begin{subfigure}{0.3\linewidth}
        \includegraphics[height=0.9\linewidth]{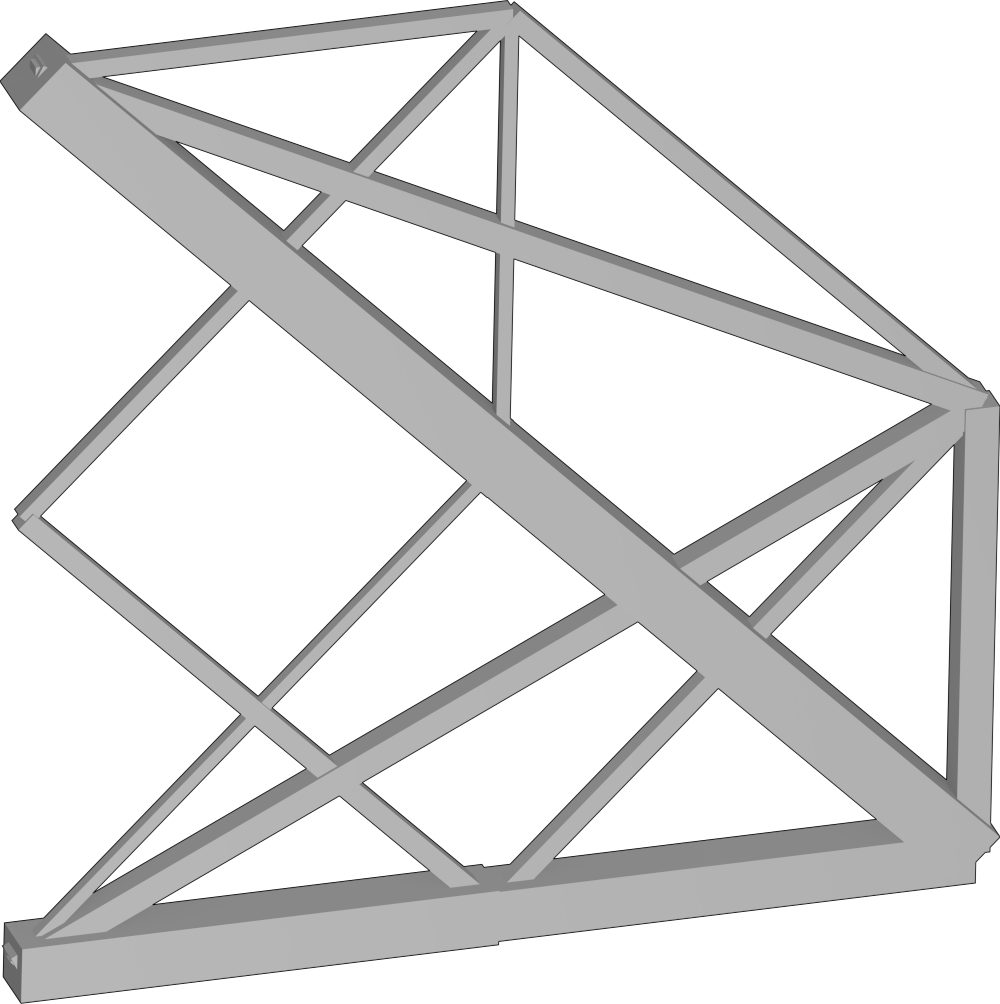}
        \caption{}
        \label{fig:52r3}
    \end{subfigure}
    \end{minipage}}
    \caption{52-element problem: (a) discretization and boundary conditions, optimal
designs for the (b) compliance and (c) free-vibration constraints alone, and feasible upper-bound designs based on the (d) first, (e) second, and (f) third relaxation under compliance and free-vibration constraints. The design in (f) is globally optimal.}
    \label{fig:52}
\end{figure}

Finally, we solve the optimization problem while accounting for both constraints. First, we find an initial feasible upper bound design based on Proposition \ref{th:feasibility_cond}, which has the weight of $238.51$~kg. After using this bound to make the feasible space of the variables compact, we adopt the formulation \eqref{eq:opt_po} and solve the moment-sum-of-squares hierarchy of the relaxations. In the first, we reach a lower bound of $34.23$~kg. The associated feasible upper-bound design has the weight of $78.16$~kg, see Fig.~\ref{fig:52r1}, which makes the optimality gap $43.93$~kg. The second relaxation provides us with the lower bound of $37.65$~kg, which is almost equal to the optimal weight, but the corresponding upper bound design in Fig.~\ref{fig:52r2} weights $44.04$~kg, so $\varepsilon=6.39$~kg. Consequently, we have to continue with the hierarchy. At the third relaxation, finite convergence occurs: the lower and upper bounds have the same weight of $37.67$~kg, so that the design in Fig.~\ref{fig:52r3} is globally optimal.

Although the global solution has been reached, the computational time needed for the solution to the third relaxation was significant; see Table \ref{tab:52nmt}. Moreover, we would not be able to solve the third relaxation while relying on the canonical basis.

\begin{table}[!htbp]
    \centering
    \setlength{\tabcolsep}{5pt}
    \begin{tabular}{rrrrrrr}
         $r$ & l.b.\ [kg] & u.b.\ [kg] & $\varepsilon$ [kg] & $\lvert \mathbf{y}\rvert$ & $n_\mathrm{c} \times m$  & $t$ [s]\\ \hline
         1 & $34.23$ & $78.16$ & $43.93$ & $377$ & $27\times1$, $1\times27$, $1\times96$, $1\times97$ & $0.4$\\
         2 & $37.65$ & $44.04$ & $6.39$ & $12454$ & $27\times27$, $1\times53$, $1\times2592$, $1\times2619$ & $2839$\\
         3 & $37.67$ & $37.67$ & $0$ & $25181$ & $27\times53$, $1\times79$, $1\times5088$, $1\times5141$ & $23765$
    \end{tabular}
    \caption{Numerical results of the $52$-element cantilever problem using the nonmixed-term basis: $r$ denotes the relaxation degree, l.b.\ and u.b.\ abbreviate the lower and upper bounds, and $\varepsilon$ denotes the absolute gap between the best u.b.\ and current l.b. Furthermore, $\lvert\mathbf{y}\rvert$ represents the number of moment variables, $n_\mathrm{c}\times m$ indicates that there are $n_\mathrm{c}$ constraints of size $m\times m$, and $t$ provides the computational time.}
    \label{tab:52nmt}
\end{table}

\section{Conclusion}

In this contribution, we have introduced a methodology for globally optimizing frame structures subject to fundamental free-vibration eigenvalue and compliance constraints using the moment-sum-of-squares hierarchy. This work represents a significant extension of the truss setting of Achtziger and Ko\v{c}vara~\cite{Achtziger2008} to the nonconvex setting of the frame elements, while simultaneously extending the work of Tyburec~et al.~\cite{Tyburec2023} to free vibrations. To the best of our knowledge, this is the first time such problems have been solved globally with convergence guarantees. 

Starting with a nonlinear semidefinite programming formulation of the optimization problem, we introduced its bilevel variant: searching for the minimal scaling of fixed cross-sectional area ratios in the lower-level problem and optimizing over the ratios in the upper level. We showed that the lower-level problem is quasiconvex in the scaling parameter and {has a non-empty interior} if and only if i) the cross-section area ratios are statically admissible, i.e., the force vector is in the range space of the stiffness matrix for given cross-section ratios,  ii) the cross-section area ratios are mass admissible, i.e., the mass matrix is in the range space of the stiffness matrix for the cross-section ratios, and iii) {certain} membrane-only eigenvalues are greater than or equal to the free-vibration eigenvalue lower bounds. For any such cross-section ratios, it is possible to find a feasible point (upper bound {design}) to the original nonlinear semidefinite program by finding a globally minimal scaling of the cross-section ratios using bisection.

Because these conditions on the cross-section ratios are semidefinite representable, a feasible point to the original problem can be constructed based on a solution to a linear semidefinite program. We adopted this bound to make the feasible set of the original nonlinear semidefinite program compact and proved that the Archimedean assumption needed for the convergence of the moment-sum-of-squares hierarchy is then satisfied. Consequently, it is possible to use the hierarchy of convex relaxations for a global solution of the free-vibration problems.

When solving the hierarchy of convex relaxations, we naturally obtain a monotonic sequence of lower bounds. In addition, we showed that the relaxed cross-sectional areas, constructed on the basis of the first-order moments, also satisfy the conditions for feasibility of the lower-level problem in the bilevel reformulation. Consequently, we constructed feasible upper bounds in each relaxation, tightening the feasible set, and also assess{ing} relaxation quality based on the gap between the lower and upper bounds. With this, we obtained a simple sufficient condition of global $\varepsilon$-optimality of the upper bounds, and further proved that this condition is tight in the limit under the assumption of the set of the global minimizers being convex.

We illustrated the theoretical results using {four} numerical examples. In the first, academic problem, we visualized the strong singularity phenomenon of frame structures under free-vibration eigenvalue constraints, and also assessed the effect of the feasible set tightening by an additional weight constraint. The second example is a middle-sized problem that can still be solved using the standard mSOS hierarchy globally. In this problem, we showed that it is possible to solve the problem globally even when relying on a nonmixed-term polynomial basis. {The third example illustrated applicability of the method to problems with multi-loadcase problems with varying kinematic boundary conditions.} The final example revealed the scalability of the approach when relying on the nonmixed-term basis.

We conclude this article with an outlook for related future works. These include extension of the procedure to the problem of harmonic oscillations \cite{ma2024truss}, investigation of the convergence of the hierarchy based on the nonmixed term basis, {exploiting problem symmetries, studying the influence of mesh refinement on the global solutions}, and solving small continuum topology optimization problems.

%\section*{Acknowledgement} The authors acknowledge the support of the . %This research was co-funded by the European Union under the project Robotics and Advanced Industrial Production (reg. no. CZ.02.01.01/00/22\_008/0004590).

\appendix
{
\section{Stiffness and mass matrices of Euler-Bernoulli beam elements}\label{app:polydep}

This appendix explains how the polynomial representation of the stiffness matrix and the linear representation of the mass matrix arise for prismatic Euler-Bernoulli frame elements in the simplest case of the cross-section scaling.

\subsection{Stiffness Matrix}

For an Euler-Bernoulli frame element $e$ with the cross-section area $a_e$, the stiffness matrix consists from membrane (m) and bending (b) contributions
$$\mathbf{K}_{e}(a_e) = \mathbf{L}_e^\mathrm{T} \mathbf{T}_e^\mathrm{T} \left[\mathbf{K}_{e}^{\text{m}}(a_e) + \mathbf{K}_{e}^{\text{b}}(a_e)\right] \mathbf{T}_e \mathbf{L}_e,$$
in which $\mathbf{T}_e$ is the transformation and $\mathbf{L}_e$ the localizing matrix \cite{Fish2007}.

The axial contribution depends linearly on the cross-section area:
$$\mathbf{K}_{e}^{\text{m}}(a_e) = E_e a_e \int_{0}^{\ell_e} \left(\mathbf{B}_e^\mathrm{m}(x)\right)^\mathrm{T} \mathbf{B}_e^\mathrm{m}(x) \, \mathrm{d}x,$$
where $E_e$ is Young's modulus, $\ell_e$ is the element length, and $\mathbf{B}^\mathrm{m}_e(x)$ is the membrane strain-displacement matrix. This gives rise to the first term $a_e \mathbf{K}_{j,e}^{\langle 1\rangle}$ in \eqref{eq:stiffness}.

The bending contribution depends on the elementwise constant second moment of area $I_e = \int_{A_e} z^2 \, \mathrm{d}A$ of the cross-section, where $A_e$ denotes the cross-section domain and $z$ is the distance from the neutral axis in the direction of the bending load. The bending stiffness matrix then becomes:
$$\mathbf{K}_{e}^{\text{b}}(a_e) = E_e I_e \int_{0}^{\ell_e} \left(\mathbf{B}_e^\mathrm{b}(x)\right)^\mathrm{T} \mathbf{B}_e^\mathrm{b}(x) \, \mathrm{d}x$$
with $\mathbf{B}_e^\mathrm{b}(x)$ being the bending strain-displacement matrix. 

Assume now that the cross-section shape is fixed. Let us denote the reference cross-section area as $a_0$ and the corresponding moment of inertia as $I_0$. The scaled cross-section area will be denoted as $a_e$. In what follows, we will consider the cases of scaling the cross-section in the $y$-direction, $z$-direction, or both directions concurrently.

Consider first $\theta$-scaling the cross-section in the direction $y$, orthogonal to the bending load. Then, we have the coordinate transformation $(y,z)\rightarrow (\theta y, z)$, so that the scaled area is $a_e = \theta a_0$. Thus, the second moment of area reads as
$$
I_e = \int_{A_{e}} z^2 \,\mathrm{d}A = \int_{A_{0,e}} z^2 \theta \,\mathrm{d}A = \theta \int_{A_{0,e}} z^2 \,\mathrm{d}A = \theta I_0 = \frac{a_e}{a_0} I_0,
$$
which shows that, in this case, $I_e (a_e)$ is a linear function in $a_e$.

Second, consider $\theta$-scaling of the cross-section in the direction of the bending load. Then, we have $(y,z)\rightarrow (y, \theta z)$ and $a_e = \theta a_0$ again, and
$$
I_e = \int_{A_e} z^2  \,\mathrm{d}A = \int_{A_{0,e}} \left(\theta z\right)^2  \theta \,\mathrm{d}A = \theta^3 \int_{A_{0,e}} z^2 \,\mathrm{d}A = \theta^3 I_0 = \left( \frac{a_e}{a_0}\right)^3 I_0,
$$
so that $I_e (a_e)$ is a cubic polynomial in $a_e$.

Third, let us scale the cross-section by $\theta$ in both directions. Then, we have $(y,z)\rightarrow (\theta y, \theta z)$ and $a_e = \theta^2 a_0$, i.e.,
$$
I_e = \int_{A_e} z^2  \,\mathrm{d}A= \int_{A_{0,e}} \left(\theta z\right)^2  \theta^2 \,\mathrm{d}A = \theta^4 \int_{A_{0,e}} z^2 \,\mathrm{d}A = \theta^4 I_0 = \left(\frac{a_e}{a_0}\right)^2 I_0,
$$
which shows that, in this case, $I_e(a_e)$ is a degree-two polynomial.

\subsection{Mass Matrix}

The stiffness-consistent mass matrix distributes the mass according to the element's shape functions $\mathbf{N}_e(x)$ as
$$\mathbf{M}_e (a_e) = a_e \rho_e \mathbf{L}_e^\mathrm{T} \mathbf{T}_e^\mathrm{T} \int_{0}^{\ell_e} \mathbf{N}_e(x)^\mathrm{T}\mathbf{N}_e(x) \, \mathrm{d}x \mathbf{T}_e \mathbf{L}_e,$$
where $\rho_e \in \mathbb{R}_{>0}$ is the constant element density. This gives us the term $a_e\mathbf{M}_{j,e}^{\langle 1\rangle}$ in \eqref{eq:mass}.
}

\newpage
\bibliography{liter.bib}
\bibliographystyle{siamplain}

\end{document}